\documentclass[a4paper,dvipsnames,oneside]{amsart}

\makeatletter
\@namedef{subjclassname@2020}{%
  \textup{2020} Mathematics Subject Classification}
\makeatother

\usepackage[left=3cm,right=3cm,top=25mm,bottom=25mm]{geometry}

\usepackage{amsmath,amssymb,amsfonts,amsthm,mathtools,comment,enumitem,thmtools,setspace,graphicx}

\usepackage[textwidth=2.5cm,colorinlistoftodos,textsize=footnotesize]{todonotes}

\usepackage{tikz}
\usetikzlibrary{patterns}
\usepackage[color,matrix,arrow]{xy}
\definecolor{bluUniud}{RGB}{119,154,171}

\definecolor{mycolor}{HTML}{F7F8E0}
\usepackage[colorlinks,citecolor=mycolor!70!black,urlcolor=black]{hyperref} 

\usepackage[capitalise,noabbrev]{cleveref}

\newtheorem{theorem}{Theorem}[section]

\newtheorem{corollary}[theorem]{Corollary}
\newtheorem{lemma}[theorem]{Lemma}

\newtheorem{proposition}[theorem]{Proposition}

\theoremstyle{definition}
\newtheorem{definition}[theorem]{Definition}

\theoremstyle{remark}  

\newtheorem{remark}[theorem]{Remark}

\newtheorem{question}[theorem]{Question}

 \newtheoremstyle{fatto}
  {}
  {}
  {\itshape}
  {5mm}
  {\bfseries} 
  {.}
  {2mm}
  {}
\theoremstyle{fatto}
\newtheorem{claim}{Claim}[theorem]

\declaretheoremstyle[
  spaceabove=-3pt,%
  spacebelow=6pt,%
  headfont=\normalfont\itshape,%
  postheadspace=2mm,%
  qed=\footnotesize{\(\square\)}]{mystyle}

\setlist{itemsep=.02em,labelsep=.4em,topsep=.5em}   

\newcommand{\rca}{\mathsf{RCA}_0}
\newcommand{\wkl}{\mathsf{WKL}_0}
\newcommand{\aca}{\mathsf{ACA}_0}
\newcommand{\pra}{\mathsf{PRA}}

\newcommand{\Pra}{{\mathsf{PRA}^2}}
\newcommand{\qfac}{\mathsf{QF}\text{-}\mathsf{AC}}
\newcommand{\bqfac}{\mathsf{BQF}\text{-}\mathsf{AC}}

\newcommand{\pa}{\mathsf{PA}}
\newcommand{\pam}{\mathsf{PA}^-}

\newcommand{\iso}{\mathrm{I}\Sigma^0_1}
\newcommand{\ido}{\mathrm{I}\Delta^0_1}

\newcommand{\is}[2]{\mathrm{I}\Sigma^{#1}_{#2}}
\newcommand{\bs}[2]{\mathrm{B}\Sigma^{#1}_{#2}}
\newcommand{\id}[2]{\mathrm{I}\Delta^{#1}_{#2}}
\newcommand{\lnp}[2]{\mathrm{L}\Pi^{#1}_{#2}}
\newcommand{\dcao}{\Delta^0_0\text{-}\mathrm{CA}}
\newcommand{\dca}{\Delta^0_1\text{-}\mathrm{CA}}
\newcommand{\dzca}{\Delta^0_0\text{-}\mathrm{CA}}
\newcommand{\qfi}{\mathsf{QF}\text{-}\mathsf{I}}

\newcommand{\E}{\,E\,}

\newcommand{\R}{\,R\,}

\newcommand{\la}{\langle}
\newcommand{\ra}{\rangle}

\newcommand{\imp}{\rightarrow}
\newcommand{\Imp}{\Rightarrow}
\newcommand{\biimp}{\leftrightarrow}
\newcommand{\Biimp}{\Leftrightarrow}
\newcommand{\Nb}{\mathbb{N}}
\newcommand{\nat}{\mathbb{N}}

\newcommand{\smf}{{}^\smallfrown}

\newcommand{\mc}{\mathcal}

\newcommand{\kl}{\mathsf{KL}}
\newcommand{\rt}[2]{\mathsf{RT}^{#1}_{#2}}
\newcommand{\coh}{\mathsf{COH}}

\usepackage{amssymb}
\usepackage{xcolor}
\usepackage{framed}

\def\msf{\mathsf}

\def\mcal{\mathcal}
\def\h{\hat}

\newcommand{\etalchar}[1]{$^{#1}$}

\title{Primitive recursive reverse mathematics}
\author{Nikolay Bazhenov}
\address{Sobolev Institute of Mathematics, pr. Akad. Koptyuga 4, Novosibirsk, 630090 Russia}
\email{bazhenov@mail.math.nsc.ru}

\author{Marta Fiori-Carones}
\address{Sobolev Institute of Mathematics, pr. Akad. Koptyuga 4, Novosibirsk, 630090 Russia}
\email{marta.fioricarones@outlook.it}

\author{Lu Liu}
\address{School of Mathematics and Statistics, HNP-LAMA, Central South University ChangSha, China, 410083}
\email{g.jiayi.liu@gmail.com}

\author{Alexander Melnikov}
\address{Victoria University of Wellington, Wellington, New Zealand} 
\email{alexander.g.melnikov@gmail.com}

\thanks{We thank Jeremy Avigard and Stephen Simpson for several useful suggestions at early stages of the project. Special thanks to Rodney Downey who actively participated in these early discussions.
}
\date{\today}

\definecolor{martablu}{HTML}{0C7988}

\def\pr{PR}
\begin{document}

\begin{abstract}
We use a second-order analogy $\Pra$ of $\pra$ to
investigate  the proof-theoretic strength of theorems in countable algebra, analysis,
and infinite combinatorics.
We compare our results with similar results in the fast-developing  field of primitive recursive (`punctual')
algebra and analysis, and with results from `online' combinatorics.
We argue that $\Pra$ is sufficiently robust to serve as an alternative base system below $\rca$
to study the proof-theoretic content of theorems in ordinary mathematics.
(The most popular alternative is perhaps $\rca^*$.)
We discover that many theorems that are known to be true in $\rca$ either hold in $\Pra$ or are equivalent to
$\rca$ or its weaker (but natural) analogy $2^{\mathbb{N}}$-$\rca$ over $\Pra$. However, we also discover that some standard mathematical and combinatorial facts 
are incomparable with these natural subsystems.


\end{abstract}



\maketitle

\tableofcontents

\section{Introduction}

Reverse mathematics is a relatively new program in mathematical logic. Its basic goal is to assess the relative logical strengths of theorems from the  `ordinary' (non set theoretic) mathematics. In reverse mathematics, one tries to find the minimal natural axiom system $\Gamma$ that is capable of proving a given theorem $\Delta$.
This is usually done by proving that, over a certain rather weak base system, $\Gamma$ is \emph{equivalent} to $\Delta$. In other words, one of the crucial steps in such investigations is proving the axioms $\Gamma$ from  the given theorem $\Delta$---thus, the name `reverse mathematics'.

\

 Friedman, Simpson, and Smith 
 presented  these ideas in a systematic way in their seminal work
\cite{FriedmanSimpsonSmith}. Their paper contains a large number of examples of classical theorems from
countable algebra analysed in several subsystems of the second-order arithmetic. Following the earlier ideas of Friedman~\cite{Friedman1976a,Friedman1976b},  Friedman, Simpson, and Smith chose $\rca$ as their most basic axiomatic system. Here $\rca$ stands for the `recursive comprehension axiom (scheme)'; 
informally speaking, this axiomatic system postulates the existence of `recursive'  (computable) subsets of $\mathbb{N}$.

It is perhaps not a coincidence that around the same time, the subject of \emph{recursive (computable, effective) algebra} was getting increasingly popular
in both the US and Australia and,  independently, in the Soviet Union.
Effective algebra investigates computability-theoretic properties of countable algebraic structures. Such investigations began in the 1960s with the works of Mal'cev~\cite{Mal-61,Mal-62} and Rabin~\cite{Rabin-60}. By the mid-1980s the subject had accumulated a large number of non-trivial results, perhaps most notably in countable field theory, countable Boolean algebras, and commutative group theory; we cite \cite{Handbook-RecMath-01,Handbook-RecMath-02,Ash-Knight-book,Ershov-Goncharov-2000}.  Around the same time, the subject of \emph{computable analysis}
was becoming increasingly popular too; we cite \cite{Pour-El-Richards,Weihrauch-book}. The main objects of investigation in computable analysis are recursively (computably) presented separable spaces and recursive (computable) functions between such spaces.

A large number of results in reverse mathematics, especially in the early stages of its development, were based on similar results in
effective algebra and computable analysis.
 Many results in Friedman, Simpson, and Smith~\cite{FriedmanSimpsonSmith} are essentially `recycled'
effective algebraic theorems.
 For example, it is well-known that every computable field can be computably embedded into
its computable algebraic closure; this is an old result due to Rabin~\cite{Rabin-60}. It is therefore perhaps not surprising that the result also holds in $\rca$. However, this of course requires some extra work since $\rca$ additionally restricts the  axiom of induction; we cite~\cite{sosoa} for the details.
 For more results based on effective algebra, we cite \cite{reverse-book-2001,sosoa,SolomonThesis,ShoreBA}.
For various results inspired by computable analysis, see, e.g., \cite{reverseanal,reverseanal1, reverseanal2}.
More recently, it has become rather common to
\emph{combine} reverse mathematics with effective algebra. Each of the two subjects suggests a certain measure of complexity of an algebraic result, and while  these measures can be somewhat related technically, usually there is no immediate implication between the two. For a few  relatively recent examples, we cite \cite{Melnikov-Greenberg,Conidis-19}. Also, there are rather explicit connections between reverse mathematics and computable analysis; e.g.,~\cite{Weihrauch-book,Brattka-05,Gherardi-Markone-09,Brattka-Survey-21}. More generally, \emph{computable mathematics} and \emph{reverse mathematics} (especially in $\rca$ and not far beyond) have become so interconnected that no firm line can be drawn between them.

\vspace{2mm}

In the recent years and beginning with \cite{Kal17}, there has been much work in primitive recursive (`punctual') algebra.
Also, there have been several recent results in primitive recursive analysis~\cite{SelSel-21,spaces-preprint}.
The main  goal of such investigations is the elimination of unbounded search from results in computable mathematics.
Such investigations often lead to unexpected results. Indeed, the technical depth of some of these results
is almost equally unexpected. For example, it is easy to see that the back-and-forth
proofs of computable categoricity for the
 dense linear order $(\mathbb{Q},<)$  and for the random graph contain exactly one instance of unbounded search at every stage.
Using degree-theoretic techniques, Melnikov and Ng~\cite{MelNg-19} discovered that the `fully primitive recursive degrees' of these structures
are not isomorphic as partial orders, and this reflects that these delays have different nature.
This difference is rather subtle and its nature is not yet fully understood.
We cite surveys \cite{Mel17-survey,SurveyOnline2019,punc2} for many more results in primitive recursive mathematics and
for a detailed exposition of the theory.
The theory has accumulated many theorems about primitive recursive algebraic and separable structures.
 Perhaps more importantly, the theory has developed enough tools
that allow to systematically investigate primitive recursive mathematical structures and processes upon such structures.

Perhaps somewhat unexpectedly, such investigations are rather closely related to another seemingly distant branch of computable mathematics,
 namely `online' combinatorics.
Beginning in the 1980's
there has been quite a lot of work on online infinite combinatorics, particularly
by Kierstead, Trotter, Remmel and others (\cite{Kier81,Kier98a,KPT94,LST89,Remmel86}). Some results
were quite surprising. For example, Dilworth's theorem says that a
partial ordering of width $k$ can be decomposed into
$k$ chains. Szemeredi and others showed that there is a
computable partial ordering of width $k$ that cannot be decomposed into
$k$
computable chains. But in 1981, Kierstead proved that
there is an online algorithm that will decompose
any online presentation of a computable partial ordering
into $\frac{5^k-1}{4}$ many (computable) chains.
Investigations here are still ongoing; e.g.,~\cite{FCM,FCSS}.
As was noted in \cite{KaMeMo}, there is a technical connection between results of this sort
and the primitive recursive `punctual' framework by means of \emph{subrecursive relativisation} (to be clarified). For instance, it has been  demonstrated in   \cite{KaMeMo} that
 there is a tight connection between definability  and relativized primitive recursion
in the context of countable algebraic structures. Based on these results and observations,
it has been proposed in \cite{punc2} that punctual algebra and online combinatorics can be studied simultaneously, and indeed
that there should exist a unified approach to the reverse mathematics of these results. However, it was not clear what would be the `right'
base axiomatic system for such investigations. Even though $\rca$ is the standard base system for reverse mathematics, it fails to capture the subtle effects related to forbidding unbounded search.

\

There are several  weaker base systems below $\rca$ that could potentially capture the subrecursive content of mathematics, we briefly go over some of them. For example, \cite{SimSm,Hatzikiriakou1989} proposed $\rca^*$,  which is $\rca$ with a weakened induction scheme. While certainly rather interesting and useful in the study of the role of induction, it seems
its power and convenience (in, e.g., countable algebra) is extremely limited. Perhaps, one of the possible reasons is that  $\rca^*$ only proves bounded primitive recursion which poses a significant limitation on the `constructive' arguments that can be imitated in  $\rca^*$. In fact, $\rca^*$ is  $\Pi^0_2$-conservative over \emph{elementary recursion arithmetic}, see \cite[Theorem 4.4]{avigad_2005} and \cite[Corollary 4.9]{SimSm}). However, with some effort several results in ordinary mathematics can be carried over  $\rca^*$, which seems very surprising (thus, interesting) since bounded primitive recursion appears to be a very weak tool in algebra. Research into $\rca^*$ is ongoing; we cite~\cite{KolodziejczykYokoyama,kky:ramsey-rca0star,fkwy,yokoyama2013,HatSimp2017}.

The other well-known `subrecursive' system is $\pra$ with one axiom for each primitive recursive scheme.  However, it is a first-order system and can really  handle only finite sets that can be identified with their codes. A  truly remarkable theorem is the $\Pi^0_2$-conservativity of $\wkl$ over $\pra$ which in particular implies that these theories are equiconsistent; see \cite[Section IX.3]{sosoa} where one can also find more references. However, while the system undoubtedly plays a rather important role in proof theory, it
cannot serve as a base  for the reverse mathematics of, e.g., countable algebra or infinite combinatorics. The obvious obstacle is, of course, that the system is not second-order.

We also mention the various sub-recursive systems specifically designed to study complexity-theoretic results; see books \cite{buss1986,buss1998,Cook-Nguyen}.
Similarly to $\pra$, such subsystems appear to be too restricted to be used as a base theory to study infinite mathematics.

To keep the intro reasonably compact, we will no longer proceed with the  discussion of various possible systems below $\rca$ and refer the reader to \cite{FFF}. Instead, we will concentrate on the main subject of the paper, namely the second-order analogy $\Pra$ of $\pra$.

\medskip

The system of our choice is $\Pra$. It is a function-based system (as opposed to the set-based systems $\rca$, $\wkl$, etc.)
that postulates that functions are closed under primitive recursive schemata.
 Informally, this corresponds to primitive recursive relativisation.
The system $\Pra$ is, of course, not new. For instance,
Avigad~\cite{avigad_2005} presented a nonstandard higher-type extension of $\pra$, which is still $\Pi_2$-conservative over $\pra$, and some weaker systems, providing some examples of statements of elementary analysis which can be proved in such systems.  Various proof-theoretic properties of higher-type analogies of $\pra$, including $\Pra$,  are thoroughly studied in the books \cite{avigad_2005,kohlenbach2008}. We also remark that Harvey Friedman in \cite{Friedman1976a,Friedman1976b} originally introduced $\rca$ in a functional language, not in the set-based one adopted in \cite{sosoa}.  Friedman defined $\rca$ as $\dca$ plus essentially $\ido$ and closure under primitive recursive functions. This subsystem implies $\iso$, so, in the end, it is another presentation of the usual basic theory. Nonetheless, this may reveal that to Friedman's eyes, primitive recursion carries a foundational import, which is perhaps hidden in the later formulation of Simpson~\cite{sosoa} who uses $\dca$ and $\iso$ to derive totality of primitive recursive functions\footnote{See also \url{https://cs.nyu.edu/pipermail/fom/2002-April/005415.html} for a further discussion.}.

 The axiomatic system $\Pra$ seems to be the most natural second-order system to study primitive recursive proofs and processes in countable algebra, separable spaces, and infinite combinatorics.
Indeed, the second-order part of the minimal $\omega$-model of $\Pra$ is just the collection of all primitive recursive functions.
As we will discuss later, $\Pra$ proves comprehension and induction with bounded quantifiers.
This corresponds to our intuition that primitive recursive processes should correspond to definability with bounded quantifiers. We are not the first to realise that $\Pra$ has a potential in the reverse mathematics of ordinary theorems.
Some 20 years before us Kohlenbach~\cite{kohlenbach2000} tested the system from the perspective of reverse mathematics.
While Kohlenbach's examples are both interesting and instructive, at that time neither countable algebra nor
analysis could really offer enough \emph{primitive recursive} results and techniques that could be partially re-used to truly test
the system in ordinary mathematics.

The main purpose of this paper is to initiate (or revive) a systematic investigation of the primitive recursive content of ordinary mathematics using $\Pra$.
 In this paper, we do only a few initial steps that we believe are sufficient to lay the foundations of this theory.

\

We now discuss the results that are summarised in Fig.~\ref{fig:fig}.  (Not all results and examples are included into the diagram.)

\begin{figure}[h]\label{fig:fig}
\begin{center}
 \begin{tikzpicture}
 \node (pra) at (0,0) {$\Pra$};
 \node (ido) at (0,1) {$\ido$};
 \node (dca) at (1,2) {$\dca$};
 \node (iso) at (-1,2) {$\is 0 1$};
 \node (baire) at (-4,3) {$\mathsf{BaireCategory}$};
  \node (concom) at (5.5,4) {$\mathsf{Connected Components Exist}$};
 \node (isodca) at (0,4) {$2^{\mathbb{N}}$-$\rca= \is 0 1 \land \dca$};
 \node (qfac) at (0,6) {$ \rca$};
 \node (cat) at (-2,6) {$\mathsf{Categoricity}$};
 \node (rt1Inf) at (-5,5.5) {$\rt 1 {<\infty}$};
 \node (rt22) at (-5,6.5) {$\rt 2 2$};
  \node (rt32) at (-5,7.5) {$\rt 3 2$};
 \node (wkl) at (4,5.5) {$\wkl$};
  \node (hbt) at (6.5,5.5) {$\mathsf{HeineBorelTheorem}$};
 \node (unifCon) at (5,7) {$\mathsf{Uniform Continuity}$};
  \node (rcaWkl) at (2,7) {$\rca+\wkl$};
 \node (2aca) at (7,8) {$2^{\Nb}\text{-}\aca$};
 \node (aca) at (0,9) {$\aca$};

\draw[->] (ido) -- (pra);
 \draw[->] (iso) -- (ido);
 \draw[->] (isodca) -- (iso);
 \draw[->] (isodca) -- (dca);
 \draw[->] (dca) -- (ido);
 \draw[->] (qfac) -- (isodca); 
 \draw[<->] (hbt) -- (wkl);
  \draw[<->] (qfac) -- (cat);
    \draw[<->] (concom) -- (isodca);
    \draw[<->] (qfac) -- (cat);
\draw[->] (qfac) -- (baire);
 \draw[->] (baire) -- (pra);
 \draw[dashed,->] (baire) -- (iso);
 \draw[->] (wkl) -- (dca);
 \draw[->] (rcaWkl) -- (wkl);
 \draw[->] (rcaWkl) -- (qfac);
 \draw[<->] (unifCon) -- (rcaWkl);
 \draw[->] (2aca)  .. controls (7,7) and (6.5,6) .. (wkl);
 \draw[->] (aca) -- (qfac);
 \draw[->] (aca) -- (2aca);
 \draw[->] (aca) -- (rcaWkl);
  \draw[->] (aca) -- (rt32);
\draw[->] (rt32) -- (rt22);
\draw[->] (rt22) -- (rt1Inf);
 \draw[dashed,->] (rt22)  .. controls (-4.5,5.5) .. (baire);
 \draw[dashed,->] (rt1Inf) -- (baire);
  \draw[dashed,->] (rt22) -- (iso);
 \draw[dashed,->] (rt1Inf) -- (iso);
 \draw[->] (rt1Inf) .. controls (-7,2) and (-3,1) .. (pra);
  \end{tikzpicture}
  \end{center}
  \caption{\small Summary of the results proved in the paper.
  Lines represent strict implications. Dashed arrows represent either not known but possible implications, or implications that we believe are true but we have not formally verified in the paper.  
  $\mathsf{Categoricity}$ stands for the three principles studied in \cref{Subsec:cat}. $\rca$ is identified with it function-based analogy $\qfac$.   } 
\end{figure}

Section~\ref{sect:prelim} is a preliminaries section. It contains some basic facts about $\Pra$ that will be necessary
in the later sections. Among other things, we verify that the formalisation of finite sets is robust in $\Pra$. More specifically, we check that the two most natural approaches (the first-order and the second-order ones) are equivalent, and that the cardinality of a finite set makes sense. We also discuss the relationship between the function-based $\Pra$ and the set-based $\rca$, and also induction and comprehension axiom schemata. For instance, we explain why, over $\Pra$, the set-based recursive comprehension $\dca$
is weaker than the full function-based version  $\qfac$ of recursive comprehension. Since we will observe that $\qfac$ implies $\iso$ over $\Pra$, one can say that up to notation, $\qfac$ is $\rca$; it is indeed very similar to what Friedman initially suggested.

Section~\ref{sec:examples} contains two alternative approaches to countable algebraic structures in $\Pra$, and it also contains a fair amount of examples. Some of these examples follow (often rather non-elementary) proofs from the literature and are perhaps somewhat routine but instructive. However, we believe that some other results presented in the section should be viewed as foundational. For instance, we verify that $\Pra$ proves that every countable field can be embedded into its algebraic closure. It is well-known that the result holds in $\rca$ \cite{FriedmanSimpsonSmith,sosoa}, but the standard proof relies on too much induction. Some care must be taken to prove it over $\Pra$.
In this section we also look at countable categoricity of $(\mathbb{Q},<)$ and the random graph. In contrast with the aforementioned result from primitive recursive algebra, $\Pra$ fails to detect the subtle difference between these two results.
More formally,  each of these results is equivalent to $\rca$ over $\Pra$. This situation is only expected: the reverse mathematics over $\rca$ also typically does not distinguish between, e.g., $\mathbf{0}'$-effective and $\mathbf{0}''$-effective arguments in computable algebra~\cite{Melnikov-Greenberg}.

In Section~\ref{sect:Baire_and_Ramsey}, we study Ramsey-type theorems and Baire category theorem.  It is well-known (and easy to see) that $\rca$ proves Baire category theorem.
It also seems that the proof makes an essential use of a truly unbounded search (quantification). However, we will prove that
Baire category theorem
is actually \emph{not} equivalent to $\rca$ over $\Pra$, but lies strictly in-between $\Pra$ and $\rca$.
In fact, with just a bit more effort we show that Baire category theorem neither implies nor is implied by  $2^{\mathbb{N}}$-$\rca$ over $\Pra$
 (recall $2^{\mathbb{N}}$-$\rca$ is the weaker set-version of the full function-based version of $\rca$).
We also examine Ramsey theorem and show that, over $\Pra$,  $\mathsf{RT}_k^n$ is incomparable with
$2^{\mathbb{N}}$-$\rca$.
We believe that these results have no direct analogy in the literature.

Section~\ref{sect:transforming} studies the following, rather general, phenomenon: for many combinatorial problems, any
computable instance can be transformed into a primitive recursive instance having either the same or (in some sense) equivalent solution.
We give a long list of examples of such problems and discuss the consequences. We note that (almost evidently) $\wkl$ is also in this list.  The results should be compared to the rather long list of examples from primitive recursive algebra which assert that, in many broad classes of countable algebraic structures, every computable structure has a primitive recursive (or even `punctual') presentation; we cite \cite{Grigorieff-90,Cenzer-Remmel-91,Cenzer-Remmel-98,Kal17}.  Results of this sort give a rather strong evidence that $\Pra$
indeed could be used as an alternative base for reverse mathematics to study combinatorial and algebraic theorems.

In \cref{Sec:WKL} we look at $\wkl$ over $\Pra$.  
 It is well-known that, over $\rca$, uniform continuity  of a continuous function on $[0,1]$ is equivalent to $\wkl$; see, e.g., Theorem~IV.2.3 in~\cite{sosoa}.
We show that, over $\Pra$, the uniform continuity of a continuous function on $[0,1]$
is \emph{strictly stronger} than $\wkl$. 
Specifically, we prove that  over $\Pra$, the uniform continuity  of a continuous function on $[0,1]$ is equivalent to $\rca+ \wkl$.

The study of $\Pra$ has many potential open questions, some perhaps routine but some likely  challenging.
We state several concrete open questions throughout the paper. We finish the paper with a brief Section~\ref{sec:que} where we pose several further open problems.


\section{Preliminaries}\label{sect:prelim}


\subsection{The finitist's first-order system PRA}

\begin{definition}\label{Def_ind}
The induction axiom for $\Sigma_n$-formulae, $\is {} n$, is the following schema of axioms
\[\forall \bar{c}\,( (\varphi(0,\bar{c}) \land \forall n \,(\varphi(n,\bar{c}) \imp \varphi(n+1,\bar{c}))) \imp \forall n \,\varphi(n,\bar{c}) ),\]
where $\varphi$ is a $\Sigma_n$-formula.

The induction axiom for $\Delta_n$-formulae, $\id {} n$, is the following schema of axioms
\[\forall \bar{c}\,( \forall n\, (\varphi(n,\bar{c}) \biimp \psi(n,\bar{c})) \imp ((\varphi(0,\bar{c}) \land \forall n \,(\varphi(n,\bar{c}) \imp \varphi(n+1,\bar{c}))) \imp \forall n \,\varphi(n,\bar{c}))),\]
where $\varphi$ is a $\Sigma_n$-formula and $\psi$ is a $\Pi_n$-formula.

The least number principle for $\Pi_n$-formulae, $\lnp {} n$, is the following schema of axioms
\[\forall \bar{c} \, (\exists n \,\varphi(n,\bar{c}) \imp \exists n\, (\varphi(n,\bar{c}) \land \forall m<n \, \neg\varphi(m,\bar{c})) ), \]
where $\varphi$ is a $\Pi_n$-formula.
\end{definition}

Any of $\Sigma_0, \Pi_0, \Delta_0$ are generally defined to be formulae with only bounded quantifiers. Also, $\Sigma_n$- and $\Delta_n$-formulae may contain  bounded quantifiers; these do not contribute to the complexity.
Recall that over $\pam$, for each $n \in \nat$, $\is {} {n+1} \Imp \id {} {n+1} \Imp \is {}  n \Biimp \lnp {} n$, and $\id {} 0 \Biimp \is {} 0$; see \cite[Theorems I.2.4, I.2.5, IV.1.29]{hajekPudlak}, plus the fact that $\pam + \exp \vdash \bs {} n \Biimp \id {} n$, for all $n>0$ by \cite{slaman2004}.
The definition below is standard (e.g., \cite{sosoa}).

\begin{definition}
 Let $\mc{L}_{\pra}$ be the first-order language with non-logical symbols $\{0, s, < \}$ and a symbol for any primitive recursive function.
   The axioms of $\pra$ are the following:
 \begin{enumerate}
   \item $\forall x (0 \ne s(x))$; $\forall x,y (s(x) = s(y) \imp x=y)$;
   \item defining equations of any primitive recursive function;
   \item $\qfi$, i.e., induction for any quantifier-free formula $\theta$:
   \[
	   	(\theta(0) \land \forall n \, (\theta(n) \imp \theta(n+1))) \imp \forall n \,\theta(n).
	\]
 \end{enumerate}
\end{definition}

The next lemma is proved in \cite[Theorem 0.35]{hajekPudlak} and in \cite[Lemma IX.3.7]{sosoa} for $\Delta^0_0$-formulae. 

\begin{lemma}[\cite{sosoa}] \label{SimpsonIX37}
For each $\Delta_0$-formula $\theta$ there exists   a primitive recursive function $f$ such that $\pra \vdash (f(n)= 1 \biimp \theta(n)) \land (f(n)= 0 \biimp \neg\theta(n))$.
\end{lemma}

\begin{proposition}\label{prop:indthing}
$\pra \vdash \id {} 0$.
\end{proposition}
\begin{proof}
Let $\theta$ be a $\Delta_0$-formula and assume that $\pra \vdash \theta(0) \land \forall n (\theta(n) \imp \theta(n+1))$. By \cref{SimpsonIX37}, let $f$ be such that  $\pra \vdash (f(n)= 1 \biimp \theta(n)) \land (f(n)= 0 \biimp \neg\theta(n))$. Then, $\pra \vdash f(0)=1 \land \forall n \,(f(n)=1 \imp f(n+1)=1)$, which by $\qfi$ implies $\pra \vdash \forall n(f(n)=1)$, so that $\pra \vdash \forall n (\theta(n))$.
\end{proof}

Note that $\pra$ is an extension by definition of the theory $\pam + \id {} 0$ plus totality of any primitive recursive function. That is any model of $\pra$ can be seen as a model of $\pam + \id {} 0$ in which any primitive recursive function is total. The following propositions are  immediate consequences of this.

\begin{proposition} \label{Prop_PraNonId01}
$\pra \nvdash \id {} 1$.
\end{proposition}
\begin{proof}[Proof Sketch]
The following proof suggested to us by Ko{\l}odziejczyk 
is similar to the proof of the fact that $\is {} {n-1} \nvdash \bs {} n$, for each $n\ge 1$ (see \cite[Chapter 10]{kaye1991}), and so that  $\is {} {n-1} + \exp \nvdash \id {} n$, since $\pam + \exp \vdash \id {} n \biimp \bs {} n$, for each $n \in \Nb$, by \cite{slaman2004}.

Let $M \vDash \pa$ and $c \in M$ be non-standard.
Consider the structure $K^1(M,c)$ of the elements of $M$ defined by a $\Sigma_1$-formula, that is $a \in K^1(M,c)$ if and only if there exists some $\Sigma_1$-formula $\varphi(x,c)$ such that $M \vDash \varphi(a,c) \land \forall b \,(\varphi(b,c)\rightarrow b=a)$. Then $K^1(M,c) \vDash \pam + \id {} 0$ and $K^1(M,c) \nvDash \bs {} 1$ by \cite[Theorems 10.3, 10.4]{kaye1991}. We argue that $K^1(M,c)$ proves totality of any primitive recursive function, so that $K^1(M,c) \vDash \pra$ once expanded to $\mc{L}_{\pra}$. To this end, let $f$ be a primitive recursive function. Since $f$ is primitive recursive and $M \vDash \pa$, then $f$ is provably total in $M$ and defined in $M$ by some  $\Delta_1$-formula $\psi(z,y)$  (see \cite[Theorem I.1.54 and Lemma I.1.52]{hajekPudlak}). Thus, if $x \in K^1(M,c)$, there exists $y \in M$ such that $M \vDash \psi(x,y)$, that is $y=f(x)$ in $M$. The following $\Sigma_1$-formula defines $y$ in $M$
\[\exists x \, \big(\theta(x,c) \land \psi(x,y)\big) \land \forall b \,\forall d \,\big(\theta(b,c) \land \psi(b,d) \imp d=y\big)\]
where $\theta(x,c)$ is the $\Sigma^0_1$-formula defining $x$.
This shows that $y \in K^1(M,c)$.
\end{proof}

Analogous reasons as above actually prove that $\pra \land \is {} {n-1} \nvdash \bs {} n$, for each $n\ge 1$, and in particular that the structure $K^n(M,c)$ of the elements of $M$ defined by a $\Sigma_n$-formula on some parameter $c \in M$ satisfies $\pra$. Moreover, as expected, the following is true and can be proved similarly.

\begin{proposition} \label{Prop:notIs}
 $\pra \land \id {} n \nvdash \is {} n$, for each $n\ge 1$.
 \end{proposition}
 \begin{proof}

 Let $M \vDash \pa$, $c \in M$ be non-standard and $n \ge 1$.
Consider $I^n(M,c)$, the downwards closure of the elements $\Sigma_n$ in $c$ definable; in other words,
$a \in I^n(M,c)$ if and only if there exists some $b \ge a$ and some $\Sigma_n$-formula $\varphi(x,c)$ such that $M \vDash \varphi(b,c) \land \forall d \,(\varphi(d,c)\rightarrow d=b)$. Then $I^n(M,c) \vDash \pam + \id {} n$ and $I^n(M,c) \nvDash \is {} n$ by \cite[Theorem 10.10]{kaye1991} and \cite{slaman2004}. We argue that $I^n(M,c)$ proves totality of any primitive recursive function, so that $I^n(M,c) \vDash \pra$ once expanded to $\mc{L}_{\pra}$. To this end, let $f$ be a primitive recursive function and $x \in I^n(M,c)$.

Define a  primitive recursive function $g \colon \Nb \to \Nb$ such that $g(z)=\sum_{i=0}^z f(i)$. Let $b\ge x$ be such that $b \in K^n(M,c)$. Since $K^n(M,c)\vDash \pra$, then $g(b) \in K^n(M,c)$. Moreover, $g$ is provably non-decreasing in $M$, and so $f(x) \in I^n(M,c)$, since $f(x)\le g(x) \le g(b)$.
  \end{proof}

The previous proposition implies that $\pra$ does not prove $\is {} n$, for any $n >0$.
Recall that $\pam + \id {} 0$ does not prove totality of $\exp$ (see \cite[Theorem 4.3]{parikh1971}).

\begin{proposition}
There exists a model of $\pam + \id {} 0$ 
which cannot be expanded to a model of $\pra$.
\end{proposition}
\begin{proof}
 Let $M \vDash \pam + \id {} 0$ be such that there exists a primitive recursive function which is not  total in $M$. Such a model exists since $\pam + \id {} 0$ does not prove totality of all primitive recursive functions.  Then, since $\pra$ proves totality of any primitive recursive function, $M \nvDash \pra$.
\end{proof}

The previous proposition contrasts with the fact that any model of $\is {} 1$ can be expanded to a model of $\pra$ (see \cite[Lemma IX.3.5]{sosoa}). It is also well-known that $\wkl$ is $\Pi^0_2$-conservative over $\pra$; see \cite{sosoa}.

The obvious issue with $\pra$ is that it is first-order, so it is not suited for reverse mathematics in algebra and analysis in the usual sense.


\subsection{The second-order system $\Pra$}
We follow Kohlenbach \cite{kohlenbach2000} and consider the second-order analogy of $\pra$ which, according to the notation in \cite{kohlenbach2000}, will be denoted by $\Pra$.  Recall that in the definition of $\pra$ we postulated the existence of primitive recursive functions, each function was given by a separate axiom. To get the second-order analogy $\Pra$ of $\pra$, we need to postulate the existence of primitive recursive \emph{functionals} (to be clarified). Equivalently, we need to postulate that functions are `closed under primitive recursion'.

\begin{definition}[$\Pra$]
Let $\mc{L}_{\Pra}$ be the two-sorted language with first and second order (function) variables,
 plus all the non-logical symbols of $\mc{L}_\pra$.

The axioms of $\Pra$ are the axioms of $\pra$ extended with defining equations for all primitive recursive functionals of Type 2 (i.e., functions of function argument). We also additionally allow that, in the quantifier-free induction, the formulae can have function-variables (parameters); equivalently, we can take the universal closure of each such axiom.
\end{definition}

We clarify what we mean by a primitive recursive functional. To define primitive recursive functionals on finitely many inputs $f_1, \ldots, f_k$ (which are themselves functions), adjoin
$f_1, \ldots, f_k$ to the list of basic primitive recursive functions and close them under primitive recursion, composition and  bounded minimisation (the latter is, of course, a mere convenience).
Each such individual definition --- that we call a scheme  primitive recursive  relative to $f_1, \ldots, f_k$ --- will correspond to a functional on arguments $f_1, \ldots, f_k$. On input $f_1, \ldots, f_k$ it will output a function defined by the scheme. If we additionally allow $k$ to be arbitrary, we get a recursive list of all possible primitive recursive schemata, each defining a functional. If $P(f_1, \ldots, f_k)$ is one such scheme, we axiomatically postulate that for every $f_1,\ldots, f_k$ there is a $g$ such that $g = P(f_1, \ldots, f_k)$.
Hence, a model $(M, \mc{X})$ is a model of $\Pra$ if $M\vDash \pra$ and $\mc{X}$  is closed under composition and primitive recursion.
This is similar to $\pra$ where we postulate the existence of all primitive recursive functions, but we have no direct access to them since the language is first-order.
Similarly, in $\Pra$, we cannot quantify over functionals, but we can quantify over functions. We also note that
a primitive recursive function can be viewed as a primitive recursive functional: formally, set $k=0$ in $g = P(f_1, \ldots, f_k)$.

\begin{definition}
We say that a class $\mathcal{K}$ of (total) functions is \emph{closed under primitive recursion}, or PR-\emph{closed},  if for every $n$-ary
primitive recursive functional $\Psi$ and any $f_1, \ldots, f_n\in\mathcal{K}$,
\[
	\Psi^{f_1, \ldots, f_n} \in \mathcal{K}.
\]
\end{definition}

\noindent We could instead have used iterated join:
\[
	(f \oplus g) (k) = \begin{cases} f(i) \, \mbox{ if } k = 2i, \\ g(i) \,  \mbox{ if } k = 2i+1, \end{cases}
\]
(which itself is a primitive recursive functional) to restrict ourselves to primitive recursive functionals of one argument, but then we also have to require that $K$ is closed under $\oplus$.

Given a collection $S$ of functions, we can define their primitive recursive closure PR$(S)$
to be the smallest PR-closed class that contains $S$. In particular, a class $\mathcal{K}$
is PR-closed if, and only if, PR$(\mathcal{K}) = \mathcal{K}$.
Note that the smallest PR-closed class is the class of all primitive recursive functions rather than the empty set.
If $\Psi$ is a primitive recursive functional, then in an $\omega$-model it can be thought of as a Turing machine as well, so, in particular, the use principle applies. Thus, occasionally we call these functionals  \emph{operators}.

\

Notice that formulae of $\Pra$ may contain numbers and functions, while $\pra$ is first-order.
Whenever $M\vDash \pra$, $(M, PRec(M)) \vDash \Pra$, where $PRec(M)$ is the collection of all primitive recursive functions over $M$.
In particular, $\Pra$ has a minimal $\omega$-model $(\omega, PRec(\omega))$. Following the convention, we let $\omega$ denote the standard natural numbers, and $\Nb$ the first order universe, which is possibly non standard.

\subsubsection{Primitive recursive induction and comprehension}\label{sub:delta00}

$\,$

\

\paragraph{\emph{Primitive recursive ($\Delta^0_0$) induction.}} Let $\Sigma^0_n$, $\Pi^0_n$, $\Delta^0_n$ denote $\Sigma_n$, $\Pi_n$, $\Delta_n$-formulae, respectively, where function-parameters are allowed. Hence, $\is 0 n$, $\id 0 n$, $\lnp 0 n$ are defined as in \cref{Def_ind} and denote respectively the induction for $\Sigma^0_n$ and $\Delta^0_n$-formulae with function-parameters and the least number principle for $\Pi^0_n$-formulae, also with parameters, respectively.

The proposition below says that $\Pra$ proves induction over formulae in which all quantifiers are bounded --- this is of course exactly as expected. Note that the formulae can have function-pa\-ra\-meters.

\begin{proposition}
$\Pra \vdash \id 0 0$.
\end{proposition}
\begin{proof}
The plan is to imitate the proof of Proposition~\ref{prop:indthing}, but this time we need to define a functional rather than a function.


\begin{lemma}  \label{Lemma_Delta00EquivQf}
For each $\Delta^0_0$-formula $\theta$, there exists   a primitive recursive functional $\Phi$ such that $\Pra$ proves the following
\begin{enumerate}
   \item $\Phi(n_0, \dots, n_m, f_0, \dots, f_\ell)= 1 \biimp \theta(n_0, \dots, n_m, f_0, \dots, f_\ell)$
   \item $  \Phi(n_0, \dots, n_m, f_0, \dots, f_\ell)= 0 \biimp \neg\theta(n_0, \dots, n_m, f_0, \dots, f_\ell)$.
 \end{enumerate}
\end{lemma}
\begin{proof}
Compare with \cite[Lemma IX.3.7]{sosoa}. It is important that the definition of $\Phi$ is derived from the syntax (i.e., the formula), and thus the existence of the functional is postulated in $\Pra$. In other words, this is a meta-argument. We give the details below.

The lemma is proved by induction on the complexity of the formula $\theta$. Assume for readability that $\theta$ has as unique parameters $n$ and $f$.

If $\theta(n,f)$ is atomic of the form $t_1 = t_2$, then $\Phi(n,f)=1 \biimp (|t_1-t_2| \times |t_1-t_2|)=0$ and $\Phi(n,f)=0 \biimp (|t_1-t_2| \times |t_1-t_2|)>0$.

If $\theta(n,f)$ is atomic of the form $t_1 < t_2$, then $\Phi(n,f)=1 \biimp t_2-t_1 > 0$ and $\Phi(n,f)=0 \biimp t_1-t_2 \ge 0$.

If $\theta(n,f)= \theta_1(n,f) \land \theta_2(n,f)$, let by induction hypothesis that $\Phi_1(n,f)$ and $\Phi_2(n,f)$ are equivalent to $\theta_1(n,f)$ and $\theta_2(n,f)$ respectively. Let $\Phi(n,f)= \Phi_1(n,f) \times \Phi_2(n,f)$.

If $\theta(n,f)= \neg \theta_1(n,f) $, let by induction hypothesis  $\Phi_1(n,f)$ be equivalent to $\theta_1(n,f)$. Let $\Phi(n,f)=1 \biimp \Phi_1(n,f)=0$.

Finally, if $\theta(n,f)= (\forall i < m)\big(\theta_1(i,n,f)\big)$, for some term $m$, let by induction hypothesis $\Phi_1(i,n,f)$ be  equivalent to $\theta_1(i,n,f)$. Let $\Phi(n,f)= \prod_{i<m}\Phi_1(i,n,f)$ which is a primitive recursive functional.
\end{proof}

The rest proceeds as in the proof of~\cref{prop:indthing}, but with $\Phi$ in place of $\theta$.
\end{proof}

\paragraph{\emph{Primitive recursive ($\Delta^0_0$) choice and comprehension.}}

The definition below allows to define functions using bounded quantifiers. It can be viewed as a `primitive recursive' variation of choice, and it serves as the function-analog of $\Delta^0_0$-comprehension.

\begin{definition}[Bounded $\Delta^0_0$ choice] Bounded $\Delta^0_0$ choice ($\bqfac$) states
for any $\Delta^0_0$-for\-mu\-la $\theta$ and any term $b$ not mentioning $m$
\[\forall n (\exists m < b) \theta (n,m) \imp \exists f\, \forall n\, \theta(n,f(n)), \]
where both $b$ and $\theta$ may contain parameters (including perhaps function variables).

\end{definition}

\begin{proposition}
$\Pra$ proves bounded $\Delta^0_0$ choice.
\end{proposition}
\begin{proof} We use Lemma~\ref{Lemma_Delta00EquivQf} and
bounded minimisation to define a function $f$ such that
$f(n) = (\mu m < b ) (\theta(n,m)  \land (\forall z<m)\big(\neg \theta(n,z)\big)$.
This, in particular, involves defining a primitive recursive functional (based on the syntactical complexity of $\theta$) and then referring to the axiom stating that the result of applying this functional to the given function-parameters exists.  It should be clear that this function satisfies the desired property.
\end{proof}

\begin{definition}
$\Delta^0_0$-comprehension axiom, $\dcao$, is the following schema
\[\exists f \,\forall n \, (\varphi(n) \imp  f(n)=1\land  \neg\varphi(n) \imp f(n)=0)\]
 where $\varphi$ is a $\Delta^0_0$-formula (perhaps, with parameters). \end{definition}

\noindent Note that $f$ is a characteristic function, i.e.,
$\forall x [f(x) =0 \vee f(x) =1]$.
 As usual, if there are parameters then we could take the universal closure of the formulae above instead. Note that different choices of parameters will correspond to different functions $f$, so each such axiom essentially postulates the existence of a  \emph{functional}. Thus, the proposition below is highly expected.

\begin{proposition}\label{prop:d00ca} $\Pra$ proves $\Delta^0_0$-comprehension.
\end{proposition}
\begin{proof} This is essentially
Lemma~\ref{Lemma_Delta00EquivQf}.
\end{proof}

\subsection{Recursive comprehension and choice}

\subsubsection{$\Delta^0_1$-comprehension axiom ($\dca$).}
\begin{definition}
The $\Delta^0_1$-comprehension axiom, $\dca$, is the following schema
\[\forall n \,(\varphi(n) \biimp \psi(n)) \imp \exists f \,\forall n \,( \varphi(n)\imp f(n)=1 \land  \neg\varphi(n)\imp f(n)=0), \]
 where $\varphi$ is a $\Sigma^0_1$-formula and $\psi$ is a $\Pi^0_1$-formula. Note that $\forall x [f(x) =0 \vee f(x) =1]$.
\end{definition}
We give an example of a familiar theorem equivalent to $\dca$.
Recall that Post's theorem asserts that a set is computable if and only if both the set and its complement are computably enumerable.

\begin{proposition}
Over $\Pra$, $\dca$ is equivalent to the following statement for each $g,h \colon \Nb \to \Nb$:
\[\forall n\, (\exists y \, g(y)=n \biimp \forall y \, h(y)\ne n) \imp \exists f \, \forall n \,((f(n)=1 \biimp \exists y \, g(y)=n) \land (f(n)=0  \biimp \exists y \, h(y)=n )).\]
\end{proposition}
\begin{proof}
For the forward direction, given $g,h \colon \Nb \to \Nb$ as in the statement, $\dca$ guarantees the existence of $f \colon \Nb \to \{ 0,1\}$ such that $f(n)=1 \biimp \exists y \, g(y)=n$. Hence, $f$ is the desired function.


For the reverse direction, let $\theta(n,y)$ and $\eta(n,y)$ be $\Delta^0_0$-formulae such that  $\forall n\, (\exists y\, \theta(n,y) \biimp \forall y\, \eta(n,y))$. Define the functions $h,g \colon \Nb \to \Nb$ such that

\begin{minipage}{0.45\textwidth}
 \begin{align*}
   g(\la n,0\ra) &= 0\\
   g(\la n,y+1 \ra) &=
\begin{cases}
 n+2 & \text{ if } \theta(n,y) \\
 0 & \text{ otherwise}
\end{cases}
 \end{align*}
\end{minipage}
\begin{minipage}{0.45\textwidth}
\begin{align*}
  h(\la n,0\ra) &= 1\\
h(\la n,y+1 \ra)&=
\begin{cases}
 n+2 & \text{ if } \neg\eta(n,y) \\
 1 & \text{ otherwise}.
\end{cases}
\end{align*}
\end{minipage}
\smallskip

Notice that $\forall m \,(\exists y \, g(y)=m \biimp \forall y \, h(y)\ne m)$. Let $f$ be as in the consequent of the statement. Then we have $\forall m \, (f(m)=1 \biimp (m=0 \lor (m\geq 2 \land \exists y \, \theta(m-2,y)))$, and we can choose $f'(n) = f(n+2)$ for the $\Delta^0_1$-comprehension.
\end{proof}

In order to give a simple example of how induction can play a role in the study of mathematical theorems over $\Pra$, we recall the following statement proved in \cite[Lemma 6.4]{avigad_2005}, which essentially states that $\iso$ is equivalent to the existence of a least upper bound for bounded functions.

\begin{lemma}
Over $\Pra$, $\iso$ is equivalent to the following: for each $f \colon \Nb\to \Nb$, if $\exists z\, \forall n \, (f(n)\le z)$, then $\exists z \forall n \, (f(n)\le f(z))$.
\end{lemma}

\subsubsection{Quantifier-free axiom of choice ($\qfac$)}

\begin{definition}
 The schema of the quantifier-free axiom of choice, $\qfac$, is the following schema
 \[\forall n \, \exists m \, \theta (n,m) \imp \exists f\, \forall n\, \theta(n,f(n)), \]
 where $\theta$ is a quantifier-free formula (perhaps, with function- or number-parameters).
\end{definition}

Notice that by \cref{Lemma_Delta00EquivQf}, $\theta$ in the previous definition may be taken $\Delta^0_0$. This implies that $\qfac$ is equivalent over $\Pra$ to $\forall n \, \exists ! m \, \theta (n,m) \imp \exists f\, \forall n\, \theta(n,f(n))$, because one can consider the $\Delta^0_0$-formula $\theta'(n,m) = \theta (n,m) \land (\forall z<m) \neg \theta(n,z)$.

\begin{proposition}\label{prop:p2}
Over $\Pra$, $\qfac$ implies $\dca$.
\end{proposition}
\begin{proof}
Let $\exists y \,\theta(n,y)$ and $\forall y \, \eta(n,y)$, with $\theta$ and $\eta$ $\Delta^0_0$-formulae, be equivalent over $\Pra$. Then it holds that $\forall n \exists y (\theta(n,y) \lor \neg\eta(n,y))$. Since the disjunction in brackets in $\Delta^0_0$, by $\qfac$, there exists a function $f$ such that $\forall n (\theta(n,f(n)) \lor \neg\eta(n,f(n)))$. By \cref{Lemma_Delta00EquivQf}, let $\Phi(n,f)$  be equivalent to $\theta(n,f(n))$.
Then let $g\colon \nat \to \nat$ be such that $g(n)= \Phi(n,f)$. It is immediate to verify that the function $g$ witnesses the satisfaction of $\dca$.\end{proof}

\begin{proposition}[\cite{kohlenbach2008}, Proposition 3.21]\label{Prop_QfacNotIS01}

Over $\Pra$, $\qfac$ implies $\is 0 1$.
\end{proposition}
\begin{proof}
 Assume that $\exists y \, \theta(0,y) \land \forall n \,\forall y\,\exists z\,( \theta(n,y) \imp \theta(n+1,z)))$ holds, for some $\Delta^0_0$-formula $\theta$. Then, by $\qfac$, it holds that
 $\exists f\, \forall \la n,y\ra \, (\theta(n,y) \imp \theta(n+1,f(n,y)))$
 Define a primitive recursive functional $\Phi$ such that
 \[\begin{aligned}
 & \Phi(0,y,f)= y\\
 & \Phi(n+1,y,f)= f(n,\Phi(n,y,f)).
 \end{aligned}\]
 Let $m$ be such that $\theta(0,m)$ holds. Then it is easy to check by $\id 0 0$ that $\forall n \, \theta(n,\Phi(n,m,f))$ and so that $\forall n \,\exists y\, \theta(n,y)$.\end{proof}

\begin{proposition}[$\Pra$]\label{prop:mini}
The following are equivalent:
\begin{enumerate}
  \item $\qfac$,
  \item totality of minimisation for functions\footnote{ That is, for each $g \colon \Nb^{i+2} \to \Nb$ such that for each $n_0, \dots, n_i \in \Nb$ there exists $m \in \Nb$ such that $g(n_0, \dots, n_i,m)=0$, there exists   $h \colon \Nb^{i+1} \to \Nb$ such that $h(n_0, \dots, n_i)=\mu m \, (g(n_0, \dots, n_i,m)=0)$.}.
\end{enumerate}

\end{proposition}
\begin{proof}
For the sake of convenience, assume $i=0$.

$(1 \Imp 2)$ The $\Delta_0^0$-formula $g(n,m)=0$ satisfies the antecedent of $\qfac$, so let $f$ be such that $\forall n \, (g(n,f(n))=0)$. Define $h(n)= \mu m \leq f(n) \, (g(n,m)=0)$ by bounded minimisation.

$(2 \Imp 1)$ Assume $\forall n \,\exists m\, \theta(n,m)$, for some $\theta \in \Delta^0_0$,  and let $\Phi$ be as in \cref{Lemma_Delta00EquivQf}. Let $h$ be the function that exists for the fixed collection of parameters that  occur in the formula, in accordance with the corresponding axiom of $\Pra$.  Since $\forall n\, \exists m \, h(n,m)=1$, there exists $g$ returning the least $m$ witnessing that $h(n,m)=1$. Thus,   $\forall n \, \theta(n,g(n))$.
\end{proof}



\bigskip

 Since we are working over $\Pra$, we shall write simply
$\qfac$ for $\Pra +\qfac$, and the same for other additional axioms that we will encounter.
Notice that the two systems  $\rca$ and $\Pra +\qfac$ share the same consequences, as one can be interpreted in the other and vice versa using characteristic and pairing functions.
We thus shall  stretch our notation even further: 

\begin{framed}
\begin{center}
We  identify $\rca$ with $\Pra +\qfac$.
\end{center}
\end{framed}

However, we must not forget that the second-order objects in our studies are (total) functions rather than sets. There is a significant difference between the function-based and the set-based approaches when we go below $\rca$;  we will encounter this difference already in the proof of Proposition~\ref{pr:rca}. In our function-based approach, the set-version $2^{\mathbb{N}}$-$\rca= \is 0 1 \land \dca$ of $\rca$ is strictly weaker than the `full' version $\mathbb{N}^{\mathbb{N}}$-$\rca = \Pra +\qfac$ that we identify with $\rca$. We will see that there are theorems that imply the natural set-version $2^{\mathbb{N}}$-$\rca$ of $\rca$, but not $\Pra +\qfac$ (see e.g.\ \cref{Cor:RT-RCA}). This distinction will be made very clear when necessary.

\subsubsection{The obvious implications are strict}

\begin{proposition}
 $\Pra \nvdash \id 0 1$.
\end{proposition}
\begin{proof}
The following lemma is claimed in \cite[last line of p.~225]{kohlenbach2000} and in \cite[Theorem 2.1]{avigad_2005} for finite-type extensions of $\pra$.
\begin{lemma}\label{lem:cons}
 $\Pra$ is arithmetically conservative over $\pra$
\end{lemma}
\begin{proof}
Let $\varphi =  \forall y \theta (x,y)$ be a formula in $L_\pra$, for $\theta$ a $\Sigma_n$ formula, for some $n \in \nat$. Clearly, if $\pra \vdash \varphi$, then $\Pra \vdash \varphi$. For the reverse direction, assume $\pra \nvdash \varphi$ and so let $M$ and $n$ be such that $M\vDash \pra \land \neg  \theta (x,n)$. Then $(M, PRec(M)) \vDash \Pra \land \neg  \theta (x,n)$, so that $\Pra \nvdash \varphi$.
\end{proof}
The proposition now follows from the fact that, by~\cref{Prop_PraNonId01}, $\pra \nvdash \id {} 1$ (thus,  the induction may fail even without function-parameters).
\end{proof}

Since $\Pra\nvdash \iso$ one needs to pay attention to the precise definition of  \lq infinity\rq\ (see, e.g., \cite[Lemma 3.2]{SimpsonYokoyama}). In this paper  \lq infinity\rq\  means `unbounded'. Unless stated otherwise, all instances of the principles mentioned in this paper have domain $\Nb$ and, if the solution is required to be infinite, then it is required to be unbounded (though it may be the case that those requirements may be relaxed for some statements).

\begin{proposition} \label{Prop_DcaNotIS01}
Over $\Pra$, $\dca$ does not imply $\is 0 1$.
\end{proposition}
\begin{proof}
Consider $M\vDash \pra \land \id {} 1 \land \neg \is {} 1$, which exists by \cref{Prop:notIs}. Let $\exists y \, \theta(x,y)$ be the formula which witness the failure of $\is {} 1$, namely such that $M \vDash \exists y \,\theta(0,y)$ and $M \vDash \forall n\,(\exists y \,\theta(n,y) \imp \exists y\, \theta(n+1,y))$, but $M \vDash \exists n \,\forall y \,\neg \theta(n,y)$. Note that $\theta$ does not have any second-order parameters. Consider now the model $(M, \Delta^0_1\text{-def($M$))}$. It is clear that it satisfies $\dca$.
Moreover,  $(M, \Delta^0_1\text{-def($M$))}\vDash  \qfi$, since $M \vDash \id {} 1$. In fact, any second-order parameter in a quantifier-free formula can be substituted with its $\Delta^0_1$-definition, so to obtain a $\Delta_1$-formula which is equivalent to the original quantifier-free formula. This allows to conclude that  $(M, \Delta^0_1\text{-def($M$))}\vDash \Pra$.
However, $\exists n \,\forall y \,\neg \theta(n,y)$ witnesses that $(M, \Delta^0_1\text{-def($M$))} \vDash \neg \is 0 1 $.
\end{proof}

\begin{proposition} \label{Prop_IS01NotDca}
Over $\Pra$,  $\is 0 1$ does not imply $\dca$.
\end{proposition}
\begin{proof}
Consider $M\vDash \pra \land  \is {} 1$ --- e.g., the standard model. Then $(M, PRec(M))\vDash \Pra \land  \is 0 1 \land \neg\dca$, since there exists a computable, and hence $\Delta^0_1$, function which is not primitive recursive.
\end{proof}

\begin{corollary}
 Over $\Pra$, $\dca$ and $\is 0 1$ are incomparable. Moreover,  $\Pra$ does not imply $\dca$.
\end{corollary}

In contrast with the previous corollary, Proposition~\ref{prop:d00ca} says that $\Pra \vdash \dzca$, which is comprehension for $\Delta^0_0$-formulae. Also, if $(M, \mc{X})  \vDash \Pra \land \dca$, then $(M, \mc{X}) \vDash \id 0 1$, since any $\Delta^0_1$-formula becomes a $\Delta^0_0$-formula and so $\id 0 1$ reduces to $\id 0 0$.

\begin{corollary}\label{cor:weaker}
 Over $\Pra$, both $\dca$ and $\is 0 1$ are strictly weaker than $\rca$ (the latter is identified with $\qfac$).
\end{corollary}
\begin{proof}
If $\dca$ ($\is 0 1$) implies $\qfac$, then by \cref{Prop_QfacNotIS01} (resp., \cref{prop:p2}), it would imply $\is 0 1$ (resp., $\dca$) contrary to \cref{Prop_DcaNotIS01} (resp., \cref{Prop_IS01NotDca}).
\end{proof}

\begin{proposition}\label{pr:rca}
Over $\Pra$, $\dca \land \is 0 1$ does not imply  $\rca$.
\end{proposition}

\begin{proof}[Proof Sketch]
We are working in a standard model, and thus we do not have to worry about induction.
Begin with the minimal model of $\Pra$ which contains only primitive recursive functions over $\omega$.
Note that $\dca$ establishes the existence of only $\{0,1\}$-valued functions, and every such
function is  bounded by a primitive recursive function.

Consider a primitive recursive functional $\Psi$ on input $f_1, \ldots, f_k$. Since each of $f_1, \ldots, f_k$ is bounded by a primitive recursive function, there is a primitive recursive bound on the use of $\Psi(f_1, \ldots, f_k)$ and, therefore, a primitive recursive bound on the value of the output function on a given input. (We can simply go over all computations and take the maximum over all potential outputs.)
We cite Lemma 3.5 of \cite{punc2} for a detailed proof of a similar result. (Notice that the mentioned lemma applies, since we can produce the primitively recursively bounded compact subspace of $\omega^\omega$ and identify each $f_i$ with a path through this space.)

Iterate the process of closing the model under instance of $\dca$ and by primitive recursive operators (as required by $\Pra$) to construct an $\omega$-model of $\Pra$ which satisfies $\dca$ but fails $\qfac$ since it does not contain computable functions that are not dominated by primitive recursive functions.
\end{proof}

\begin{remark}
Recall that $\rca^*$ is the weakening of $\rca$ in which $\Sigma^0_1$-induction is replaced by $\exp$, stating the totality of exponentiation, and $\Sigma^0_0$-in\-duc\-tion, a.k.a.\ induction over formulae with only bounded quantifiers.
Since $\rca^*$ includes $\dca$ into its axioms, and so its minimal model includes
general recursive functions that are not primitive recursive, but does not prove totality of primitive recursive functions, while $\Pra$ does vice versa, the two theories give two independent axiomatic foundations below $\rca$.  Moreover, $\rca^*$ is a set-based second-order system, while $\Pra$ is function-based.

A peculiar fact is that, over $\rca^*$, $\Sigma^0_1$-induction is equivalent to
the statement that the universe of (total) functions is closed under primitive recursion; see, e.g., Lemma 2.5 in \cite{SimSm}.  That is, over $\rca^*$,  $\pra$ is equivalent to $\is 0 1$ (and, thus, to $\rca$).
 \end{remark}

\begin{remark}
In this paper we study some statements which have already been analysed from the classical reverse mathematics point of view. Such  statements typically  are formalised in the  set-based language of reverse mathematics~\cite{sosoa}, and thus they have to  be translated into our function-based language to be studied using $\Pra$ (as was done, for example, for $\dca$). Nonetheless, quite often such a careful distinction is not necessary, since sets can be canonically identified with their characteristic functions that are elements of $2^{\mathbb{N}}$. For example, $\wkl$  formulated in $\Pra$ guarantees that for each $T\colon 2^{<\Nb} \to 2$, such that $T^{-1}(1)$ is an infinite tree, there exists a function $P\colon \Nb \to 2$ such that $T(\la P(0), \dots, P(n)\ra) =1$ for each $n\in \Nb$.
When more care is needed, or when we adopt a different representation, we will mention it explicitly.
\end{remark}

\subsection{Calculus of finite sets}\label{sec:fin}
The main purpose of this subsection is to establish the following informal principle:
\begin{framed} \begin{center}
The formalisation of finite sets is robust in $\Pra$.
\end{center}
\end{framed}

When working in $\pra$, a finite set is usually identified with its code which is a number (a string).
In $\Pra$, where we actually do have sets (identified with their characteristic functions), we can also define a finite set to be a bounded set. It also makes sense to specify the bound rather than just state that it exists---the latter requires an unbounded quantifier.
In $\pra$,  \cite{sosoa} defines the cardinality of a finite set using a primitive recursive function (via the sum of a string) completely avoiding second-order considerations.
In $\Pra$, it is perhaps more natural to define the cardinality of a finite set using bijections with initial segments of $\mathbb{N}$.
We will see that these two approaches (the first-order and the second-order ones) to finite sets are equivalent over $\Pra$.
As a consequence, we can use them interchangeably.
This will be convenient when dealing with finite subsets of infinite sets.
We also establish some basic properties of finite sets that will be used throughout the rest of the paper. We will later use the notion of a cardinality to bound our search by looking at `the first $m$ elements of a structure'; specifics in the end of the subsection (Remark~\ref{rem:cardsearch}).

\subsubsection{Two definitions of a finite set} As usual, for every $i\in\mathbb{N}$, let $p_i$  denote the $i$-th prime number.

\begin{definition}
	Let $n\in \mathbb{N}$, and let $\bar a = a_0,a_1,\dots,a_n$ be a tuple from $\mathbb{N}$. The \emph{code} of the tuple $\bar a$ is the number
	\[
		\mathrm{code}(\bar a) = p_0^{a_0+1} \cdot p_1^{a_1 + 1} \cdot \dots \cdot p_n^{a_n+1}.
	\]
\end{definition}

\begin{lemma}\label{lemma:fin-tuple1}
		\textnormal{(}$\Pra$\textnormal{)}  The set
			\[
				CT = \{ m\in\mathbb{N} \,\colon (\exists n)(\exists a_0,\dots,a_n) [m = \mathrm{code}(\bar a)]\}
			\]
			is $\Delta^0_0$-definable.
\end{lemma}

\begin{proof}
 It is known that the following functions (on natural numbers) are primitive recursive:
	\begin{align*}
		\mathrm{ex}(i,x) &= \begin{cases}
			\max \{ l\in\mathbb{N} \,\colon (p_i^l \mid  x) \}, & \text{if } x > 0,\\
			0, & \text{if } x = 0;
		\end{cases}\\
		\mathrm{long}(x) &= \begin{cases}
			\max \{ i\in\mathbb{N} \,\colon (p_i \mid  x) \}, & \text{if } x > 1,\\
			0, & \text{if } x \in \{ 0,1\}.
		\end{cases}
	\end{align*}	
	Therefore, we deduce that $m\in CT$ if and only if $(m\geq 2) \land (\forall i \leq \mathrm{long}(m)) (\mathrm{ex}(i,m) > 0).$
	\end{proof}

In a function-based language,
we could choose to
identify a finite set with a function $f$ having bounded support, so that the bound is also given.
The main point of the elementary lemma below is to verify that these two intuitions coincide over $\Pra$.
More formally, Lemma~\ref{lemma:fin-tuple}  implies that there are two equivalent approaches to finite sets. Consider a non-empty finite set $F = \{ b_0 < b_1 <\dots < b_k\} \subset \mathbb{N}$.
\begin{enumerate}
	\item The set $F$ can be encoded by a single number $m = \mathrm{code}(\bar a) \in CT$, where
	\begin{itemize}
		\item $\bar a = a_0,a_1,\dots,a_{b_k}$;
		
		\item if $i\leq b_k$ and $i\not\in F$, then $a_i = 0$;
		
		\item if $i\leq b_k$ and $i\in F$, then $a_i = 1$.
	\end{itemize}
	
	\item The set $F$ can be encoded by a function $f$ and a number $\ell$ such that:
	\begin{itemize}
		\item $\ell = b_k$;
		
		\item $(\forall x > \ell)(f(x) = 0)$;
		
		\item $(\forall x \leq \ell)[ (x\in F \rightarrow f(x) = 1) \land (x\not\in F \rightarrow f(x) = 0)]$.
	\end{itemize}
\end{enumerate}

\begin{lemma}\label{lemma:fin-tuple}
		\textnormal{(}$\Pra$\textnormal{)} Suppose that $n,a_0,a_1,\dots,a_n \in \mathbb{N}$.  Then the following are equivalent:
			\begin{itemize}
			 	\item[(a)] there is $m\in CT$ such that $m = \mathrm{code}(a_0,\dots,a_n)$;
		 	
			 	\item[(b)] there exists a unary function $f$ such that
					\begin{itemize}
						\item $(\forall i\leq n) (f(i) = a_i + 1)$;
			
						\item $(\forall i > n) (f(i) = 0)$.
					\end{itemize}
		\end{itemize}
		
\end{lemma}
\begin{proof}
 (a$\Rightarrow$b)\ Assume that $m\in CT$. Then the desired function $f$ can be defined as follows:
	\[
		f(x) = \mathrm{ex}(x, m).
	\]
	
	(b$\Rightarrow$a)\ Given a function $f$  and a number $n$ (satisfying the conditions in $(b)$), the desired code $m$ is recovered as follows:
	\[
		m = \prod_{i=0}^n p_i^{f(i)}.
	\]
	Lemma~\ref{lemma:fin-tuple} is proved.
\end{proof}

In other words, a set is finite if, and only if, it is (explicitly) bounded.
We slightly abuse our notation and identify $S$ with the `pair' $(f, \ell)$ even though we actually do not use a pairing function of any kind to code $f$ and $\ell$ together into one parameter.

We also remark that, when we use the first-order approach to finite sets (and to finite maps alike) we can use $\Pi_2$-conservativity of $\wkl$ over $\pra$ and Lemma~\ref{lem:cons} to derive some of the basic, first-order, facts about (codes of) finite sets while arguing in  $\wkl$. Similar notions of finite sets and of cardinality are defined in the theory $\pam + \id {} 0 + \exp$ in \cite[Chapter 1.b]{hajekPudlak}. We highlight, in particular, Theorem 1.41 of \cite{hajekPudlak}, where a notion of cardinality similar to the one in \cref{Def_card} is introduced. Note that the results in \cite[Chapter 1.b]{hajekPudlak} are provable in $\Pra$, since $\Pra \vdash \pam + \id {} 0 + \exp$. However, for conveniency in the reverse-mathematical context, we chose a coding method for finite sets that differs from that  of H{\'a}jek and Pudl{\'a}k. Also, using concervativity  would not be much of a simplification though, as one can equally easily argue directly in $\Pra$. In the next few subsections we shall give these elementary proofs in $\Pra$. 


\subsubsection{Cardinality}
We use $\Delta^0_0$-induction and $\Delta^0_0$-comprehension throughout; recall that $\Pra$ proves these axiom schemata (see \S\,\ref{sub:delta00}).
Let $S = (f, d)$ be a finite set, where $f$ is a $\{0,1\}$-valued function and $d$ bounds its support.
Up to notation, the following definition is equivalent to the one found in Simpson~\cite{sosoa}:

\begin{definition}
Define the cardinality of  a finite set $S = (f,d)$ to be $$|S| = \sum_{i\leq d} f(i).$$
\end{definition}

Note that the above definition is witnessed by a primitive recursive functional
and therefore makes sense, and in particular for any $f$ and $d$ the cardinality $|S|$ is a number that can be obtained `uniformly' in the representation of $S$.

A different, perhaps occasionally more useful, notion of cardinality
is more similar to the usual set theoretic approach via bijections. However,
it will take some work to show that it is robust and is equivalent to the definition above.

\begin{definition} \label{Def_card}
Let $S$ be a finite set coded as $(f,d)$.
Define $card(S, m)$ to be the formula saying that there is a bijection between
$S$ and the initial segment $[0, \ldots, m-1]$.
\end{definition}

We note that the notion of a bijection between finite sets can be formalised in the language of $\Pra$; we omit this. Observe that one needs only bounded quantifiers to state that
a given function is a bijection between two given finite sets.
It is also easy to see that $\Pra$ proves that if $g: S \rightarrow L$ is a $1$-$1$ and onto map between two finite sets, then $f^{-1}$ exists and is also $1$-$1$ and onto.
We shall use these properties without explicit reference.

\begin{proposition}\label{prop:finite-cardinalities} \textnormal{(}$\Pra$\textnormal{)}
Let $S = (f,d)$ be a finite set. Then $card(S, m)$ holds if, and only if, $|S| = m$.
\end{proposition}

\begin{proof}
The proposition follows from the two lemmas:

\begin{lemma}\label{lem1:card} For any finite set $S$,
$card(S, |S|)$ holds.
\end{lemma}

\begin{lemma}\label{lem2:card}
For any finite set $S$ and any $m,k \in \mathbb{N}$,
$card(S, m) \land card(S, k)$ implies $m=k$.
\end{lemma}

\begin{proof}[Proof of Lemma~\ref{lem1:card}]  For simplicity, assume $S$ is not empty.
Using primitive recursion, define

\

$g(0) = \mu_{y \leq d} \, f(y) =1,$

\smallskip

$g(k+1) = \begin{cases}
\mu_{y \leq d} \, [y> g(k) \land f(y) =1], \\
d+1, \mbox{ if no such $y$ exists}.
\end{cases}$

\

Let $\psi(k)$ be a $\Delta^0_0$ formula (with parameter $S = (f, d)$)  saying that, if $g(k) \neq d+1$ then:
\begin{itemize}
\item $g \upharpoonright_{[0, \ldots, k]}$  is a bijection between $[0, \ldots, k]$ and $S \upharpoonright_{\leq g(k)}$;
\item $k+1= |S \upharpoonright_{\leq g(k)}|$.
\end{itemize}
It should be clear that we need only bounded quantifiers to write down $\psi(k)$.
We now can use $\Delta^0_0$-induction to demonstrate that $\forall k \psi(k)$ holds.
Recall $S \neq \emptyset$, so $g(0) \neq d+1$ is defined. We clearly have $g: [0] \rightarrow \{g(0)\}$ is a bijection, and $| \{g(0)\} | = 1$.

For the step, assume the statement holds for $k$. If $g(k+1) = d+1$, then we are done. Otherwise, $g(k+1)$ is defined and $g: [0, \ldots, k+1 ] \rightarrow S \upharpoonright_{\leq g(k+1)} = S \upharpoonright_{\leq g(k)} \cup  \{g(k+1)\}$ is a bijection. Also, $|S \upharpoonright_{\leq g(k)}| = \sum_{i \leq g(k+1)} f(i) = |S \upharpoonright_{\leq g(k)}| + 1$, and the lemma is proved\footnote{We implicitly used that $g$ is strictly increasing unless is equal to $d+1$; this also follows by $\Delta^0_0$-induction.}.
  \end{proof}

\begin{proof}[Proof of Lemma~\ref{lem2:card}]
It is easy to see that $|[0, \ldots, m]| = m+1$, by $\Delta^0_0$-induction.
It is sufficient to prove that, in $\Pra$, if there is a bijection $g : [0,\ldots, k] \rightarrow [0, \ldots, m]$ then $k=m$.
Assume $k<m$. Let $\psi$ be a bounded formula (with parameters $k$ and $g$) saying that, if $g: [0,\ldots, k] \rightarrow [0, \ldots, m]$ then $| g([0, \ldots, k]) | = k+1$.
If we can prove $\psi$, then the lemma will follow from
$m+1 = | [0, \ldots, m] | =| g([0, \ldots, k]) | = k+1,$
which is a contradiction.  But $\psi$  follows easily by ($\Delta^0_0$) induction, as follows.
$\psi(0)$ says that $|g( [0])| = |\{g(0)\}| =1$. For the step, observe that $g([0, \ldots,  k+1]) = g([0, \ldots,k]) \cup \{g(k+1)\}$ where the union is disjoint, so $|g([0, \ldots, k+1])| = |g([0, \ldots,k])|+1.$
\end{proof}

Proposition~\ref{prop:finite-cardinalities} is proved.
\end{proof}

\subsubsection{Set theoretic operations and bounded search} We can formalise the basic operations on finite sets (such as union, intersection, cartesian product, etc.) in the language of $\Pra$. In fact, all these elementary set theoretic operations with finite sets have a pleasant property of uniformity, meaning that each such operation is witnessed by a primitive recursive functional. In particular, we can uniformly calculate the upper bound of the output.
(This can also be formalised in $\pra$ using codes rather than second order names, and this would be equivalent in the right sense; we omit this.)

Using $\Delta^0_0$-induction, we can derive the following basic properties of finite sets and their cardinalities:

\begin{lemma}\label{lem:sets} Let $S = (f, d)$ and $K = (g, k)$ be finite sets.

\begin{enumerate}

\item When $S\cap K = \emptyset$ then $|S \cup K| =|S| +|K|.$
\item $|S \times K| = |S | \times |K|$.
\item $|S^n| = |S|^n$, for any $n \in \mathbb{N}$.
\item $S \subseteq K \implies |S| \leq |K|$.
\item $|S| < |K|$ implies that $\exists x \in S \setminus K$.
\end{enumerate}

\end{lemma}

\begin{proof} (1) and (2) follow by, e.g., $\Delta^0_0$-induction in the cardinality of $S$  while the cardinality of $K$ is held fixed (as a parameter). Item
(3) follows from (2) by $\Delta^0_0$-induction, and  so does $(4)$.
To see why (5) holds,  use $(4)$ to conclude that $|S \cap K| < |K|$.
So we can assume $S \subseteq K$.
 If for all $x \leq k = \max \{d,k\}$, $x \in K \iff x \in S$ then, by $\Delta^0_0$-induction, we would have $ \sum_{i\leq d} f(i) =  \sum_{i\leq k} f(i)  = \sum_{i\leq k} g(i)$,
 and since $f \leq g$,
  it must be that, for some $x \leq d$, $f(x)< g(x)$.
\end{proof}

We note that in $(5)$, we can uniformly search for such an $x \in S \setminus K$ in the sense that there is a primitive recursive operator which, on input names of $S$ and $K$ (recall names include their upper bounds), outputs the least such $x$. If we prefer functions rather than functionals, we can of course use codes instead of (explicitly) bounded functions. It is rather convenient that, at least in this case, the first-order and the second-order approaches agree.

\begin{remark}\label{rem:cardsearch}
As promised at the beginning of the subsection, we explain
how to  use the notion of cardinality to bound a search through $\mathbb{N}$. Suppose we know that the cardinality of the finite set $\{x: \varphi(x)\}$ is $m$. Recall we already observed that $| [0,\ldots, m]| = m+1$.
Using Lemma~\ref{lem:sets}  conclude that there is a $y \in [0, \ldots, m]$ such that $\neg \varphi(y) $.
\end{remark}

	

\section{Examples from countable algebra and infinite combinatorics}\label{sec:examples}

In this section we present several relatively basic  results  carried over $\Pra$.
We also present two rather different approaches to countable structures, one seems to be more suited for model theory, and the other one for countable algebra.
This section is essentially a semi-preliminaries section with lots of examples, however, it appears that all results
discussed here are actually new.



%
%

\subsection{Algebraic structures and vector spaces}\label{sub:str}
In $\Pra$ all second-order objects are functions. For instance, if we want to represent a countable
 algebraic structure in a finite signature 
 we do it as follows.

 \begin{enumerate}
 \item The domain (each domain, if a structure is $n$-sorted) is either $\mathbb{N}$ or an initial segment of $\mathbb{N}$
 (identified with its characteristic function).

  \item The operations are functions on the domain.

 \item Relations are represented by their characteristic functions.

 \end{enumerate}

 \begin{remark}We restrict the domain to make the search for the $k$th element of the structure a bounded search. As argued in, e.g., \cite{Kal17}, without this assumption structures are not `fully' primitive recursive in the standard minimal model.
 \end{remark}

\begin{remark}
Note that if a structure is not infinite, it does not necessarily mean we can always `uniformly' access the finite code of its domain; recall such a code must also include the upper bound, see \S\,\ref{sec:fin}. (Formally, there is no primitive recursive functional that, on input a structure, outputs the upper bound for its domain.)
\end{remark}



For instance, a countable vector space $V$ over $\mathbb{F}$
is a two-sorted structure in which $(\mathrm{dom}(V), +_{V},$ $-_{V}, 0_{V})$
is an abelian group together with scalar multiplication by elements of $\mathbb{F}$.




Let $V$ be a countable vector space over $\mathbb{F}$. Then a \emph{basis} of $V$ is given as a function $b\colon \mathbb{N} \to \mathbb{N}$ with the following property: every $v\in V$ can be expressed uniquely in the form
\[
  v = \sum_{k\in E_0} \alpha_k \cdot b(k),
\]
where:
\begin{itemize}
  \item there exists $n_0\in\mathbb{N}$ such that $k\leq n_0$ for every $k\in E_0$;

  \item for every $k\in E_0$, we have $\alpha_k\in \mathbb{F}\setminus\{ 0\}$.
\end{itemize}

It is not difficult to show that the fact below fails in $\rca$ if $\mathbb{F} = \mathbb{Q}$, even in the standard minimal model. (An observation that can be traced back to Mal'cev~\cite{Mal-62}.)
\begin{proposition}\label{prop:Basis-for-Vector-Spaces}
  \textnormal{(}$\Pra$\textnormal{)} Let $V$ be a countable $\mathbb{F}$-vector space
  over a finite field $\mathbb{F}$. 
  Then $V$ has a basis.
\end{proposition}
\begin{proof}
 We use the upper bound $\ell$ of $\mathbb{F}$  as well as its cardinality $k$, throughout.
  The procedure that we describe below can be witnessed by a primitive recursive functional that takes $V$ and $k$ and outputs the basis identified with its characteristic function. For simplicity, we restrict ourselves to infinite spaces and we assume that the domain of $V$ is $\mathbb{N}$.
  We also assume that $0$ denotes the zero of the space.

  We (usually, implicitly) use the materials of \S\,\ref{sec:fin} to operate with finite sets. In particular,  we use Lemma~\ref{lem:sets} to calculate cardinalities of sets and Remark~\ref{rem:cardsearch} to bound our search.

 \

\noindent  \emph{The idea} is to follow  the usual effective algebraic proof and search for the
smallest index element which is not already in the span of the finite part of the basis enumerated so far. It is not hard to see that, since the span of $n$ elements has the size of at most $k^n$, we can  uniformly bound our search.
The short version of the formal proof below is: ``this works in $\Pra$''. The construction would definitely work in the standard minimal model.
 But it takes some work to formally verify ---  using $\Delta^0_0$-induction, $\Delta^0_0$-comprehension, and properties of finite sets --- that this procedure works in $\Pra$. We give the details, but in later proofs similar details will often be omitted.

 \

\noindent  \emph{Formal proof.}  Recall sets are identified with their characteristic functions.
  First, we define the auxiliary set coding the relation of linear dependence:
  \[
    S = \{  \langle a_0,a_1,\dots,a_{n}\rangle \,\colon (\forall i\leq n) (a_i\in \mathrm{dom}(V)) \text{ and } [n = 0, \text{ or } n\geq 1 \text{ and } a_{n} \in \mathrm{span}(a_0,a_1,\dots,a_{n-1})]\}.
  \]
  The set $S$ is definable by a $\Delta^0_0$ formula with parameter $k = |\mathbb{F}|$. This is because, using primitive recursion, we can express $a_{n} \in \mathrm{span}(a_0,a_1,\dots,a_{n-1})$  as a $\Delta^0_0$-fact.

  We define a function $b$ (which provides a basis of $V$) by primitive recursion. The value $b(0)$ is chosen 
as some non-zero element from $V$, say, having index $1$.
  Suppose that the values $b(0),b(1),\dots,b(n)$ are already defined. Then we set
  \[
    b(n+1) := (\mu z \leq (k^{n+1} + 1))[z \in \mathrm{dom}(V) \land \langle b(0),b(1),\dots,b(n),z\rangle \not\in S].
  \]
  Since there exist at most $k^{n+1}$ linear combinations
   of the vectors $b(0),b(1),\dots,b(n)$, we deduce that the value $b(n+1)$ is well-defined.


  This concludes the construction of the function $b$. Now we need to prove that $b$ gives a basis of $V$. For convenience, for $n\in\mathbb{N}$, by $b_n$ we denote the vector $b(n)$.

  \smallskip

  \textbf{(1)}\ First, we show that there is no $n\in \mathbb{N}$ such that there exists a non-trivial linear combination $u$ of $b_0,b_1,\dots,b_n$ such that $u = 0$.

  Consider the set
  \[
    T_0 = \{ n\in\mathbb{N} \,\colon \text{some non-trivial linear combination of } b_0,\dots, b_n \text{ equals } 0\}.
  \]
 Note that $T_0$ is $\Delta^0_0$-definable. Towards a contradiction, assume that $T_0$ is not empty.  Since $\Pra \vdash \id 0 0$, we deduce  that there exists the least number $n$ belonging to $T_0$.  Consider the non-trivial combination
  \[
    0 =  \alpha_0 b_0 + \alpha_1 b_1 + \dots + \alpha_{n-1} b_{n-1} + \alpha_n b_n.
  \]
	Without loss of generality, one may assume that $n\geq 1$.

  \emph{Case~1.} Assume that $\alpha_n \neq 0$. Then we have
  \[
    b_n = (-\alpha_n^{-1}\alpha_0) b_0 + \dots + (-\alpha_n^{-1}\alpha_{n-1})b_{n-1},
  \]
  and hence, $\langle b_0,\dots,b_{n-1},b_n\rangle \in S$, which contradicts with how the vector $b_n$ is chosen in the construction.

  \emph{Case~2.} Otherwise, we have $\alpha_n = 0$. Then $u = 0$ is a non-trivial combination of $b_0,\dots,b_{n-1}$, and this contradicts the minimality of the number $n$.

  We conclude that the set $T_0$ is empty, and the vectors $b_n$, $n\in\mathbb{N}$, are linearly independent.

  \smallskip

  \textbf{(2)} Second, we show that every non-zero vector $w$ is a linear combination of $b_0,b_1,\dots,b_n$, for some $n\in\mathbb{N}$. We consider an auxiliary set
  \[
    T_2 := \{ i\in\mathbb{N} \,\colon i\not\in \mathrm{span}(b_0,b_1,\dots,b_i) \text{ and } i\neq b_{i+1}\}.
  \]
  It is sufficient to prove that the set $T_2$ \emph{is empty}: indeed, if this is true, then every $i$ can be written as a linear combination of vectors $b_0,b_1,\dots,b_{i},b_{i+1}$.
  Assume that $T_2$ is non-empty. Since $T_2$ is $\Delta^0_0$-definable, there exists the least $i$ belonging to $T_2$. Without loss of generality, we may assume that $i\neq 0$.
Consider the vector $b_{i+1}$~--- the construction ensures that $b_{i+1}\not\in \mathrm{span}(b_0,b_1,\dots,b_i)$. There are two cases:

\emph{Case~1.} Assume that $b_{i+1} =j < i$. We have
\[
  b_{i+1} \not\in \mathrm{span}(b_0,b_1,\dots,b_i) \supset  \mathrm{span}(b_0,b_1,\dots,b_{j}).
\]
In addition, $b_{i+1}\neq b_{k+1}$ for all $k < i$. In particular, $j = b_{i+1} \neq b_{j + 1}$. Hence, $j\in T_2$, which contradicts the fact that $i$ is the minimal element of $T_2$.

\emph{Case~2.} Otherwise, $b_{i+1} > i$. But then the choice of $b_{i+1}$ implies that $i\in \mathrm{span}(b_0,b_1,\dots,b_i)$; a contradiction.

The remaining case is when $b_{i+1}  =i$.
We deduce that for every $i\in\mathbb{N}$,
\begin{equation*} 
  i\in \mathrm{span}(b_0,b_1,\dots,b_i) \text{ or } i = b_{i+1}.
\end{equation*}

\smallskip

\textbf{(3)} Now it is sufficient to prove that every non-zero vector $u\in V$ admits a \emph{unique} decomposition in our basis. Consider a $\Delta^0_0$-definable set
\[
  T_3 = \{ n\in\mathbb{N}\,\colon \text{two different linear combinations over } b_0,b_1,\dots,b_n \text{ are equal}\}.
\]
If the set $T_3$ is non-empty, then it contains the least element $n_0$. But then, a standard argument shows that the vectors $b_0,b_1,\dots,b_{n_0}$ are linearly dependent, i.e., $n_0$ also belongs to the set $T_1$, which gives a contradiction.
\end{proof}

In the `classical' reverse mathematics over $\rca$, vector spaces have attracted a considerable attention.
For example, in \cite{subspaces-Downey-et-al}, it was shown that  the existence of a nontrivial proper subspace of a vector space of dimension greater than one (over an infinite field) is equivalent to $\wkl$
 over $\rca$, and that the existence of a finite-dimensional nontrivial proper subspace of such a vector space is equivalent to $\aca$
 over
$\rca$. Further related results can be found in \cite{infinite-dimensional-Conidis}.
We suspect that many of these results might still hold over $\Pra$, in one way or another.

\begin{question}\label{q:1}
Investigate  proper subspaces of vector spaces over $\Pra$.
\end{question}


\subsection{Countable categoricity}\label{Subsec:cat} Many standard results in infinite combinatorics and model theory are somewhat evidently relying on unbounded search (unbounded existential quantification) with no further restriction on the search.  If $\Pra$ is the `right' system to study unbounded search, then these basic results \emph{should} be equivalent to $\rca$ over $\Pra$.  In this subsection we clarify this intuition with a number of examples that are summarised in the theorem below.

Recall that $\rca$ is identified with its function-based version $\qfac$. All structures in the theorem are countable.

\begin{theorem}\label{thm:easyhomo}  Over $\Pra$, $\rca$ is equivalent to each of the following:
\begin{enumerate}
\item Categoricity of  dense linear orders without end points.
\item Categoricity of random graphs.
\item Categoricity of atomless Boolean algebras.

\end{enumerate}

\end{theorem}
The rest of the section is devoted to the proof of the theorem.
We will define all terms used in the theorem very shortly.

\

We begin with the folklore result about dense linear orders.
Categoricity of countable dense linear orders without end points says that, whenever  $(\mathbb{N},<_A)$ and $(\mathbb{N},<_B)$ are dense linear orders without end points, then there exists an isomorphism $h$ from $(\mathbb{N},<_A)$ onto $(\mathbb{N},<_B)$ such that its inverse $h^{-1}$ also exists.

\begin{proposition}\label{Lem_CatDense}
 Over $\Pra$, $\rca$ is equivalent to categoricity of countable dense linear orders without end points.
\end{proposition}
\begin{proof}
We give a rather detailed proof, but in later arguments  similar details will be omitted.

\

 $(\Imp)$. Let $(A,<_A)$ and $(B,<_B)$ be two  dense linear orders without end points. Since it holds that
 \[
 	\forall \la a_0,a_1 \ra \, \exists \la c, d, e\ra \, (a_0 <_A a_1 \rightarrow c <_A a_0 <_A d <_A a_1 <_A e ), 
\]
then by $\qfac$ (identified with $\rca$), there exists $f \colon A \to B$ which given any pair of distinct elements of $A$, returns a triple constituted of one element smaller than the pair, one in-between the pair, and one greater than the pair, namely such that
\[
	\forall \la a_0,a_1 \ra \, (a_0 <_A a_1 \rightarrow \pi_1f(\la a_0,a_1\ra) <_A a_0 <_A \pi_2f(\la a_0,a_1\ra) <_A a_1 <_A \pi_3f(\la a_0,a_1\ra).  
\]
Following an analogous reasoning, we also get $g\colon \Nb \to \Nb$ which does the same for $B$.

We define an isomorphism $h \colon A \to B$ by the usual back-and-forth argument. Without loss of generality, we may assume that $0<_{A}1$ and $0<_{B}1$. So, beforehand we put $h(0) = 0$ and $h(1) = 1$.

Assume $h$ is a partial isomorphism between $\{a_0, \dots, a_n\}$ and  $\{b_0, \dots, b_n\}$, where $n\geq 1$, $a_0=b_0 = 0$, and $a_1 = b_1 = 1$. Assume that $a_{i_0} <_A \dots <_A a_{i_n}$, and let $a_{n+1}$ be the least element of $A \setminus \{a_0, \dots, a_n\}$. There are three cases to be considered:
 \begin{enumerate}
   \item $a_{i_n} <_A a_{n+1}$, then let $h(a_{n+1})=\pi_3g(\la 0, h(a_{i_n})\ra)$, 
   \item $a_{n+1} <_A a_{i_0}$, then let $h(a_{n+1})=\pi_1g(\la h(a_{i_0}), 1\ra)$, 
   \item $a_{i_j} <_A a_{n+1} <_A a_{i_k}$, for some $j,k \le n$. Then let $h(a_{n+1})=\pi_2g(\la h(a_{i_j}), h(a_{i_k})\ra)$. 
 \end{enumerate}
 When we have to define $b_{n+1}$, we do the same using $f$ in place of $g$.

It is immediate to prove that $h$ is injective, surjective (recall that $\iso$ is implied by $\qfac$), and respects $<_A$ and $<_B$. In addition, the construction also gives the existence of the inverse map $h^{-1}$.

\smallskip

$(\Leftarrow)$. Our argument relies on a coding strategy from  Theorem~2 of~\cite{BazhKal-21}.
Let $\psi(x,y)$ be a $\Delta^0_0$ formula (possibly with function parameters) such that $\forall x \exists y \psi(x,y)$. We need to build a function $f(x)$ such that $\forall x \psi(x,f(x))$. 

First, we fix a dense linear order $\mathcal{A} = (\mathbb{N},<_{A})$ that comes with the Skolem function $g_A(x,y)$: if $x<_A y$, then $x<_A g_A(x,y) <_A y$.
We can appeal to, e.g., the standard construction of the rationals adapted to $\Pra$ and then either illustrate that the Skolem function is primitive recursive or appeal to
 Propositions~\ref{pr:m1} and \ref{pr:m2} (and the well-known fact that the theory of dense linear orders admits primitive recursive elimination of quantifiers) to conclude that such a dense linear order exists.

\begin{footnotesize}\begin{remark}\label{rem:dense-ext} It is also not hard to argue in $\Pra$ directly, and explicitly  define $\mathcal{A} = \bigcup_{s\in \mathbb{N}} A_s$  by primitive recursion as follows.
Define $A_0 = \{ 0<_{A} 1\}$. Consider the primitive recursive function
\[
	q(0) = 2, \quad q(x+1) = 2q(x)+1.
\]
Suppose we have $A_s = \{  a^s_1 <_A a^s_2 <_A \dots <_A a^s_{q(s)}\}$.

We choose the least numbers $c_0 <_{\mathbb{N}} c_1 <_{\mathbb{N}} \dots <_{\mathbb{N}} c_{q(s)}$ from $\mathbb{N} \setminus A_s$, and we define
\[
	A_{s+1} = \{  c_0 <_A a^s_1 <_A  c_1 <_A a^s_2 <_A  c_2 <_A  \dots  <_A a^s_{q(s)} <_A c_{q(s)}\}.
\]
Say that $A_{s+1}\setminus\{ c_0,c_{q(s)}\}$ is the \emph{finite dense extension} of the order $A_s$ by numbers $c_1,c_2,\dots,c_{q(s)-1}$.
The desired order $\mathcal{A}$ is defined as follows: $x <_A y$ if and only if inside the finite order $A_{\max(x,y)}$, $x$ is less than $y$.
Similarly to the previous proofs, one can argue in  $\pra^2$ and show that $\mathcal{A}$ is a well-defined linear order with domain $\mathbb{N}$. In addition, there is a function $g_A(x,y)$ with the following property: if $x<_A y$, then $x<_A g_A(x,y) <_A y$. This, in particular,  shows that the order $\mathcal{A}$ is dense. In a similar way, one can show that $\mathcal{A}$ does not have end points.
\end{remark}
\end{footnotesize}

\smallskip

Second, we define another dense linear order $\mathcal{B} = (\mathbb{N}, <_{B})$. This order `encodes' the formula $\psi(x,y)$. Without loss of generality, we may assume that $\psi(0,0)$ is true. The order $\mathcal{B}$ is built by primitive recursion, as follows.

 We put $B_0 = \{ 0 <_{B} 4 <_B 2 <_B 6\}$. Assume we have defined $B_s$ (going from $<_B$-left to $<_B$-right) such that:
\begin{itemize}
	\item the order $A_s$ is copied on the numbers $4k$~--- more formally, we have $\{ 4a^s_1 <_B 4a^s_2 <_B \dots <_B 4 a^s_{q(s)}\}$;
	
	\item for each $k\leq s$, the interval $[4k+2;4k+6]_{B_s}$ is such that each number $x$ strictly between $4k+2$ and $4k+6$ is odd.
\end{itemize}
The order $B_{s+1}$ is then defined as follows:
\begin{enumerate}
	\item We add the number $4s+10$ as its greatest number.
	
	\item In a natural way, we extend the copy of $A_s$ to the copy of $A_{s+1}$.
	
	\item For each $k\leq s+1$, if $(\exists y \leq s+1) \psi(k,y)$, then the $B_{s+1}$-interval $[4k+2;4k+6]_{B_{s+1}}$ is constructed as the finite dense extension of $[4k+2;4k+6]_{B_{s}}$ (see Remark~\ref{rem:dense-ext}) by the least odd numbers not belonging to $\mathrm{dom}(B_{s+1})$ at the moment.
\end{enumerate}

Similarly to $\mathcal{A}$, we say that $\mathcal{B} \models (x <_B y)$ if and only if inside the finite order $B_{\max(x,y)}$, $x$ is less than $y$. It is not hard to show that $\mathcal{B}$ is a well-defined linear order on $\mathbb{N}$. In addition, it does not have end points.

One can easily prove that there is a function $g'_B(x,y)$ with the following property: if $x<_B y$ and $\{ x,y\} \neq \{ 4k+2,4k+6\}$, then $x <_B g'_B (x,y) <_B y$.

In order to show that $\mathcal{B}$ is dense, now we need to consider the remaining non-trivial case: suppose that $x = 4k+2$ and $y = 4k+6$. Then we know that there exists $z_0$ such that $\psi(k,z_0)$ holds. Then our construction ensures that inside the order $B_{\max(4k+6,z_0)}$, there exists an element $w$ with the property $x <_B w <_B y$.

\smallskip

Let $h$ be an isomorphism from $\mathcal{B}$ onto $\mathcal{A}$. Consider the function
\[
	\xi(k) = h^{-1}(g_A( h(4k+2), h(4k+6))).
\]
The construction of $\mathcal{B}$ guarantees the following: the number $\xi(k)$ is odd, and
\[
	\xi(k) \geq \text{the least $s$ such that the interval } [4k+2;4k+6]_{B_s} \text{ contains odd numbers}.
\]
Hence, we deduce $(\exists y \leq \xi(k)) \psi(k,y)$. We define the function $f(k) := (\mu y \leq \xi(k))[\psi(k,y)]$. It is clear that we have $\forall x \psi(x,f(x))$. \cref{Lem_CatDense} is proved.
\end{proof}






\

\begin{definition}\label{Def_Random}
An undirected graph $(\nat,E)$ is \emph{random} if for each pair of disjoint non-empty finite sets $X,Y \subseteq \nat$, there exists a vertex $z \in \nat$ such that $\forall x \in X \,(x \E z)$  and  $\forall y \in Y \,\neg(y \E z)$.
\end{definition}

\begin{proposition}\label{Lem_CatRandom}
  Over $\Pra$, $\rca$ is equivalent to categoricity of countable random graphs.
\end{proposition}
\begin{proof}[Extended sketch]

Assuming categoricity, we sketch how to prove $\qfac$.
The construction is similar to \cref{Lem_CatDense}, but we have to be more careful with `witnesses' since they will no longer be independent from each other.
  We observe that the standard construction of the random graph via Fra{\"\i}ss{\'e} limit of finite graphs is primitive recursive; we use $\Delta^0_0$-induction to verify
 that the resulting structure $\mathcal{A}$ indeed satisfies \cref{Def_Random}
   and, furthermore, has a primitive recursive Skolem function for existential formulae.

We also fix an instance $\psi(x,y)$ of $\qfac$ and define a `bad' random graph $\mathcal{B}$ by primitive recursion, as follows.
 Define $B_s$ to be a clique on $s$ nodes unless  $\psi(0,s)$ holds; by bounded minimisation, we can assume that $s$ is the least such (in other words, `$f(0) =s$', where $f$ is the minimal solution to the instance).
In this case define $B_{s_0}$ by adding a new point not connected to any other point defined so far. (Note that, to calculate $f(0)$ primitively recursively, it is now sufficient to find at least two nodes in $\mathcal{B}$ not connected by an edge.)

Then we temporarily switch to defining $\mathcal{B}$ according to
the standard Fra{\"\i}ss{\'e} construction, but beginning with $B_{s_0}$ (rather than with
the empty graph).  We continue according to the Fra{\"\i}ss{\'e} construction  until the $n$th finite configuration requirement, in the primitive recursive list of Fra{\"\i}ss{\'e} extension requirements, is met.  This way we define $B_{t_0}$ where $t_0$ is uniformly primitive recursive in $s_0$.

We then turn to coding $f(1)$, as follows. Resume adding fresh nodes to $\mathcal{B}$
and declare them connected to the already existing nodes. Do so unless $\psi(1,s_1)$ holds (where $s_1$ is the least such). If  $\psi(1,s_1)$ holds, then we declare that the node $s_1$
is not connected to all nodes $x< s_1$.
(Note that, to calculate $f(1)$ primitively recursively, it is now sufficient to find at least $(t_0+1)$-many nodes in $\mathcal{B}$ at least one of which is not connected to the rest of nodes by an edge.)
  We then switch again to the Fra{\"\i}ss{\'e} construction for primitively recursively many steps,
  and then code $f(n+1)$ primitively recursively using $f(n)$, $B_{t_n}$, and $\psi(n+1,x)$.

  Using the materials of \S\,\ref{sec:fin} we can argue that the definition of $\mathcal{B}$ is primitive recursive, so $\mathcal{B}$
  exists, and that it satisfies \cref{Def_Random}. For the latter, we appeal to the  Fra{\"\i}ss{\'e} construction  which is used simultaneously with the coding, albeit with a potentially unbounded `delay'.

  Now suppose $g$ is an isomorphism from $\mathcal{A}$ to $\mathcal{B}$. To calculate
  $f(0)$ so that $\psi(0, f(0))$ holds, primitively recursively pick a pair of points in $\mathcal{A}$ not connected by an edge and calculate their $g$-images in $\mathcal{B}$.
  Assume $f(0), \ldots, f(k)$ have already been calculated.  Primitively recursively, fix $(t_k+1)$-many nodes in $\mathcal{A}$ so that at least one of them is not connected to the rest by an edge. By $\Delta^0_0$-induction, at least one of the $g$-images of these nodes has index $d \geq f(k+1)$.
\end{proof}

\

We finish this section with a similar, also expected, result about countable atomless Boolean algebras. We view a Boolean algebra as an algebraic structure in the signature $(\vee, \wedge, \neg{\,}, 0, 1)$ satisfying the standard axioms of Boolean algebras. We say that a Boolean algebra is atomless if for every $x\neq 0$, there exist non-zero $z, y$ such that $z \vee y = x$ and $z \wedge y = 0 $; all these definitions can be formalised in $\Pra$.

\begin{proposition}[$\Pra$]
  Over $\Pra$, $\rca$ is equivalent to categoricity of countable atomless Boolean algebras.
\end{proposition}

\begin{proof}[Sketch] The usual, the folklore `computable' proof can be formalised in $\rca$.

 Following the general pattern,
we observe that in $\Pra$ there is the `natural' atomless Boolean algebra $\mathcal{A}$ with a Skolem function.
We informally explain how to code an instance of $\qfac$ into an atomless Boolean algebra $\mathcal{B}$ so that any isomorphism from $\mathcal{A}$ onto $\mathcal{B}$ can be used to primitively recursively recover a solution to the instance.
As before, fix a $\Delta^0_0$ instance  $\psi(x,y)$ of $\qfac$.

We use properties of finite sets throughout (\S\,\ref{sec:fin}). Without loss of generality, we may assume that we have $\neg\psi(x,y)$ for every $y < x$: if needed, replace $\psi(x,y)$ with
\[
	\psi'(x,y) = \begin{cases}
		\text{false}, & \text{if } y<x,\\
		\psi(x,y-x), & \text{if } y\geq x.
	\end{cases}
\]
Then using primitive recursion, we can define the function $\ell(s)$ that outputs
the cardinality of the longest initial segment of $\mathbb{N}$
such that, for every element $x$ of this segment, $\psi(x, y)$ holds for some $y \leq s$.

In $\mathcal{B}$, reserve a special element $d \notin \{0,1\}$. Define $\mathcal{B} = \bigcup_s B_s$ by initial segments so that
 a new element $s$ is added below $d$ in $B_s$ only if $\ell$ has increased.
 (Otherwise,  adjoin a new element below $\neg d$.)
 Note that in this case $s$ bounds all witnesses that have been used in the definition of $\ell$. Informally, the numbers of elements below $s$ `code' the enumeration stages of a solution of the instance of $\qfac$. Since $\psi$ was an instance of $\qfac$, it follows by $\Delta^0_0$-induction that the resulting $\mathcal{B}$ satisfies the definition of a countable atomless Boolean algebra.

  We can also  show in $\Pra$ that there is a function which, on input (an index of) a finite set
  with at least $m$ elements below $d$, outputs the finite tuple of solutions
  $f(x)$ for all $x < m$ (together with their common bound).

  Now, if $g: \mathcal{A} \rightarrow \mathcal{B}$ is an isomorphism, then we can fix $d' \in \mathcal{A}$ such that $f(d') = d$. Since $\mathcal{A}$ possesses a Skolem function, given $m$ we can calculate a finite set $D$ containing only elements in $\mathcal{A}$ that lie below $d'$
  and such that $|D| = m$.
  Since $f$ is an isomorphism, it follows that  $f(D)$ is a finite subset below $d$ having the same cardinality as $D$. By the argument outlined above, this gives a primitive recursive procedure that defines a solution to the instance of $\qfac$.
\end{proof}

Note that each  categoricity result in Theorem~\ref{thm:easyhomo} evidently holds in $\rca$.
In fact, (1)--(3) of Theorem~\ref{thm:easyhomo} would be provable in $\Pra$ if we used structures augmented with a Skolem function
for existential formulae.
It is expected that results in $\Pra$ are more sensitive to the choice of coding than similar results in $\rca$. However,
$\rca$ also distinguishes between `structures' and `structures with Skolem functions': this reflects that, in computable algebra, not every computable structure
is decidable. 
Structures with Skolem functions are very useful (and indeed, seem unavoidable) when one needs to appeal to elementary model theory, as will be explained in  \S\,\ref{sebs:models}. More `honest' presentations of graphs will also play a significant role in the subsection below.

\subsection{Infinite combinatorics done in $\Pra$}\label{subsect:combinat}
We claim that many classical results in infinite combinatorics from the literature can be proved in $\Pra$. We give several examples below.
 In many cases we get these results almost for free if we follow proofs from the literature very closely, even though some extra care must be taken. Some of these proofs are non-trivial and quite lengthy.
We therefore shall not give many formal details since it would  drastically inflate the paper.  Thus, some of the claimed results below should perhaps be viewed as strong conjectures since we leave the details to the reader.

Often in combinatorics  theorems that fail to be computable in general become computably true when  restricted to a specific subclass of instances. We recall here two such results, namely Rival-Sands theorem for graphs and Hall's theorem, which in their generality are equivalent to $\aca$ as proved in \cite[Theorem 3.5]{FCSS} and \cite[Theorem 2.2]{hirstPhD} respectively. Nonetheless, `computable' restrictions of these results are also known. In the next subsections we (essentially) verify that those restrictions hold primitively recursively as well. We also have to be careful and make sure that only bounded quantifier induction is used (if any).


\subsubsection{Szpilrajn’s Theorem and graph reorientation} An \emph{oriented graph} is a directed graph such that at most one of the edges between two vertices exist. An oriented graph is \emph{pseudo-transitive} if for every $a,b,c
\in V$ such that $a \imp b$ and $b \imp c$ we have also $a \imp c \lor c \imp a$.  A \emph{reorientation} of an oriented graph $(V,\imp)$ is an oriented graph obtained by reversing some of the edges, or more formaly a relation $R$ on $V$ such that for each
$a,b\in V$, if $a\imp b$ then either $a \R b$ or $b \R a$ and if $a \R b$
then either $a\imp b$ or $b\imp a$.
A \emph{transitive reorientation} of $(V, \imp)$ is a reorientation of $(V,
\imp)$ which is also transitive.

\begin{proposition}
$\Pra$ proves the following:
\begin{enumerate}
    \item Szpilrajn's Theorem, i.e., each poset can be linearly extended.
    \item Every pseudo-transitive oriented graph has a transitive reorientation
\end{enumerate}
\end{proposition}
\begin{proof}[Proof idea]
$(1)$ The proof of the computable version of Szpilrajn's Theorem (see  \cite[Beginning of Sect.~10.2]{Hirschfeldt15})  can be transformed into a proof in $\Pra$. We outline the proof.

Given a poset $(P,<_P)$ and an enumeration of the vertices $(p_n)_{n\in \Nb}$, the linear extension is defined by stages. At a stage $s$ a linear extension $\prec_s$ of $<_P$ has been defined on  $\{p_0, \dots, p_{s-1}\}$. Then at the stage $s+1$, $\prec_s$ is extended with either $p_s \prec p_i$ or $p_i \prec p_s$, for each $i \leq s$. The relation between $p_s$ and $p_i$ is settled via checking only $<_P \upharpoonright \{p_0, \dots, p_{s}\}$ and $\prec_s$. Hence, it does not involve any unbounded search in the input, this means that one can actually write a primitive recursive functional, defined by primitive recursion, that takes $<_P$ as parameter and, at each stage $s$, inspecting (the code for) $s$, outputs (the code for) the linearisation of $\{p_0, \dots, p_s\}$.
In order to verify that the described construction gives a solution, one needs to check that the defined relation is a linear order and that the relation extends $<_P$. This can be done using only bounded induction.

\smallskip

$(2)$ Fiori-Carones and Marcone \cite{FCM} have recently designed an `on-line' algorithm to transitively reorient  pseudo-transitive oriented graphs. As discussed in the cited paper, `on-line' means that there is a functional which, given the pseudo-transitive oriented graph as input, outputs the transitive reorientation. Moreover, one can observe that, once the first $n$ vertices in the enumeration of the graph are transitively reoriented, then the relations between them and the $(n+1)$-st vertex are decided by the algorithm based only on the adjacency relations between those vertices and on the partial output (which transitively reorients the first $n$ vertices). In other words, it is possible to decide the first $n+1$ bits of the output looking only at the first $n+1$ bits of the input, provided that each vertex comes along with the entire information about its adjacency relation with the vertices previously enumerated, and thus no  search, in particular no unbounded search, is needed for the functional. This observation leads us to conclude that that functional is actually primitive recursive. To claim that the statement can be proved in $\Pra$, one also needs to check that the induction used in the proof is limited to $\id 0 0$.
\end{proof}

\subsubsection{Rival-Sands theorem  and graph colouring}   

Let $(V,E)$ be an undirected graph. Then $N(x)$ denotes the neighbours of $x$, for any $x \in V$; for convenience we assume that $v\in N(v)$, for all $v \in V$.
 The graph is \emph{locally finite} if $N(x)$ is finite for each $x \in V$.
An \emph{honest (locally finite) presentation}
of a locally finite graph $ (V,E)$ is a presentation of $(V,E)$ (as usual) together with
a    function $b \colon V \imp \Nb$ such that $b(x)$ gives the code of all neighbours of $x$, for each $x \in V$.
Intuitively, `honest' presentations correspond to `highly recursive graphs' in computable combinatorics.

Let $(V,E)$ be a graph. A total function $c \colon V \to n$ is said to be an $n$-colouring
iff  $c(v)\ne c(u)$ for each $\{v,u\} \in E$. A graph is $n$-colourable if there exists an $n$-colouring for it.

 In computable combinatorics,
Schmerl~\cite{schmerl_1980} proved that if $(V,E)$ is highly recursive and $n$-colourable, then $(V,E)$ is computably $(2n-1)$-colourable, but there exists  such a graph that is not computably $(2n-2)$-colourable (see also \cite[Theorem 4.21]{Gasarch}). In order to fit into the context of reverse mathematics (inspired by the known result that a graph if $n$-colourable if and only if every finite subgraph is $n$-colourable), we consider a weakening of the statement, and prove that it holds in $\Pra$ following essentially the Schmerl's argument. For more results about colourings of graphs, see \cite[Section 4]{Gasarch}.

\begin{proposition}
$\Pra$ proves the following:
\begin{enumerate}
 \item Rival-Sands theorem for honestly presented graphs, i.e., for every honest presentation of
  a locally finite infinite graph $ (V,E)$, there is an infinite $H \subseteq V$ such that for every $x \in V$, $x$ is   adjacent to at most one vertex in $H$.
 \item   If $(V,E)$ is honest and each finite subgraph is $n$-colourable, then $(V,E)$ is $(2n-1)$-colourable.
\end{enumerate}
\end{proposition}
\begin{proof}
$(1)$ We  follow \cite[Proposition 3.4]{FCSS} closely. Given an instance of Rival-Sands theorem $(V,E)$ and a function $b \colon V \to \Nb$ witnessing that $(V,E)$ is honest, 
a solution is defined by primitive recursion, as follows. Once $x_0, \dots, x_{s-1}$ are defined, consider the set of neighbours of neighbours of those vertices, which can be primitively recursively computed. The set is clearly finite, and thus can be coded by a number $c$. Then let $x_s= c+1$. This shows that the solution can be computed by a primitive recursive functional, which takes the instance as a parameter. At each stage, the functional searches for a new vertex (in the enumeration of $V$), and the performed search is primitively recursively bounded.

\smallskip

$(2)$ Let $(v_n)_{n\in\mathbb{N}}$ be an enumeration of $V$, and $b \colon V\to \Nb$ be the function witnessing that $(V,E)$ is honest.
 A colouring $c \colon V \to 2n-1$ is defined by (primitive) recursion, so that   $c\upharpoonright \{v_0, \dots, v_s\}$ is defined at `step' $s$. At step $0$ let $c(v_0)=1$. Assume that at a step $s$ the following two conditions are met:
 \begin{enumerate}
        \item $c_s \colon X \to 2n-1$ is a colouring,  $X\subseteq V$ is finite, and $\{v_0, \dots, v_s\} \subseteq X$,
        \item  the vertices in the set $B_s= \{v \in X \mid  \exists u \in b(v)\, (u \notin X)\}$ are either coloured with $\{1, \dots,n-1 \}$ or $\{n+1,\dots, 2n-1\}$.
    \end{enumerate}
At step $s+1$ we colour $v_{s+1}$ and possibly some other vertices. If $v_{s+1} \in X$ (i.e., $v_{s+1}$ is already coloured),  we let $c_{s+1}=c_s$, so that conditions (1)--(2) still hold, with $v_s$ and $c_s$ replaced by $v_{s+1}$ and $c_{s+1}$. Thus, we proceed to the next stage.

Otherwise, let
\[H=\{v \in V\setminus X \mid \exists w \in b(v)\,\exists u \in b(w)\, (u \in X)\}  \cup \{v_{s+1}\}. \]
Assume $B_s$ is coloured with $\{1, \dots,n-1 \}$, the other case being analogous. Notice that $H$ is a finite set, since  the graph is locally finite and $H$ is a subset of the neighbours of the neighbours of  $X$, which is assumed to be finite by (1). Moreover, one can explicitly bound the size of $H$ thanks to $b$. Thus, let $d \colon H \to \{n, \dots, 2n-1\}$ be an $n$-colouring of $H$ such that $d(v_{s+1}) \ne n$. To guarantee that (2) is satisfied at step $s+2$, consider the set $S=\{v \in V \mid c(v)=n \land \exists u \in b(v)\,(u \notin X \cup H)\}$. Extend $c_s$ to $c_{s+1}\colon X \cup (H\setminus S)\to 2n-1$ using $d$ to colour the vertices in $H\setminus S$. It is easy to see that at step $s+2$ the conditions are still met.
\end{proof}

\subsubsection{Hall's theorem and  bipartite graphs}
A graph $(V,E)$ is \emph{bipartite} if there are two totally disconnected subsets $A,B\subseteq V$ such that $A\cup B=V$ and $A\cap B=\emptyset$. We represent $(V,E)$ directly as $(A,B,E)$. A bipartite graph $(A,B,E)$ satisfies \emph{Hall's condition} if for every finite $X \subseteq  A$, $|N(X)| \geq  |X|$. Hall's theorem guarantees that for any bipartite graph $(A,B,E)$ there exists an injective function $f \colon A \to B$ such that $\forall a\in A \,(a \, E\, f (a))$ if and only if Hall's condition is satisfied.

Hirst in \cite[Chapter 2]{hirstPhD} studied the strength of Hall's theorem, proving that in its full generality the theorem is not computably true. Nonetheless, it is possible to weaken the premise and formulate  versions of Hall's theorem which are computably true. Hirst himself restricted the instances to bipartite graphs $(A,B,E)$ with $A$ finite (see  \cite[Theorem 2.1]{hirstPhD}).  Gasarch in \cite[Theorem 5.19]{Gasarch} considered  bipartite graphs $(A,B,E)$ with $A$ possibly infinite, but strengthened Hall's condition itself and formulated the extended Hall's condition. A bipartite graph $(A,B,E)$ satisfies the \emph{extended Hall's condition} if there exists a function $h\colon \Nb \to \Nb$ such that $h(0) = 0$  and, for every finite $X \subseteq  A$,
\[
|X| \geq h(n) \Imp |N(X)| \setminus |X| \geq n.
\]
Notice that the former requirement on $H$ implies that if $(A,B,E)$ satisfies the extended Hall's condition, then it also satisfies Hall's condition.

We prove that both of the proposed weaker versions are primitively recursively true. We give a new simpler proof of  Hirst's weakening of Hall's theorem, which does not use $\iso$.

\begin{proposition}
 $\Pra$ proves the following:
\begin{enumerate}
  \item Hall's theorem for graphs $(A,B,E)$ with $A$ finite.
  \item If $(A, B, E)$ is a honest bipartite graph that satisfies extended Hall's condition, then there exists an  injective function $f \colon A \to B$ such that $\forall a\in A \,(a \, E\, f (a))$.
\end{enumerate}
\end{proposition}
\begin{proof}
$(1)$  Recall that \lq $A$ is finite\rq\ means that there exists a code for $A$, and notice that any subset of $A$ is itself coded by some code less or equal to (the code for) $A$. Here we identify the set with its code, for simplicity of notation.  Recall that Hall's condition guarantees that for each $X\subseteq A$  there are $|X|$ vertices in $B$ adjacent to vertices in $X$.  Thus,  the following is true:
 \[
 \forall X \le A\, \exists y \, (y=\la b_0, \dots, b_{|X|-1} \ra \land \forall j < |X| \, (b_j \in N(X)) ).
 \]
 Note that the formula in the parentheses is $\Delta^0_0$. Hence, by $\bs 0 0$, there exists a uniform bound $c$ such that
  \[\forall X \le A\, \exists y < c \, (y=\la b_0, \dots, b_{|X|-1} \ra \land \forall j < |X| \, (b_j \in N(X)) ).\]
 Note that $(A, [0,c], E)$ is finite and still satisfies Hall's condition. Thus, by finite Hall's theorem, there exists a solution for it, which is clearly a solution for $(A,B,E)$. \smallskip

 $(2)$ 
  Assume that $(A,B,E)$ satisfies the hypothesis of  the statement. Let  $h\colon \Nb \to \Nb$ be a function witnessing that $(A,B,E)$ satisfies extended Hall's condition, and let $a \in A$. The solution $f$ is defined by (primitive) recursion, so that if at `step' $s$ $$f_s \colon \{a_0, \dots, a_{s-1}\} \to \{b_0, \dots, b_{s-1}\}$$ is defined and $(A\setminus \{a_0, \dots, a_{s-1}\} , B\setminus \{b_0, \dots, b_{s-1}\},E )$ with $h'\colon \Nb \to \Nb$ satisfy the hypothesis of the statement, then at the next step the elements $a_s \in A \setminus \{a_0, \dots, a_{s-1}\} $ and $f(a_s) \in B\setminus \{b_0, \dots, b_{s-1}\}$ are picked, and a function $h''\colon \Nb \to \Nb$ is defined so that $(A\setminus \{a_0, \dots, a_{s-1}\} , B\setminus \{b_0, \dots, b_{s-1}\},E )$  with $h''$ satisfy the hypothesis of the statement.

In order to do so, pick  $a_s \in A\setminus \{a_0, \dots, a_{s-1}\}$ and by $\dcao$ define $S$, the subset of $A\setminus \{a_0, \dots, a_{s-1}\}$ containing vertices which are connected to $a_s$ by a path of length at most $2h(s+1)$, as follows:
\[
	\{x  \mid \exists i \le 2h(s+1) \, \exists x_0, x_1, \dots, x_i\, (x_0=x \land x_i=a_s \land \forall j < i \, (x_j \in b(x_{j+1}))) \},
\]
where $x_j \in b(x_{j+1})$ expresses the fact that $x_j$ belongs to the string coded by $b(x_{j+1})$. Let $T$ be the set of neighbours of $S$, which can still be defined by $\dcao$ thanks to the function $b$. The graph $(S,T,E)$ is bipartite and satisfies Hall's condition. Thus, by the finite Hall's theorem, there exists an injective function $g \colon S \to T$ such that $\forall a\in S \,(a \, E\, g(a))$. Let $f(a_s)=g(a_s)$ and $h''(0)=0$, $h''(n)=h'(n+1)$ for all $n\geq 1$. One now needs to check that  $(A\setminus \{a_0, \dots, a_{s-1}\} , B\setminus \{b_0, \dots, b_{s-1}\},E )$  with $h''$ satisfies the hypothesis of the statement. This is done essentially exactly as in the proof of \cite[Theorem 5.19]{Gasarch}.
\end{proof}

\subsubsection{Connected components of a graph}
Let $(V,E)$ be an undirected graph. Then $C\subseteq V$ is a \emph{connected component} of $V$ if $C$ is a maximal set such that  any pair of vertices in  $C$ is connected by a path. Gura, Hirst, and Mummert~\cite{GuraHirstMummert} studied the strength of the principle stating the existence of a connected component of a graph. In the same paper a modification of the statement is proposed, so to let it be provable in $\rca$. We prove that, over $\Pra$, the proposed modification is equivalent to  $2^{\mathbb{N}}$-$\rca =\dca \land \iso$.  

    \begin{proposition}
Over $\Pra$, the following are equivalent:
\begin{enumerate}
    \item $2^{\mathbb{N}}$-$\rca = \dca \land \iso$.
    \item  Let $(V,E)$ be a graph and $\{v_0, \dots, v_n\}\subseteq V$ be such that each $v\in V$ is connected to at least one of  $v_0, \dots, v_n$. Then the connected components of $V$ exist\footnote{More specifically, there exists a function $f\colon V\times V \to 2$ such that if $v,u \in V$ are connected, then $f(v,u)=1$, and outputs $0$ otherwise.}.
\end{enumerate}
    \end{proposition}
\begin{proof}
($1 \Imp 2$) By $\iso$ we claim that there exists a subset $S$ of $\{v_0, \dots, v_n\}$ which is maximal totally disconnected, that is a maximal set such that any vertex in $V$ is connected with exactly one vertex in $S$. To see this, consider an enumeration $S_0, \dots, S_{2^n-1}$ of all  subsets of $\{v_0, \dots, v_n\}$, such that if $S_i\subseteq S_j$, then $j \le i$. Consider the $\Pi^0_1$-formula $\varphi(k)$ stating that there is no path between elements of $S_k$, if not the trivial one from a vertex to itself. Since $\varphi(2^n-1)$ holds, because $S_{2^n-1}$ is a singleton, by $\lnp 0 1$ there exists the minimal $k$ such that $\varphi_k$ holds. Note that $(S_k,E)$ is maximal totally disconnected, by the choice of $k$. Define, by $\dca$,  a function $g \colon V\times S_k \to 2$ such that
\begin{align*}
  g(v,s_i)=1 & \Biimp \text{there exists a path from $v$ to $s_i$}\\
             & \Biimp \text{there is no path from $v$ to $s_j$, for any $j \ne i$.}
\end{align*}
Finally, let $f \colon V\times V \to 2$ be such that
\[f(v,u)=1 \Leftrightarrow \exists x\in S_k \, (g(v,x)=1=g(u,x) )\]
It is easy to see that $f$ is the desired function.\smallskip

$(2\Imp 1)$ We first prove $\dca$. Consider two $\Delta^0_0$-formulae $\theta, \eta$ such that  $\forall n \,(\exists s\,\theta(n,s) \biimp \forall s \neg\eta(n,s))$. Let $V=\{a,b\} \cup \{x_{n,s}\mid n, s \in \Nb\}$, and $E\subseteq V\times V$ satisfying the followings
\begin{enumerate}
    \item $x_{n,s} E x_{m,t}$ if and only if $n=m$,
    \item $x_{n,s} E a$ if and only if $\theta(n,s)$,
    \item $x_{n,s} E b$ if and only if $\eta(n,s)$.
\end{enumerate}
It is easy to see that, for each $n,t \in \Nb$, $x_{n,t}$ is connected with precisely one node from $\{a,b\}$. In fact, for each $n$ there exists an $s$ such that either $\theta(n,s)$ or $\eta(n,s)$. If the former is the case, then $a E x_{n,s} E x_{n,t}$ witnesses that  $x_{n,t}$ is connected with $a$; otherwise, $x_{n,t}$  is connected with $b$. Thus,  $(V,E)$ and the set $\{a,b\}$ satisfy the hypothesis of the statement. Let $f\colon V \times V\to 2$ be a solution. Define $g \colon \Nb \to 2$ be such that
\begin{align*}
 g(n)=1& \Biimp f(x_{n,n},a)=1 \Biimp \exists s\, \theta(n,s);\\
 g(n)=0& \Biimp   f(x_{n,n},b)=1 \Biimp \exists s\, \eta(n,s).
\end{align*}

We now prove $\lnp 0 1$. Let $\theta$ be a $\Delta^0_0$-formula and $n \in \Nb$ such that $\forall s\, \theta(n,s)$. We find the least  $m \in \Nb$ such that $\forall s\, \theta(m,s)$. Let $V=\{v_0, \dots, v_n,a\}\cup \{x_s \mid s \in \Nb\}$. Define $E \subseteq V \times V$  as follows
\begin{enumerate}
     \item $x_s E a$ for each $s$,
     \item $x_s E v_i$ if and only if $\neg \theta(i,s)$.
 \end{enumerate}
It is immediate to check that $(V,E)$ and   $\{v_0, \dots, v_n,a\}$ satisfy the hypothesis of the statement, so let $f \colon V \times V \to 2$ be a solution. Consider the sequence $\la f(v_0,a), \dots, f(v_n,a) \ra$, and search for the smallest $m\le n$ such that $f(v_m,a)=0$. Since by construction, $v_m$ is connected with $a$ if and only if $\exists s \, \neg\theta(m,s)$, such $m$ is the smallest such that $\forall s\, \theta(m,s)$.
\end{proof}

\smallskip

\subsection{Models and algebraically closed fields}\label{sebs:models}
The main purpose of this subsection is to verify that $\Pra$ proves that every countable field can be embedded into its algebraic closure, and similarly for ordered fields and their real closures.

\

Having in mind `decidable' algebraic structures in elementary computable model theory, we
shall need a more expressive way to code  an algebraic structure suitable for developing basic constructions (such as Henkin's) in $\Pra$. This is also consistent with the approach in Simpson~\cite{sosoa}.

We define a countable \emph{model} $M$ in a given finite signature, as follows.

\begin{itemize}
\item $M$ is represented as a structure (\S\,\ref{sub:str}).
\item  There is a $\{0,1\}$-valued function deciding the truth of  first-order facts about $M$ (perhaps, with parameters in $M$).

\end{itemize}
  We  identify $M$ with the function evaluating the truth of first-order statements in $M$, but we keep in mind that $M$ also has to have its domain an initial segment of $\mathbb{N}$ (to make the search for its $k$th element bounded).
  If a first-order formula $\sigma$ contains free variables, then (by definition) we set $M(\sigma) =1$ if
$M(\sigma') =1$, where $\sigma'$ is the universal closure (the generalisation) of $\sigma$.

   \begin{remark}
  We can additionally require that there is a function which, whenever $M(\sigma) =1$ for an existential $\sigma$, on input $\sigma$ returns an existential witness $x \in M$. The proofs that we give in this subsection would still work for this stronger notion.
  This assumption would not make any difference in $\rca$ or a stronger system, but in $\Pra$ it does. If we choose to additionally require that $M$ comes together with a Skolem function, we can drop the restriction on the domain to be (an initial segment of) $\mathbb{N}$
  since the search for the next element in the structure becomes bounded.
 \end{remark}

Every model is an algebraic structure.
It is not difficult to see that the notions of a `model' and an `algebraic structure' differ already in the minimal model of $\Pra$; we cite \cite{Kal17} for several results that imply this fact.

\

We fix calculus of first-order formulae (coded in $\Pra$).
We can assume that our formal proof system uses only modus ponens (see \cite[Section 2.4]{enderton}), so any initial segment of the proof is also a proof. 
A \emph{theory} is a set of sentences, represented in $\Pra$ through its characteristic function, closed under logical consequence. In particular, it includes all basic axioms
of our proof system.
Then $M$ is a model of $T$, $M \models T$, if $M(\sigma)=1$ whenever $  \sigma \in T$.


\begin{proposition}\label{pr:m1}
 $\Pra$ proves that a complete consistent theory  has a countable  model.
\end{proposition}
\begin{proof}
This is essentially \cite[Proposition 2.7]{SurveyOnline2019}.
One also needs to recall that among the primitive recursive enumeration of sentences in the expanded language, there are also formulae
$(\exists x) \bigwedge_{i <m} (x \neq c_i)$ for each $m \in \Nb$.
Since the Henkin's proof is primitive recursive, and such formula is considered at a stage $g(m)$ of the proof, where $g$ is also primitive recursive in $m$, this allows to produce a primitive recursive
enumeration of representatives of the quotient classes, without repetition.
\end{proof}

\begin{proposition}\label{pr:m2}
 $\Pra$ proves that if $T$ has a model, then $T$ is consistent.
\end{proposition}

\begin{proof}
The argument that can be found in \cite{sosoa} would not work since we cannot use $\Sigma^0_1$-in\-duc\-ti\-on and neither can we use  recursive comprehension. We need to be a bit more careful.

Suppose $T$ is not consistent and $M \models T$.
Fix $\sigma_0, \ldots, \sigma_k \in T$ such that $T$ proves $\neg \bigwedge_{i \leq k} \sigma_i$, and
let $p$ be a proof.
 Let $S$ be the collection of all formulae that are mentioned in $p$.
Let $\varphi(\sigma, k)$ say that if there is a subproof (of $p$) of $\sigma$ of length $k$ then $M \models \sigma$, and consider $$\psi(k) = (\forall \sigma \in S) \varphi(\sigma, k)$$
which is a bounded formula. (Here  $M$ is a parameter in the formula, and we do allow parameters in our induction scheme.)

We have that $\psi(0)$ holds because $M \models T$.
Since the only rule of inference is modus ponens,
we have that  $\forall k (\psi(k) \rightarrow \psi( k+1))$ since any instance $\sigma$
of $\varphi(\sigma, k+1)$ in $S$ is either an axiom in $T$ or is obtained
using modus ponens from an instance $\sigma' \in S$ having proof of length $k$.
By the principle of bounded induction, we arrive at $M(\neg \bigwedge_{i\leq k} \sigma_i) = 1$
which contradicts the assumption that  $M(\neg \sigma_i) = 0$, for $i = 0, \ldots, k$.
\end{proof}

\begin{definition} The \emph{algebraic closure} of a field $F$ is an algebraically closed field $X$
(more formally, $X \models ACF$)
together with an embedding $f: F \rightarrow X$ such that for every $x \in X$ there is a $\bar{z} \in F$ such that $x$ is algebraic over $\bar{z}$.
The real closure of an ordered field is defined similarly using $RCF$.

\end{definition}


\begin{remark}
We have to be a bit careful in our definition of the algebraic closure since the standard  textbook proof of transitivity of  `being algebraic over' seemingly relies on unbounded search.
Although there are other arguments that involve elementary matrix analysis,  we  chose to use the model-theoretic version, i.e., $x \in cl_{alg}{Y}$ if $x$ is first-order definable over $Y$, which is clearly transitive. Note that in both cases ($RCF$ and $ACF$) we have quantifier elimination  so this is (classically) equivalent to the algebraic definition, and we conjecture that this can be demonstrated in $\Pra$.
We feel that a more detailed analysis and comparison of the several potentially different ways of defining `algebraicity over', albeit perhaps interesting, is outside the scope of this article.
\end{remark}

\begin{theorem}[$\pra^2$] \mbox{}
\begin{enumerate}
\item Every field can be embedded into its algebraic closure.
\item Every ordered field can be embedded into its real closure.
\end{enumerate}

\end{theorem}

\begin{proof}[Extended sketch]
We follow the proof of Theorem 11.9.4 in \cite{sosoa} closely (which itself is based on folklore in computable model theory).
Recall that quantifier elimination in both ACF and RCF is a primitive recursive manipulation with formulae. Our proof relies on Propositions~\ref{pr:m1} and \ref{pr:m2}  instead of the analogous results in $\rca$.

Form $ACF \cup D_0(F)$, where $D_0(F)$ is the quantifier-free diagram of the field $F$. If $AF$ stands for the field axioms, then  $AF \cup D_0(F)$ has a model (being $F$),
and thus is consistent by Propositions~\ref{pr:m2}. Because of the quantifier elimination, this also implies that  $ACF \cup D_0(F)$ is consistent, and thus has a countable model by Proposition~\ref{pr:m1}.
Henkin's construction guarantees that the embedding of $F$ into its natural image in the resulting $M$ is primitive recursive.
It remains to set $U = cl_{alg}(F)$; the latter can be primitively recursively listed (because of the quantifier elimination). We can use primitive recursiveness of the image of $F$ and padding (a delay of computation which consists in repeating segments of a sequence, more detailed examples are given in the next two sections) to make sure that the domain of $U$ is equal to $\mathbb{N}$.

(2) is proved similarly.
\end{proof}

In computable mathematics,  computable model-theoretic results are the standard tools for establishing various existential closure results.
For example, Harrington~\cite{Harrington-74} uses computable prime models to derive that every computable differential field is
contained in its differential closure. Ershov in a series of works~\cite{Ershov-Dokl,Ershov-Skolem,Ershov-Decidability-Book} develops an effective model-theoretic machinery and a general notion of an effective closure. He applies it to show that every computable locally nilpotent torsion-free group can be embedded into its divisible closure. We conjecture that many results of this sort hold primitively recursively.
But to apply these results in $\Pra$ one needs to develop a sufficient amount of `soft' model theory in $\Pra$; this may prove to be a challenging task.
For instance, it seems that the aforementioned result of Harrington about prime models holds primitively recursively but requires too much induction (since it is a priority construction).


\begin{question}\label{q:2}
Study model theory over $\Pra$.
\end{question}

\noindent The reverse mathematics of model theory has been studied in, e.g., \cite{hirschfeldt2009atomic,Belanger-14,Belanger-15,Induction_Bounding}.


\section{Baire category theorem and Ramsey theorem}\label{sect:Baire_and_Ramsey}

Recall that we fixed a primitive recursive coding of finite strings in $\mathbb{N}^{<\mathbb{N}}$. Elements of $\mathbb{N}^{<\mathbb{N}}$ can be identified with total functions. We could follow the basic ideas from \cite{sosoa} and formalise Polish metric spaces in $\Pra$,
but for now we restrict ourselves to the space $\mathbb{N}^{\mathbb{N}}$.

An open set $V$ in $\mathbb{N}^{\mathbb{N}}$ is coded by the sequence of basic open sets that together make up $V$; each basic open set is identified with the respective finite string.
An open set $V$ is said to be \emph{dense} if for every finite string $\sigma$ there is a string $\tau$ extending $\sigma$ and a basic open subset $B$ of $V$ such that $\tau$ lies in $B$.
These definitions can be formalised in $\mathcal{L}_{\Pra}$.  In particular, an open set is identified with a function that lists basic open sets (coded as finite strings) that together make up the open set.
A sequence of open sets is coded using a primitive recursive function in two arguments.

This choice of coding can be criticised. One could argue that a more `honest' coding should involve characteristic functions rather than enumeration. However, as further explained in Remark~\ref{codingremark}, Theorem~\ref{Bairethm} would still remain true under this seemingly more expressive coding.

\subsection{Baire category theorem}
The following can be made sense of in $\mathcal{L}_{\Pra}$.

 \begin{definition}[$\mathsf{Baire Category Theorem}$]
An instance of $\mathsf{Baire Category Theorem}$
is a sequence  $(V_n)_{n \in \mathbb{N}}$
of dense open sets in $\mathbb{N}^{\mathbb{N}}$. Then a function $h\in \mathbb{N}^{\mathbb{N}}$ is a solution
if $h\in \cap_n V_n $.
\end{definition}

\begin{theorem}\label{Bairethm}
 $\Pra +\mathsf{Baire Category Theorem}$
 lies strictly between $\Pra$ and $\rca$.

\end{theorem}


\begin{proof}
It is easy to see that $\mathsf{Baire Category Theorem}$ can be demonstrated in $\rca$; see \cite{sosoa}.
We demonstrate that there is an instance of $\mathsf{Baire Category Theorem}$ in the standard minimal model of $\Pra$ which does not have a primitive recursive solution, and thus $\Pra$ does not prove $\mathsf{Baire Category Theorem}$:

\begin{lemma}\label{lem:easybaire}
There is a uniformly primitive recursive sequence of dense open sets $(V_n)$ in $\omega^{\omega}$ such that there is no primitive recursive point (path, function) in their intersection.
\end{lemma}

\begin{proof} Let $(f_e)_{e\in\omega}$ be a uniformly computable enumeration of all primitive recursive functions.

The idea is to describe a primitive recursive procedure of  simultaneous enumeration of open sets $(V_n)_{n\in\omega}$ such that  either $V_{2n}  = \omega^\omega \setminus \{ f_n\}$ and $V_{2n+1} =  \omega^\omega$, or $V_{2n}  = \omega^\omega$ and $V_{2n+1} =   \omega^\omega \setminus \{ f_n\}$.

For a $\sigma \in \omega^{<\omega}$,  let $B_\sigma = \{\rho \in \omega^{<\omega}: \sigma\preceq \rho\}$.
Regardless of the outcome, we will eventually put the basic open set $B_0$ into $V_{2n}$ and
 the basic open set $B_1$ into $V_{2n+1}$. This will not be done immediately though.
 We initiate a primitive recursive enumeration of $$B_{00}, B_{01},B_{02}, \ldots $$ into $V_{2n}$,
 and similarly we initiate a primitive recursive enumeration of  $$B_{10}, B_{11},B_{12}, \ldots $$ into $V_{2n+1}$.

We wait for $f_n(0)$ to converge. Without loss of generality, we can assume $f_n(0) \neq 0$;
the case when $f_n(0) \neq 1$ is symmetric.
Since $f_n(0) = m_0 \neq 0$ we can use $V_{2n}$ to diagonalise, since no part of $B_{m_0}$ has yet been listed in $V_{2n}$. (In this case we proceed to list all basic balls into $V_{2n+1}$ thus making it equal to the whole space.)

We wait for $f_n(1)$ to halt. Meanwhile,  we keep enumerating the sequence $B_{00}, B_{01},B_{02}, \ldots $ into $V_{2n}$.
If $f_n(1) = m_1$, then we initiate the enumeration of $B_{m_0} \setminus B_{m_0m_1}$ into $V_{2n}$.

We iterate the procedure. Eventually, we will start listing elements in $B_{m_0 m_1} \setminus B_{m_0m_1 m_2}$, and the same for $m_3$, etc. We keep enumerating $B_{m_0} \setminus B_{m_0m_1}$ into $V_{2n}$ while we wait for $f_n$ to halt on one more argument.
This way we produce a primitive recursive procedure that lists an open set equal to $\omega^{\omega} \setminus \{f_n\}$.

It remains to argue that these open sets can be build uniformly primitively recursively.
We can fix a primitive recursive function $U$ such that $f_{e,s}(x) = U(e,x,s)$.
(Note that the computable function $s_e(x) = \mu_t  \, f_{e, t} (x)$$\downarrow$ is not primitive recursive.) We use $U$ to run the construction of $(V_n)_{n\in\omega}$ simultaneously,
so that in each $V_n$ we either use  $B_0$ or $B_1$ until $s_n$ halts on one more input.

It is clear that each $V_n$ is dense, but for any $n$, $f_n \notin \bigcap_n V_n$.
\end{proof}

To show that $\rca$ is strictly above $\mathsf{Baire Category Theorem}$, we build a (standard) model of $\Pra+ \mathsf{Baire Category Theorem}$ such that the model does not contain all computable functions.
 We work  with the standard natural numbers $\omega$.
All functions in this proof are total. We identify functions with paths through $\omega ^{ \omega}.$


Fix a computable function $f\in \omega^\omega$
that is not dominated by any primitive recursive function.
We iteratively apply Lemma \ref{lem0} below to the minimal model of $\Pra$ to build a model
$\mathcal{M}\subseteq \omega^\omega$
of $\mathsf{Baire Category Theorem}$
so that $f$ is not dominated by any function in $\mathcal{M}$.
\begin{lemma}\label{lem0}
Let
$\mathcal{K}\subseteq \omega^\omega$ be a countable PR-closed class.
Suppose $f$ is not dominated by any function in $\mathcal{K}$.
Then, for every sequence  $(V_n)_{n \in \omega}$ of  dense open sets in $\omega^\omega$,
there is a function $h\in \cap_n V_n$ such that $f$ is not dominated by any
function in
PR$(\mathcal{K}\cup\{h\})$.

\end{lemma}
\begin{proof}

For each $g\in \mathcal{K}$ and each primitive recursive
functional $\Psi$, we need to meet the following requirements:
$$
\mathcal{R}_{\Psi,g}: \ \Psi^{g\oplus h} \text{ does not dominate }f;
$$
and also for every $n \in \omega$,
$$
\mathcal{D}_n: \ h\in V_n.
$$

Suppose we have already defined a finite initial segment $\rho\prec h$.
To meet $\mathcal{D}_n$, use that $V_n$ is dense, combined with the claim below that allows to meet  $\mathcal{R}_{\Psi,g}$.
\begin{claim}
For any $\rho\in \omega^{<\omega}$,
there is a $\sigma\succ\rho$ such that for some $m$,
$\Psi^{g\oplus\sigma}(m)<f(m)$.
\end{claim}
\begin{proof}
Since $\Psi$
is a primitive recursive functional,  the function
$w= \Psi^{g\oplus (\rho^\smallfrown 0^\omega)}$ is total and is indeed primitive recursive in $g$. In other words,
$$ w \in \mathcal{K}.$$
So, by the assumption about $\mathcal{K}$ there must be   $m,j$ such that
$\Psi^{g\oplus (\rho^\smallfrown 0^j)}(m)<f(m)$.\end{proof}
Using the claim and by the use principle, we can always extend a given string $\rho$ to a string $\rho' = \rho^\smallfrown 0^j$ such that any extension of this string meets $\mathcal{R}_{\Psi, g}$.
We then further extend the string $\rho'$ to meet the open set $V_{n+1}$, and so on.
\end{proof}
This finishes the proof of Theorem~\ref{Bairethm}.
\end{proof}

\begin{remark}\label{codingremark}
We represented an open set via a function that lists its basic open subsets.
We could instead use the characteristic function, i.e., a function $g$ so that $g(\sigma) =1$ iff $B_\sigma \in V$, and $g(\sigma) =0$ otherwise. It is however not hard to see that Lemma~\ref{lem:easybaire} would still hold true under this new coding.
More specifically, suppose we have to quickly decide whether $B_\sigma$ is in $V$
but we are not yet sure whether we can say ``yes'' (because $f_n$ has not yet halted on enough inputs, so it is still possible that $f_n \in B_\sigma$). In this case
we can always quickly declare that $B_\sigma$  is not in $V_n$. We can later list $B_\sigma \setminus \xi$ into $V_n$, where $\xi$ is either $f_n$ or $\sigma 0^\omega$ in case if $f_n \notin B_\sigma$.

This adds quite a bit of extra noise to the construction of Lemma~\ref{lem:easybaire}.
One needs to argue that it can be arranged so that for each basic open $B_\sigma \not\ni f_n$, the intersection $B_\sigma \cap V_n$
is missing at most finitely many points of $B_\sigma$. (This property implies that $V_n$ is dense.)
Indeed, if $B_\sigma \not \ni f_n$, then eventually this will be recognised, and from this stage on we can stop extracting points from $B_\sigma$ for the sake of producing a rapid definition of $V_n$.

Alternatively, we could assume that the enumeration of $V_n$ has no repetitions---this would also have no effect on the results, with just a bit of extra care.
In other words, the results of this subsection are essentially independent of the specific (natural)  choice of coding.
 \end{remark}


\subsection{The stronger result} In fact, using similar techniques exploiting the speed of growth, we can establish, perhaps, a more unexpected fact:

\begin{theorem} $\mathsf{Baire Category Theorem}$ neither implies nor is implied by  $\dca$ (or $2^{\mathbb{N}}$-$\rca$) over $\Pra$.

\end{theorem}

Indeed, the theorem holds even when restricted to standard models (in particular, with full induction). It is also clear that combined with~Corollary~\ref{cor:weaker} the theorem implies
the less elementary half of Theorem~\ref{Bairethm}. However, the proof below relies on
Lemma~\ref{lemma:abitmore} which is established using a generalisation of the much more transparent argument in Lemma~\ref{lem:easybaire} used in the proof of Theorem~\ref{Bairethm}. Thus,  we decided to keep both proofs.
\begin{proof}

Recall that in the proof of Proposition~\ref{pr:rca}, we argued that
 the minimal model of $\dca$ consists of primitively recursively bounded computable functions.
Thus, to construct a standard model of $\dca$ in which $\mathsf{Baire Category Theorem}$ fails, we need to push the proof of Lemma~\ref{lem:easybaire} and diagonalise against
all solutions that are primitively recursively bounded rather than just against all primitive recursive solutions. (Since we are working in standard models, the $2^{\mathbb{N}}$-$\rca$ part of the theorem will also follow.)

\begin{lemma}\label{lemma:abitmore}
There is a primitive recursive instance of $\mathsf{Baire Category Theorem}$ that has no primitively recursively bounded  solutions.
\end{lemma}

\begin{proof} For a fixed primitive recursive bound $h$, let $h^{\omega}$ denote the $h$-branching homeomorphic copy of the Cantor space $2^{\omega}$. Note that  $\omega^{\omega} \setminus h^{\omega}$ is open and  dense in $\omega ^{\omega}$.
So the idea is to
build a primitive recursive instance of $\mathsf{Baire Category Theorem}$ in which $V_n = \omega^{\omega} \setminus f_n^{\omega}$, where $f_n$ is the $n$th primitive recursive function.  If we succeed, then evidently no $h$ bounded by any $f_n$ can possibly be a solution to this instance, and thus the minimal model of $\dca$ would fail to contain any solution to this instance.

The usual issue is that, of course, there is no uniformly primitive recursive enumeration of $f_n$. Thus, we have to deal with a primitive recursive simultaneous approximation to $f_n$, $n \in \omega$, using a primitive recursive function in two arguments.

The idea is to delay the enumeration of a basic open set $B_\sigma$ into $V_n$, as follows.
If we are not yet sure whether we should put $B_\sigma$ in or not, declare it out.
We then wait for $f_n$ to converge on sufficiently many inputs to decide whether we should actually have listed $B_\sigma$ in $V_n$ if we had quick access to $f_n$. If this is indeed the case, we can always later initiate a primitive recursive enumeration of a sequence of basic clopen subsets of $B_\sigma$ that together make up $B_\sigma$.

However, this might lead to the issue of totality of enumeration in the sense that each $V_n$
has to be enumerated by a primitive recursive procedure, so we have to put at least one basic clopen set into each $V_n$.

We resolve this as follows.   Instead of one dense open  $V_n$ for each $n$, define a sequence $V_{n,k}$, $k \in \omega$. (This is similar to how we had two open sets in the proof of Lemma~\ref{lem:easybaire}.) More specifically, we always put $B_{k }$ into $V_{n,k}$ initially. We then wait for $f_n(0)$ to halt. If $f_n(0) \geq k$ then we can proceed to enumerating the entirety of $\omega^\omega$ into $V_{n,k}$. Otherwise, for each $k> f_n(0)$, we can proceed
with the strategy of local delay (as described above) and build $V_{n,k} = \omega^{\omega} \setminus f_n^{\omega}$. We omit further details which we believe are sufficiently elementary. (Remark~\ref{codingremark} applies to this argument as well.)
\end{proof}

We conclude that $\dca$ does not imply $\mathsf{Baire Category Theorem}$.

\

To establish that $\mathsf{Baire Category Theorem}$ does not imply $\dca$ over $\Pra$, we shall construct a standard model of $\Pra + \mathsf{Baire Category Theorem}$ that
does not include some $\{0,1\}$-valued computable function; the latter can be picked to be an arbitrary total computable characteristic function that is not primitive recursive.

Recall that constructing a standard model of  $\Pra$ is the same as defining a collection of
total functions closed under primitive recursive operators, i.e., PR-closed.
Note also that for any instance $(V_n)_{n \in \omega}$ of $\mathsf{Baire Category Theorem}$ and any monotonically increasing
total function $\ell$, $(V_n)_{n \in \omega}$ has a solution $h$ with the property:
$$
	\,\,\text{For infinitely many} \,\, i,  \,\,  h\upharpoonright_{[\ell(i), \ell(i+1)-1]} \mbox{ is a constant function}.
$$

The property holds true since $V_{n}$ is dense in $\omega^{\omega}$, and therefore
we can delay the correction of $h$ for as long as we desire
and still hit $V_{n}$.  In the property, $\ell$ simply stands for the delay that we choose.
We thus call the property described above the \emph{local delay property}, or $\mathsf{LD}$-property for short.
(It is essentially a property that allows one to use `padding'.)


\

We shall construct our model by iteratively applying the lemma below.


\begin{lemma}\label{lem:halflemma} Fix a  non-primitive recursive function $g\in\omega^\omega$,
a primitive recursive operator $P$, $\sigma\in \omega^{<\omega}$, and $m\in\omega$. 
There exist $n,i\in\omega$ such that
$g(i)\ne P^{\sigma^\smallfrown m^n}(i)\downarrow$
(where $m^n$ denotes the  string of the form $m\cdots m$ having length $n$).
\end{lemma}

\begin{proof} This is simply because
$P^{\sigma^\smallfrown m^\omega}$ is a (total) primitive recursive function.
Thus, take $i\in\omega$ so that $g(i)\ne P^{\sigma^\smallfrown m^\omega}(i)\downarrow$
and take $n$ so that the use of $P^{\sigma^\smallfrown m^\omega}(i)$
is $\sigma\cdot m^n$.
\end{proof}

In particular, strings $\sigma_1, \ldots, \sigma_k$ could be taken as initial segments of
solutions to $k$ instances of \textsf{Baire\-Ca\-te\-goryTheorem} that we have built so far.
We can assume they have equal lengths, say length $s$.
We then can use the lemma (with $\sigma = \sigma_1\oplus \dots\oplus \sigma_k$) to define $\ell$ on one more argument and successfully diagonalise against the operator $P$. We then use $\mathsf{LD}$-property
and the use principle to conclude that any extension of the strings that we now have will be diagonalising against the operator $P$.

We thus can iterate the lemma to build a $PR$-closed family $K$ of total functions that
includes at least one solution for each instance of $\mathsf{Baire Category Theorem}$
while simultaneously meeting the requirements:
$$g \neq P^{f_1, \ldots, f_k} $$
 for all $f_1, \ldots, f_k \in K$ and each primitive recursive scheme $P$ with $k$ parameters.
 We do so by simultaneously defining the $f_i$ and also building the common `local delay' function  $\ell$; see the $\mathsf{LD}$-property. We omit the further elementary (but somewhat tedious) details.
\end{proof}





  \begin{question}\label{q:3}
 Over $\Pra$, does $\mathsf{Baire Category Theorem}$ imply $\iso$ or $\ido$?
  \end{question}
In \cite{simpson2014baire} it is proved that, over $\rca^*$, $\mathsf{Baire Category Theorem}$ implies $\iso$, or in other words that,  over $\rca^*$, $\mathsf{Baire Category Theorem}$ is equivalent to $\rca$.

\subsection{Ramsey Theorem}

Recall that $\mathsf{RT}_k^n$ abbreviates that any $k$-colouring of
$\mathbb{N}^n$ admits a homogeneous set.   For more background on Ramsey Theorem in reverse mathematics, see \cite{Hirschfeldt15}.
It should be clear that  $\rt n k$ can be formalised in $\Pra$. However, there are two natural ways to code a solution to an instance of  $\mathsf{RT}_k^n$.
One possibility is to represent a solution via an injective function that lists the solution; we will return to this approach in the next section.
In this subsection we focus on the coding that views a solution of  $\mathsf{RT}_k^n$ as a set that is identified with its characteristic function.

\begin{theorem} \label{The:RTDelta}
Over $\pra^2$,  $\rt n k$ is incomparable with
$\dca$, for any $n, k \geq 2$. (This holds already for standard models.)
\end{theorem}\label{thm:rt22}
\begin{proof}
Fix $n, k \geq 2$.
We first argue that $\mathsf{RT}_k^n$ does not imply $\dca$.
 If  $X$ is an instance of $\rt n k$ and $Y$ is a solution to $X$, then
any infinite subset $\hat{Y}$ of $Y$ is also a solution to $X$. In particular, we can keep arbitrarily long segments of the characteristic function equal to zero. This is very similar to the local delay property that was used in the proof of the previous theorem, the only difference being that the delaying interval could be longer than $[\ell(i), \ell(i+1)-1]$ because there is no guarantee that $Y(\ell(i+1)) =1.$ However, Lemma~\ref{lem:halflemma} still applies, and the argument that follows the lemma can be slightly adjusted to work for a longer delay.  We omit further details.
\smallskip

To see that $\dca$ does not imply $\rt n k$  for $n,k \geq 2$, recall that
 a computable instance of $\rt n k$, $n,k \geq 2$,  does not have to possess a computable solution \cite{Specker,JockuschRamsey}, and thus there exists a model $(M, \mc{X}) \vDash \rca \land \neg\rt n k$. Let $(M,\mc{Y})$ be the functional version of $(M,\mc{X})$,  that is let $\mc{Y}$ be composed by the characteristic functions of sets in $X$. Then  $(M, \mc{Y}) \vDash \rca$, and so in particular  $(M, \mc{Y}) \vDash \dca$, but fails to satisfy  $\rt n k$. 
\end{proof}


The previous theorem allows to immediately derive the following corollaries, which show that even above $\rca$ the behaviour of the principles in classical and primitive recursive reverse mathematics may be different. Recall that,  over $\rca$, $\rt n k$, for $n\ge 3$ and $k \ge 2$, is equivalent to  K{\"{o}}nig's lemma (see \cite[Theorem III.7.6]{sosoa}). Moreover, $\rt 1 {< \infty}$, namely the infinite pigeonhole principle for arbitrary numbers of colours, is strictly stronger than $\rca$ (see \cite[Theorem 6.4]{hirstPhD}).

\begin{corollary}
 Over $\Pra$, $\rt n k$, for $k,n\geq 2$, does not imply $\kl$, i.e., for each finitely branching tree  $T\subseteq \Nb^{<\Nb}$ there exists a path $P$ (that is, $P \colon \Nb \to \Nb$ such that $\forall n \, P(n)\in T$).
\end{corollary}
\begin{proof}
$\kl$ implies $\rca$ by \cref{Prop:noLeavesKL} below, and hence it implies $\dca$, which is not implied by $\rt n k$ by \cref{The:RTDelta}.
\end{proof}

As firstly noted in \cite{CholakJockuschSlaman} (stable) $\rt 2 2$ implies $\rt 1 {<\infty}$. In fact, if $c\colon \Nb \to k$, for some $k\in \Nb$, then one can $\Delta_0$ in $c$ define a colouring $d\colon [\Nb]^2 \to 2$ such that $d(x,y)=0 \Leftrightarrow c(x)\ne c(y)$ and  $d(x,y)=1 \Leftrightarrow c(x)= c(y)$, and note that any homogeneous set for $d$ is homogeneous for $c$. It is clear the the implication still holds in $\Pra$. Since $\rca$ implies $\dca$, while $\rt 2 2$ does not, we can derive the following corollary.

\begin{corollary}\label{Cor:RT-RCA}
Over $\Pra$, neither $\rt 1 {< \infty}$ nor $\rt n k$, for $k,n\geq 2$, imply $\rca$, and are thus incomparable with it.
\end{corollary}

On the other hand, it is easy to see that, for each $k \in \Nb$, $\Pra \vdash \rt 1 k$.

The following questions remain open.

\begin{question} \label{q:ramseyBaire}
Over $\Pra$, do $\rt 1 {< \infty}$ or $\rt n k$, for $k,n\geq 2$, imply $\mathsf{Baire Category Theorem}$?
\end{question}

Note that the reverse cannot hold since it does not hold over $\rca$.

\begin{question} \label{q:ramseyInd}
Over $\Pra$, do $\rt 1 {< \infty}$ or $\rt n k$, for $k,n\geq 2$, imply $\is 0 1$?
\end{question}

Note that, over $\rca^*$, $\rt n k$, for $k,n\geq 2$, do not imply $\is 0 1$ (see \cite[Corollary 3.7]{yokoyama2013}). On the other hand, over $\rca^*$, $\rt n k$, for $k,n\geq 2$, do imply $\is 0 1$, whenever the homogeneous set is required to have universe size (as opposed to be only unbounded) (see \cite{fkk:weak-cousins}).

\section{Transforming a computable instance to  a primitive recursive instance}\label{sect:transforming}

It seems to be a general phenomenon in computable algebra that many computable algebraic structures
are isomorphic, and indeed often computably isomorphic, to primitive recursive structures. Many results of this sort can be found in \cite{Grigorieff-90,Cenzer-Remmel-91,Kal17}.
It requires some effort to find an example of a computable structure without a primitive recursive or fully primitive recursive (`punctual') presentation; we cite \cite{Cenzer-Remmel-92,Cenzer-Remmel-98,Kal17}.
In this section we discuss a similar phenomenon that occurs in $\Pra$; we have already encountered it a few times (implicitly or explicitly) in the preceding sections.
Specifically, we have seen that typically some delaying  `padding' argument shows that for many
problems,
their primitive recursive instance can be as
powerful (with respect to coding) as their computable instance
in the sense that
\begin{align}\nonumber
&\text{for every computable $\mathsf{P}$-instance $X$, there is a primitive recursive $\mathsf{P}$-instance
$\hat X$ }\\ \nonumber
&\text{such that every solution of $\hat X$ computes a solution of $X$.}
\end{align}

Indeed, it is not uncommon that every solution of $\hat X$
can be turned into a solution of $X$ uniformly primitively recursively,
i.e., using a primitive recursive operator.
For some problems, every solution of $\hat X$ is a solution of $X$; and the aforementioned transformation operator is simply the identity operator.
 For such problems,
many results that are known over the base system $\rca$ can be transformed into $\Pra$ proofs with minimal effort.
We give more examples of such problems below.


\subsection{More notation}\label{sec:morenot} Before we state the next result, we clarify our notation and our approach to continuity in $\Pra$.
We represent rationals as pairs of integers.
We represent a real via a function $f: \mathbb{N} \rightarrow \mathbb{Q}$ such that $|f(i) - f(i+1)| < 2^{-i-1}$.
We represent a function $g: \mathbb{Q} \rightarrow \mathbb{R}$ via a function $g( r, m)$ such that for every $r \in \mathbb{Q}$ and $m\in\mathbb{N}$,
$g( r, m)$ is a rational.

 A function $h: \mathbb{R} \rightarrow \mathbb{R}$ is represented by a pair of functions $f$ and $
\delta$, where
$f: \mathbb{Q} \rightarrow \mathbb{R}$ is its `restriction to $\mathbb{Q}$' and
$\delta: \mathbb{Q} \times \mathbb{N} \rightarrow \mathbb{N}$ is the `attempted modulus of continuity':
$$  h( (r-2^{-\delta(r,n)},r+2^{-\delta(r,n)}))\subseteq
(  f(r)-2^{-n},   f(r)+2^{-n}).$$
\begin{remark}
Having such a presentation does not imply
  that $h$ can be continuously extended to $\mathbb{R}$, as some points could be `missing' in a model.
Given a presentation
$(h,\delta)$,
we can express ``$(h,\delta)$ is continuous"
as a  (second order) arithmetic sentence.
But
it only ensures that the function is continuous at every real $r$ (whose presentation is) in the model.
We will study this effect later in much detail.
\end{remark}

\begin{definition}\label{def:anprob}
We specify the following problems:
\begin{itemize}
\item
An instance of $\mathsf{IntermediateValueTheorem}$
is (a presentation of)
a function $X$ on $[0,1]$.  
A solution to $X$ is (a presentation of) a real $r$ so that
either $X$ is not continuous at $r$
or $  X(r) = 0$.

\item An instance of $\mathsf{Completeness\ of\ \mathbb{R}}$
is (a presentation of) a sequence $(\rho_n\in 2^n:n\in\mathbb N)$
of strings so that $\lim_{n\rightarrow\infty }\rho_n$ is Cauchy.
A solution to $(\rho_n\in 2^n:n\in\mathbb N)$
is  $\lim_{n\rightarrow\infty }\rho_n$.

\item An instance of $\mathsf{HeineBorelTheorem}$
is (the presentation) of a
sequence $(I_s:s\in\mathbb{N})$ of open intervals
 with rational end points (some of the intervals could be empty)
such that $\bigcup_{s<t}I_s \nsupseteq [0,1]$ for all $t\in\mathbb{N}$.
A solution is (the presentation) of a real
$x$ such that $x\notin \bigcup_s I_s$.

\end{itemize}

\end{definition}

\begin{remark} In the definition of $\mathsf{HeineBorelTheorem}$, the assumption that the end points are rational is a mere convenience. Indeed,
for any primitive recursive sequence
$(I_s:s\in\omega)$ of intervals (whose end points are not necessarily rational),
there is an instance $(\hat I_s:s\in\omega)$
primitive recursive in $(I_s:s\in\omega)$
with rational end points of $\hat I_s$ such that
$\bigcup_s \hat I_s= \bigcup_s I_s$.
\end{remark}


Finally, recall that  $\mathsf{COH}$ stands for the Cohesive Principle:  For any family $\{R_x: x \in \mathbb{N}\}$ of subsets of $\mathbb{N}$
there is an infinite $H$ such that for each $x$, either $\forall^{\infty} z \in H \, (z \in R_x)$ or $\forall^{\infty} z \in H \, (z \notin R_x)$. In $\Pra$ we represent $\mathsf{COH}$ as follows: for any function $r \colon \Nb \times \Nb \to 2$, there exists a function $h \colon \Nb \to \Nb$ which gives value $1$ infinitely many times and such that, for each $x$, either $\forall^{\infty} z \,(h(z)=1 \imp r(x,z)=1)$ or $\forall^{\infty} z\, (h(z)=1 \imp r(x,z)=0).$

\subsection{The main transformation result}

The theorem below  essentially says that, for the listed problems, each computable instance can be (usually, uniformly) turned into a primitive recursive instance so that the solutions
are the same up to a Turing degree.
We also note that the result below is a recursion-theoretic result, not a result in $\Pra$, at least as stated.
\begin{theorem}\label{prop10}
For the following problems $\mathsf{P}$,
 for every computable $\mathsf{P}$-instance $X$, 
 there is a primitive recursive $\mathsf{P}$-instance $\hat X$,
such that every solution of $\hat X$ computes a solution of $X$:

\begin{enumerate}

\item $\mathsf{WKL}$,

\item $\mathsf{Completeness\ of\ \mathbb{R}}$,

\item  $\mathsf{RT}^n_k$, for $n\ge 1$ and $k \ge 2$,

\item $\mathsf{SRT}_2^2$,

\item $\mathsf{COH}$,

\item  $\mathsf{IntermediateValueTheorem}$,

\item $\mathsf{HeineBorelTheorem}$.

\end{enumerate}
 \end{theorem}

 \begin{remark}\label{Wei:remark}
 The  reductions between solutions and instances tend to be uniform. For instance, typically we have that  $\h X  = \Phi^{X}$ for some Turing functional $\Phi$, where furthermore the running time of the computations in $\h X$
 are also bounded by a primitive recursive timestamp  function that does not depend on $X$. Also, there is a Turing operator $\Psi$ such that for a solution $\xi$ of $\h X$, we have that $\Psi^{\xi}$ is a solution of $X$.

 In other words, this is a sub-recursive version of the Weihrauch reduction---see, e.g.,~\cite{BraGhe}.
Usually these reductions that we get enjoy various uniformities that allow to relativise the results.
For example, we could throw in a total function $f$ and (subrecursively) relativise everything to $f$ (including the timestamp function),
in the sense of primitive recursive operators. That is, we could allow $f$ to be `primitive recursive', and the results would still typically hold
relative to $f$.
 \end{remark}

\begin{proof}

\noindent (1)
Note that for any computable tree $T\subseteq 2^{<\omega}$,
there is a primitive recursive tree $\hat T\subseteq 2^{<\omega}$
such that $[\hat T] = [ T]$.
To see that, define $\hat T$ so that $\rho\notin \hat T$
iff it is found at time $|\rho|$ that for some $\sigma\preceq\rho$,
$\sigma\notin T[|\rho|]$.\medskip

\noindent (2)
Using a straightforward padding  (a delay of computation by repetition) we can argue that for any computable sequence of strings
$(\rho_n\in 2^n:n\in\omega)$,
there is a primitive recursive sequence of strings
$(\sigma_n\in 2^n:n\in\omega)$ such that
if $\lim_{n\rightarrow\infty }\rho_n$ exists,
then $\lim_{n\rightarrow\infty }\rho_n = \lim_{n\rightarrow\infty} \sigma_n$.
\medskip

\noindent (3)
To explain the idea behind this proof, consider $\mathsf{RT}_2^2$, and view an instance as a computable graph.
The idea is to replace every vertex $x$ with many identical copies of $x$ that form (say) an anti-clique $A_x$. We grow $A_x$ until the graph is calculated on one more vertex $x+1$. Then we begin growing $A_{x+1}$ and wait for the graph on $\{0,1, \ldots, x+2\}$ to be calculated, etc.

Now fix $n\ge 1$ and $k\ge 2$.
Given a computable $\mathsf{RT}^n_k$-instance $c:[\omega]^n \to k$,
we construct a primitive recursive instance $\hat c:[\omega]^n\to k$ together with
a primitive recursive nondecreasing function $p:\omega\rightarrow\omega$ so that
\begin{equation}\label{eq10}
\hat c(x_1,\dots, x_n) = c(p(x_1),\dots, p(x_n)) \text{ whenever } p (x_i)\ne p(x_j), \text{ for any } i,j\le n.  
\end{equation}
Suppose by stage $t$, we have constructed
$\hat c:[t]^n\rightarrow k$ and a nondecreasing function $p:t\rightarrow t$
so that
\begin{align*}
\hat c(x_1,\dots, x_n) = c(p(x_1),\dots, p(x_n)) \text{ for all } x_1, \dots, x_n \in t \text{ with } p (x_i)\ne p(x_j).
\end{align*}
At stage $t$, let $k = p(t-1)$. If
$c(x_1, \dots, x_{n-1},k+1)[t]\downarrow$ for all $x_1, \dots, x_{n-1}\leq k$,
then let $p(t) = k+1$ and  $\hat c(y_1, \dots, y_{n-1},t) = c(p(y_1), \dots, p(y_{n-1}),k+1)$ for all $y_1, \dots, y_{n-1} <t$;
otherwise, let $p(t) = k$ and $\hat c( y_1, \dots, y_{n-1},t ) = c(p(y_1), \dots, p(y_{n-1}),k)$ for all $y_1, \dots, y_{n-1}<t$ with $p(y_j)<k$,
and let $\hat c(y_1, \dots, y_{n-1},t) = 0$ for all $y_1, \dots, y_{n-1} \in p^{-1}(k)$.
Obviously, $\hat c$ satisfies (\ref{eq10}).

Let $Y$ be a solution to $\hat c$. It is clear that we can compute a solution to the original problem from $Y$
using the function $p$. (For instance, consider $p(Y)$.)
\medskip

\noindent (4)
Argue as in (3), noticing that if $c$ is stable, then $\hat c$ is stable as well.
\medskip

\noindent (5)
We think of a $\msf{COH}$ instance $(R_j:j\in\omega)$ as a sequence of strings
$(\rho_j\in 2^j: j\in\omega)$
where
$\rho_j$ is the membership vector $(R_0(j),\cdots, R_{j-1}(j))$.
A set $G\subseteq\omega$ is
 a solution iff $\lim_{i\in G}\rho_i$ exists.
Given a computable sequence $(\rho_j\in 2^j:j\in\omega)$,
we will apply delay of computation to code it by a primitive recursive sequence
$(\sigma_j:j\in\omega)$.
Suppose  before stage $t$, $(\rho_k:k<j)$ has been computed, but  $\rho_j$ has not been computed.
If at stage $t$, $\rho_j$ is still not computed, then let $\sigma_t\in [\sigma_{t-1}]^\preceq\cap 2^t$
be arbitrary.
Otherwise, we choose arbitrary $\sigma_t\in [\rho_j]^\preceq\cap 2^t$
(in which case we say $\sigma_t$ is added due to $\rho_j$;
if at stage $t$, $\rho_j$ is not computed, then $\sigma_t$ is added due to the same
string that was behind the choice of $\sigma_{t-1}$, which must be $\rho_{j-1}$).
Let $\h G\subseteq\omega$ be a solution to $(\sigma_j:j\in\omega)$
(say $\h G= \{j_i:i\in\omega\}$),
and suppose $\sigma_{j_i}$ is added due to $\rho_{k_i}$.
Let $G\subseteq\omega$ be the set $\{k_i:i\in\omega\}$.
Clearly, $G\leq_T \h G$
and   $\lim_i \rho_{k_i} =\lim_i\sigma_{j_i} $.
\medskip

\noindent (6) Fix a computable instance $X$ of   $\mathsf{IntermediateValueTheorem}$.
We produce a primitive recursive instance $\hat X$ such that
every solution of $\hat X$ is a solution of $X$. We do not necessarily require that $\hat X$ and $X$ have the same
set of solutions.

Without loss of generality, assume $X(0)<0$, $X(1)>0$,
and $X(r)\ne 0$ for all $r\in\mathbb{Q}$.
When (at stage $t$) we cannot decide whether $X(1/2)>0$ or $X(1/2)<0$,
we set $\hat X(1/2)[t] = 0$\footnote{Recall that $\hat X$ is seen as a function from
$\mathbb{Q}$ to $\mathbb{Q}^\mathbb{N}$, so $\hat X(r)[t]$
is the $t^{th}$ rational in the rational sequence converging to $\hat X(r)$.}.
Meanwhile,  we set $\hat X(r)[t]= 0$
for all $r\in\mathbb{Q}$ so that $\langle r\rangle\leq t$.
Here $\langle r\rangle$ is the presentation of $r$, i.e., the associated pair of integers.

Once we found, say $X(1/2)<0$, at stage $t$,
we set $\hat X(1/2)[\hat t] = -1/2^t$ for all $\hat t\geq t$.
We define $\hat X(r)[\hat t] = -1/2^t$ for all $r\in\mathbb{Q}\cap [0,1/2]$
(i.e., we don't care about the value of $\hat X(r)$ for $r\in [0,1/2]$,
since we know that $X$ has a solution in $(1/2,1)$;
this is to guarantee the continuity of $\hat X$).

Next, we look at the value $X(3/4)$. If (at stage $t$)
we cannot decide whether $X(3/4)>0$ or $X(3/4)<0$,
we set $\hat X(r)[t] = 0$ for all $r\in \mathbb{Q}\cap (1/2,1]$
with $\langle r\rangle \leq t$.
If at stage $t$ we see  that, e.g., $X(3/4)>0$,
we let $\hat X(r)[\hat t] = 1/2^t$ for all $\hat t \geq t$
and all $r\in \mathbb{Q}\cap [3/4,1]$.

The rest of the construction goes similarly.
It is not hard to define the continuity modulus
$\delta$; we omit further details.
\medskip

\noindent (7) Fix a computable instance $(I_s:s\in\mathbb{N})$ of $\mathsf{HeineBorelTheorem}$.
We produce a primitive recursive
instance $(\hat I_s:s\in\mathbb{N})$
such that every solution of $(\hat I_s:s\in\mathbb{N})$
primitively recursively computes a solution of $(I_s:s\in\mathbb{N})$.

Recall  that we are allowed to have the empty interval in the
sequence $(I_s:s\in\mathbb{N})$, we can use it to delay our computation.
If at stage $t$, we have not finished computing
$I_s$, then list the empty interval into $(\hat I_s:s\in\mathbb{N})$.
Clearly, $\bigcup_s\hat I_s= \bigcup_s I_s$.
\end{proof}

\begin{remark}
In (7) above,  the conclusion
would still hold even if the empty interval was not allowed.
We can use the interval $[0,1/2]$ to `code' the interval $[0,1]$.
Then any interval $I\subseteq [0,1]$ is `coded' by the interval $I/2$.
Thus, we can always spam the interval $(1/2,2)$ by enumerating it into $(\hat I_s:s\in\omega)$---we
put interval $(1/2,2)$ into it when we wait for our computation to halt.
So given a solution $r $ of $(\hat I_s:s\in\omega)$,
we have $2r$ as a solution of $(I_s:s\in\omega)$.
\end{remark}

\begin{remark}
Note that for $\mathsf{COH}$, the computation of a solution of $(R_n:n\in\omega)$
using $(\sigma_j:j\in\omega)$ is not primitive recursive.
Indeed, there is a computable $\mathsf{COH}$ instance $X$ such that
there is no primitive recursive instance $\hat X$ such that every
solution of $\hat X$ primitively recursively computes a solution of $X$.
Actually, there is a single computable set $R\subseteq \omega$ such that
for every primitive recursive instance $X$ of $\mathsf{COH}$,
there is a solution $Y$ of $X$ such that
no $Z\in PR(Y)\cap 2^\omega$ is cohesive for $R$. That is, either $Z$ is finite or
 $Z\cap R,Z\cap \overline{R}$  are both infinite.
\end{remark}

The above property allows to derive information about minimal models failing some principles. Obviously, if $\rca \nvdash \mathsf{P}$, then the \lq functional-translation\rq\ of the model witnessing the unprovability of $\mathsf{P}$ over $\rca$,  witnesses the unprovability of $\mathsf{P}$ over $\Pra$. However, if such $\mathsf{P}$ has the above property and $(\omega, \Delta^0_1\text{-Def}(\omega)) \nvDash \mathsf{P}$, then we can argue that $(\omega, PRec(\omega)) \nvDash \mathsf{P}$. In fact, let $X$ be a computable $P$-instance  with no computable solutions. Let $\hat X$ be a primitive recursive $P$-instance, so that $\hat X \in PRec(\omega)$. Then there is no solution $\hat{Y}  \in PRec(\omega)$, otherwise a solution $Y$ to $X$ would belong to $\Delta^0_1\text{-Def}(\omega)$, since $Y \le_T \hat{Y}$, contradicting the assumption.

\medskip

We can  of course use Theorem~\ref{prop10}
as a base of our intuition or to argue in standard models.
For example, we conjecture that the following holds:

\begin{proposition} Over $\Pra$, $\mathsf{IntermediateValueTheorem}$ is equivalent to
$\dca$.
\end{proposition}

\begin{proof}[Sketch]
Noting that a real can be viewed as a $\{0,1\}$-valued function, we can repeat the
usual dichotomy argument to see that $\dca$ implies $\mathsf{IntermediateValueTheorem}$.
For the other direction, we recycle the well-known fact from computable analysis that
any computable real can be realised as a solution to some computable instance
of $\mathsf{IntermediateValueTheorem}$.
We will then have to  mimic the proof of Theorem~\ref{prop10}(5) to actually produce a primitive recursive instance
of $\mathsf{IntermediateValueTheorem}$ directly from an instance of $\dca$.
Then we could argue in $\Pra$ that this works.
\end{proof}

Even though we strongly conjecture that the idea outlined above can be indeed implemented, the actual formal implementation would likely be a bit tedious
(cf.\ the proof of Proposition~\ref{Lem_CatDense}).
It would be very nice to have a general fact that would imply this sort of results, rather than checking the details
for each specific result that involves a problem with `enough' primitive recursive instances.


\subsection{Can we always use padding?}

It seems that for all combinatorial problems that we are aware of, primitive recursive instances are as
computationally powerful as computable instances.

\begin{question}\label{q:padding}
Is there a natural problem $\mathsf{P}$ so that for some computable instance $X$,
there is no primitive recursive instance $\hat X$ such that every solution of $\hat X$
computes a solution of $X$?


\end{question}

The question above is of course loosely stated, since `natural' is a subjective quality.
Preferably, we would like to find a reverse mathematical problem  $\mathsf{P}$ that
has already been studied in the past rather than manufacture an ad hoc problem.

We therefore leave the question open, but we give an example
of a \emph{somewhat natural} problem to illustrate what can potentially go wrong with primitive recursive instances: there could be simply not enough
such instances. The example below is in the spirit of the `categoricity' examples that we have seen in Section~\ref{sec:examples}.
 It is based on an old result in effective algebra that can be traced back to Mal'cev~\cite{Mal-62} and the well-known
description of subgroups of the rationals $(\mathbb{Q},+)$ by their Baer types~\cite{Baer-37}. The elementary result says that two rank 1 computable TFAGs (torsion-free abelian groups) having the same types
have to be isomorphic. The corresponding old effective algebraic result says that groups having the same type are indeed computably isomorphic, and that every computable rank 1 group has to have a c.e.\ (but not necessarily computable) type.
 We clarify these terms below.

\

\begin{definition}
The following definitions can be formalised in $\Pra.$ We represent groups as \emph{structures} (\S\,\ref{sub:str}).
\begin{itemize}

\item We say that an additive torsion-free abelian group (TFAG) $G$ has rank 1 if
$\forall g, h \in G \setminus \{0\}$ we can find non-zero $m,n \in \mathbb{Z}$ such that
$mg = nh$\footnote{Using primitive recursion, define $mg$ via $1g = g$ and $(k+1)g = kg +g$. For a negative integer $m$, $mg = - (-m)g$. In particular, we can express that a group is torsion-free by stating that the order of any non-zero element is infinite.}.

\item Say that  rank 1 TFAGs  $G$ and $H$
have the same (Baer) type, written ${\bf t}(G) = {\bf t}(H)$, if there exist non-zero $g$ in $G$ and a non-zero $h$ in $H$ such that,
for any $m \in \mathbb{N}$,
$$ \,\, H \models m | h \iff G \models m | g,$$
where $A \models k | x$ means that there is a $y \in A$ such that $ky = x$.

\item A  rank 1 TFAG $G$ is
\emph{Baer categorical}
if whenever  $H$ is a rank 1 TFAG such that ${\bf t}(G) = {\bf t}(H)$,
we have that there is an isomorphism from $H$ onto  $G$.

\item An instance of $\mathsf{BaerType}$ is a rank 1 Baer categorical $G$. A solution is a set
$H \subseteq \mathbb{N}$ such that for some non-zero $g \in G$,
$H  = \{m: m | g\}.$  (Output the empty set otherwise.)

\end{itemize}
\end{definition}

In other words, $\mathsf{BaerType}$ takes a  `categorical' group and outputs its isomorphism invariant.
The example below exploits that there are simply not enough primitive recursive Baer categorical groups.
On the other hand, Baer's classification of rank 1 TFAGs holds computably, and thus there are enough instances to code
an arbitrary c.e.~set.

\begin{proposition}\label{prop:whatever}
For any c.e.~Turing degree $\mathbf{b}$, there is a computable instance of $\mathsf{BaerType}$
any solution of which has degree $\mathbf{b}$.
In contrast, the only primitive recursive instance of $\mathsf{BaerType}$ is the trivial group, and the only solution is the empty set.
\end{proposition}

\begin{proof}
Use the aforementioned classification of Baer combined with the fact that
every rank 1 group is \emph{computably categorical} meaning that any two isomorphic computable copies are computably isomorphic.
Code any c.e.~set $S$ into a computable $G \leq \mathbb{Q}$ as follows:

\begin{center}
$p_i | 1 \in G$ iff  $ i \in S$.
\end{center}
Since any other non-zero element is a rational multiple of $1$, any such element also codes $S$ up to a finite difference.
Thus, this coding is degree-invariant. (This argument is folklore; see \cite{Melnikov-Survey-Groups} for further details.)

However, in the minimal standard model of $\Pra$, the only  \emph{Baer categorical} rank 1 TFAG is the trivial group.
This is because for any nontrivial  primitive recursive TFAG, there exists a (fully) primitive recursive group computably isomorphic to it but not primitively recursively isomorphic to it.
(We cite \cite{Kal17} for a detailed proof.)
It follows that the only possible primitive recursive instance is $\{0\}$, and the only possible solution is the empty set.
\end{proof}

 \subsection{A note about $\aca$ and $2^{\mathbb{N}}$-$\aca$}\label{subs:aca} We conjecture that, much in the spirit of the categoricity results discussed in \S\,\ref{subsect:combinat},
the proof of Proposition~\ref{prop:whatever} outlined above can be carried out in models that are not necessarily standard or minimal.
We conjecture that
$\Pra \vdash \mathsf{BaerType} $ while  $\rca \vdash \aca \leftrightarrow \mathsf{BaerType}$;  we leave the verification of this claim to the reader.

We shall not really look at problems equivalent to $\aca$ in the present paper, but we conjecture that a large portion of results 
known to be equivalent to $\aca$ over $\rca$ will be equivalent to $\aca$ or $2^{\mathbb{N}}$-$\aca$ over $\Pra$ as well; we clarify what we mean by $\aca$ and $2^{\mathbb{N}}$-$\aca$ below.

The function-based version of $\aca$ is similar to the function-based version $\qfac$ of $\rca$, but it asserts the existence of $\Sigma^0_n$ definable functions rather than just $\Delta^0_1$-definable functions. Similarly to $\rca$, it also has a bounded version that is strictly weaker.
More specifically, the bounded version, that we denote $2^{\mathbb{N}}$-$\aca$, postulates the existence of arithmetically definable $\{0,1\}$-valued functions.
 Notably,  over $\Pra$, $2^{\mathbb{N}}$-$\aca$ does not imply $\rca$ (this is similar to~Proposition~\ref{pr:rca}).

Notice that, $2^{\mathbb{N}}$-$\aca$, and hence $\aca$, implies arithmetical induction, since any arithmetical formula becomes equivalent to a quantifier-free formula (with extra parameter the defining $\{0,1\}$-valued function), over which one can apply $\qfi$.

\section{$\wkl$ over $\Pra$} \label{Sec:WKL}

Recall that instances of $\wkl$ are binary trees, and solutions are paths through the trees. We can represent a binary tree via a set of finite $\{0,1\}$-strings (identified with its characteristic function) closed under taking the prefix. Note that a solution is necessarily a $\{0,1\}$-valued function; in particular, it is primitively recursively bounded.
We therefore obtain the following (seemingly well-known) fact.

\begin{proposition}\label{lem:wkl} Over $\Pra$, $\wkl$ is strictly stronger than $\dca$, is incomparable with $\rca$, and is strictly weaker than $\aca$.
\end{proposition}

\begin{proof}
To see why $\Pra+\wkl \vdash \dca$, fix an instance of $\dca$ whose solution is $f$.    Use the idea in (1) of Theorem~\ref{prop10} to define a primitive recursive tree $T$
such that the only path through $T$ is $f$. (Note that we do not need induction to argue that the only path through $T$ is $f$.)
Of course, if we were to give full details, then we would define the tree using bounded versions of formulae from the instance of $\dca$ and primitive recursion to produce the tree.

\

\begin{footnotesize}
For instance, we can argue as follows. Consider an instance of $\dca$, i.e., $\Delta^0_0$-formulae $\varphi(n,x)$ and $\psi(n,x)$ such that $\forall n[\exists x \varphi(n,x) \leftrightarrow \forall x \neg\psi(n,x)]$. Then a string $\sigma$ belongs to our tree $T$ if and only if
\[
	(\forall i < |\sigma|) [ (\sigma(i) = 0 \rightarrow (\forall x\leq|\sigma|) \neg\varphi(i,x)) \ \wedge\ (\sigma(i) = 1 \rightarrow (\forall x\leq|\sigma|) \neg\psi(i,x))].
\]
Then the only path through our tree is $f$. Indeed, if $g$ is an arbitrary path through $T$, then:
\begin{itemize}
	\item[] $g(i) = 0$ $\Rightarrow$ $\forall x \neg \varphi(i,x)$ $\Leftrightarrow$ $f(i) = 0$;
	
	\item[] $g(i) = 1$ $\Rightarrow$ $\forall x \neg \psi(i,x)$ $\Leftrightarrow$ $f(i) = 1$.
\end{itemize}
This is similar to the proof of $\Sigma^0_1$-separation from $\wkl$ given in \cite[Lemma IV.4.4]{sosoa}.
\end{footnotesize}

\

Since there are infinite primitive recursive binary trees  with no computable paths ((1) of Theorem~\ref{prop10} combined with folklore),
the standard minimal model of $\Pra+\rca$ illustrates that $\Pra+\rca \not \vdash \wkl$, and in particular $\Pra+\dca \not \vdash \wkl$.
To see why $\Pra+ \wkl \not \vdash \rca$, follow the proof of Proposition~\ref{pr:rca} to construct a standard model of $\Pra+ \wkl$ that contains only primitively recursively bounded functions. The proof that $\aca$ implies  $\wkl$ is essentially the same as the standard proof in the set-based system, up to notation.\end{proof}

\begin{remark}
We do not need the full power of $\aca$ to deduce $\wkl$; the existence of arithmetical $\{0,1\}$-valued functions would suffice.
\end{remark}

It is immediate to see that it is possible to compute a path in each infinite pruned tree, i.e., a tree without leaves. Thus,  $\rca$ proves both $\mathsf{WKL}$ and $\mathsf{KL}$ for pruned trees. In this setting we observe the following.

\begin{proposition}
$\Pra$ proves that each binary pruned tree has an infinite path.
\end{proposition}
\begin{proof}[Proof sketch]
Let $T \subseteq 2^{<\Nb}$ be without leaves. One can define a path $P$ inductively as follows
\begin{align*}
P(0)&=r  \\
P(n)&= P(n-1)\smf i
\end{align*}
where $r$ is the root and $i \in \{0,1\}$ is minimal such that $P(n-1)\smf i \in T$.
\end{proof}
In contrast, we have the following, also highly expected, fact.

\begin{proposition}\label{Prop:noLeavesKL}
Over $\Pra$, $\rca$ is equivalent to the following: each infinite pruned tree $T\subseteq \Nb^{<\Nb}$  has an infinite path.
\end{proposition}
\begin{proof}
Let $T\subseteq \Nb^{<\Nb}$ be an  infinite pruned tree, so that it holds that $\forall \sigma\, \exists m\, (\sigma\smf m \in T)$. Let $f \colon \Nb \to \Nb$ be a choice function for \[
\theta(\sigma,m) = (\sigma\smf m \in T),
\]
and define a path $P$ through $T$ such that $P(n)=f^n(r)$, for $r$ the root of $T$ and $n \in \Nb$.

Let $\theta$ be a quantifier-free formula such that $\forall n \, \exists m\, \theta(n,m)$. By $\Delta^0_0$-comprehension define a tree $T\subseteq \Nb^{\Nb}$ as follows:
\[\sigma \in T \Leftrightarrow \forall n < |\sigma|\, (\sigma(n)=m \leftrightarrow \theta(n,m))\]
It is immediate to check that $T$ is infinite and pruned. The path provides the desired choice function for the formula $\theta$.
\end{proof}

Recall that we defined  $\mathsf{HeineBorelTheorem}$ in Definition~\ref{def:anprob}.
\begin{proposition}
\label{no4prop1}
Over $\pra^2$,
 $\mathsf{HeineBorelTheorem}$
is equivalent to  $\wkl$.

\end{proposition}
\begin{proof}
Working in $\Pra$, we give primitive recursive definitions. It then takes only quantifier-free induction (combined with appealing to the primitive recursive schemata)
to argue that these processes define the desired objects.
Also, recall that all our intervals have rational end points, so, in particular, inclusion of two given intervals becomes a quantifier-free formula.

\smallskip

($\Rightarrow$).
Fix the natural primitive recursive homeomorphism $h:2^\mathbb{N}\rightarrow \mathcal{C}$, where
$\mathcal{C}$ denotes the Cantor set.
The homomorphism $h$ and its inverse are realised by primitive recursive functionals.

 \begin{remark}\label{homeo-remark}  We shall avoid giving the formal definition of a homeomorphism
in $\Pra$ and treat the operator $h$  merely as a notation that can be extracted from the primitive recursive definition of the Cantor set $\mcal{C}$.
It should be clear to the reader at this stage how this sort of operators can be formally defined; we omit this.
\end{remark}


Let $T\subseteq 2^\mathbb{N}$ be an infinite tree.
We will primitively recursively compute a $\mathsf{HeineBorelTheorem}$
instance $(I_s:s\in\omega)$
such that any solution of $(I_s: s\in\omega)$
primitively recursively computes a $\wkl$-solution for $T$.
Firstly, let all intervals in $[0,1]\setminus\mathcal{C}$
be included in $(I_s:s\in\omega)$.
Secondly, for each string $\rho\notin T$,
put the interval corresponding to $\rho$ into $(I_s:s\in\omega)$.
Obviously, this  $(I_s:s\in\omega)$ is primitive recursive in $T$.
One can arrange the construction, by slowly enumerating `small enough' intervals in $[0,1]\setminus\mathcal{C}$, in such a way that $\bigcup_{s<t} I_s\nsupseteq [0,1]$ for all $t$. Hence, $(I_s:s\in\omega)$ is an instance. Now, let (a presentation of) a real $r$ be so that
$r\notin \bigcup_s I_s$.

It is easy to see that if $X\not\in [T]$, then $\bigcup_s I_s$ contains the real $h(X)$.
Therefore, $h^{-1}(r)\in [T]$.

\smallskip

($\Leftarrow$).
Let $(I_s:s\in\omega)$ be an $\mathsf{HeineBorelTheorem}$ instance.
Recall that each string $\rho\in 2^\mathbb{<N}$ represents
an interval of form $[k_\rho/2^{|\rho|},(k_\rho+1)/2^{|\rho|}]$.
To define $T$,
whenever we see $\bigcup_{s<n} I_s\supseteq [k_\rho/2^{|\rho|}, (k_\rho+1)/2^{|\rho|}]$,
we put $\rho$ of length $n$ in $\overline{T}$, i.e., the complement of $T$. Otherwise,  declare  $\rho$ in $T$.
Let $Y\in [T]$.
Clearly
$(k_{Y\upharpoonright n}/2^n:n\in\omega)$ is a sequence of rationals representing a real $r\in [0,1]$
such that $r\notin \bigcup_s I_s$. \end{proof}

We conjecture that many basic theorems, such as $\mathsf{G\ddot{o}delCompletenessTheorem}$,
that are equivalent to $\wkl$ over $\rca$ remain equivalent to $\wkl$ over $\Pra$.  The following elementary but useful fact
 helps to study problems whose solutions lie in $2^{\mathbb{N}}$.

\begin{lemma}\label{lemma:whatever}
Suppose $P$ and $Q$ are problems such that $P$-instances lie in $2^{\mathbb{N}}$
and $Q$-solutions  lie in $2^{\mathbb{N}}$. Suppose also that $P$ implies $\dca$ over $\Pra$.
If $P$ implies $Q$ over $\rca$, then $P$ implies $Q$ over $\Pra$.

\end{lemma}

\begin{proof}
Suppose $M \models \Pra +P$. Define  an expansion $\hat{M}$ of $M$ by taking the collection of all $\Delta^0_1$-definable functions
in $M$. It should be  clear that $\hat{M} \models \rca$. Note also that $\hat{M} \cap 2^{\mathbb{N}} =M \cap 2^{\mathbb{N}} $ because  $P$ implies $\dca$ over $\Pra$ (note that $\Delta^0_1$-definability is transitive). It follows that $M$ already contains all $P$-instances that are present in $\hat{M}$.
Since $\hat{M}$ is an expansion of $M$ and $M \models P$, it evidently contains all the solutions of $P$ too.
So it follows that $\hat{M} \models \rca+ P$, and thus $\hat{M} \models Q$. Recall that all solutions of $Q$ are in $2^{\mathbb{N}}$, and $\hat{M} \cap 2^{\mathbb{N}} =M \cap 2^{\mathbb{N}} $. We conclude that $M \models Q$.\end{proof}

We obtain:

\begin{theorem}\label{Cor:WKLPra}
Suppose all $Q$-solutions  lie in $2^{\mathbb{N}}$. If $\wkl$
implies $Q$ over $\rca$, then $\wkl$ implies $Q$ over $\Pra$.
\end{theorem}
\begin{proof} By \cref{lem:wkl}, we have that $\Pra +\wkl \vdash \dca$.
Under a suitable coding of subsets of $2^{\mathbb{N}}$, instances of $\wkl$ can be viewed as primitively recursively bounded functions. We can represent instances of $\wkl$ as $2$-bounded functions, i.e., elements of $2^{\mathbb{N}}$. It remains to apply~\cref{lemma:whatever}.
\end{proof}

We now derive several corollaries of the result stated above.

\

In this contest, where $\iso$ may fail, one needs a bit of care to formalise, inside the theory, the notion of Turing reduction. We borrow the definition of \lq being recursive in\rq, as in \cite{chong-yang:jump-cut}, so that $\forall X\, \exists Y\, (Y \leq_T X)$ means that there exists a monotonic $\Sigma^0_1$-functional $\Phi$ such that $y \in Y$ ($y \notin Y$) if and only if there are two coded sets $P\subseteq X$ and $N \subseteq \Nb\setminus X$ such that $\la x,1,P,N \ra \in \Phi$ ($\la x,0,P,N \ra \in \Phi$). Notice that, in our contest, both  $P\subseteq X$ and $N \subseteq \Nb\setminus X$ are $\Delta^0_0$-properties.  For more details we refer to the cited paper.

A set $H \subseteq \Nb$ is homogeneous for a $\sigma \in 2^{<\Nb}$ if there exists a colour $i<2$ such that $\forall n \in H\, (n < |\sigma| \imp \sigma(n)=i)$. A set  $H \subseteq \Nb$ is homogeneous for an infinite tree $T \subseteq 2^{<\Nb}$ if the tree  $\{\sigma \in T \mid
H \text{ is homogeneous for } \sigma\}$ is infinite.

Let $T$ be a theory. A formula $\varphi(x_0,\dots, x_n)$ of $T$ is an atom of $T$ if for each formula $\psi(x_0,\dots, x_n)$ it holds that $T \vdash \varphi \imp \psi$ or $T \vdash \varphi \imp \neg\psi$, but not both.
The theory $T$ is atomic if, for every formula $\psi(x_0,\dots, x_n)$ consistent with $T$,
there is an atom $\varphi(x_0,\dots, x_n)$ of $T$ such that $T \vdash \varphi \imp \psi$.
The types of $T$ are subenumerable if there exists a set $S$ such that, for every type $\Gamma$ of $T$, there is an $i$ such that $\{\varphi \mid \la i,\varphi\ra \in S \}$ and $\Gamma$ imply the same formulae in $T$.
A model $\mathcal{M}$ of $T$ is atomic if every $n$-tuple from $\mathcal{M}$ satisfies an atom of $T$.

\begin{corollary}
Over $\Pra$, the following principles are implied by $\wkl$:
\begin{enumerate}
  \item $\mathsf{WWKL}$, Weak Weak K{\H{o}}nig's Lemma, i.e., every tree $T \subseteq 2^{<\Nb}$ such that
\[\frac{ |\{\sigma \in  2^n \mid \sigma \in T \}|}{2^n}\]
is uniformly bounded away from zero for all $n$ has an infinite path.
\item $\mathsf{DNR}$, Diagonally Non-Recursive function, i.e., for each $A\subseteq \Nb$ there exists a function $f\colon \Nb \to \Nb$ such that $f(e)\ne \varphi^A_e(e)$, for any $e \in \Nb$.
\item $\forall X\, \exists Y\, (Y \nleq_T X)$.
\item  $\mathsf{AST}$, i.e., Atomic model theorem with Subenumerable Types: Let T be a complete atomic theory whose types are subenumerable. Then T has an atomic model.
\item $\mathsf{RWKL}$, Ramsey-type Weak K{\H{o}}nig’s Lemma, i.e., for every infinite subtree of $2^{<\Nb}$,
there is an infinite homogeneous set.
\end{enumerate}
\end{corollary}
\begin{proof}
In light of \cref{Cor:WKLPra} we only need to check that the items above are consequences of $\wkl$ and that their solutions belong to $2^{\Nb}$.

Items $(1)$ and $(3)$ are clear consequences of $\wkl$.


Over $\rca$, $\wkl$ implies the existence of $\{0,1\}$-valued diagonally non-computable functions. Moreover,  the existence of $\{0,1\}$-valued diagonally non-computable functions trivially implies $\mathsf{DNR}$. Thus, $(2)$ holds.

Over $\rca$, $\mathsf{AST}$ is implied by $\wkl$ by \cite[Theorem 6.3]{hirschfeldt2009atomic}. In order to apply \cref{Cor:WKLPra}, we represent a model as a $\{0,1\}$-valued function, namely the signature functions are represented through their graphs.

Over $\rca$, $(5)$ is implied by $\wkl$ by \cite[Theorem 3]{flood2012reverse}.
\end{proof}

\subsection{Uniform continuity} 
Recall that, over $\rca$, $\mathsf{UniformContinuityOn[0,1]}$ is equivalent to $\wkl$ (see \cite[Exercise IV.2.9]{sosoa}).

\begin{definition}
An instance of
$
\mathsf{UniformContinuityOn[0,1]}
$ (in a model $\mathcal{M}$)
is a presentation $(X,\delta)$ of
a function on $[0,1]$
(see \S\,\ref{sec:morenot}).
A solution of $X$
is  a \emph{modulus of uniform continuity} of $X$, which is
a function $h\in \omega^\omega$
 such that $|r-\hat r|<2^{-h(n)}$ implies
 $|X(r)-X(\hat r)|<2^{-n}$
 for all $r,\hat r\in   [0,1]$.
\end{definition}

We now determine the proof-theoretic strength of 
$\mathsf{UniformContinuityOn[0,1]}$ over $\Pra$.

\begin{theorem}\label{theorem:UniformContinuityOn[0,1]}
Over $\Pra$,
 $\mathsf{UniformContinuityOn[0,1]}$ is equivalent to  $\wkl+\rca$.

\end{theorem}

\begin{proof}
Note that over $\rca$,
 $\mathsf{UniformContinuityOn[0,1]}$ is equivalent to  $\wkl$.
 So it suffices to show that over $\Pra$,
  $\mathsf{UniformContinuityOn[0,1]}$  implies $\rca$.
The proof is based on a recursion-theoretic lemma.
We first explain how to prove the lemma and then we explain how to turn it into an argument in $\Pra$.
\begin{lemma}
Fix a computable function $g$.
There is a primitive recursive continuous
function $(h,\delta)$ (in the sense of \S\ref{sec:morenot}) so that any  uniform continuity modulus of $(h,\delta)$ primitively recursively computes
 a function dominating $g$.
\end{lemma}
\begin{proof} We define a continuous function  represented via primitive recursive  $(h, \delta)$. In order for $(h, \delta)$ to be a presentation of a continuous function we must
make sure that $\delta$ gives arbitrarily small covers of $[0,1]$. But we do not have to produce these covers `quickly'. In other words, we can delay
the definition of the next refined cover until we are ready, as long as every rational point that we consider at any stage is within its $\delta$-neighbourhood 
that could be quite small.

 If $m$ is a modulus of uniform continuity for $g$, then 
 our goal is to make sure that  $m(i)> g(i)$, for every $i$. 
 Fix some irrational but primitive recursive point $\xi$, say $\xi =
\sqrt{2}/2.$ (Fixing $\xi$ ahead of time is not really necessary, but it will make things a bit more transparent at least in the standard model.)
 We build it so that the infinitely many breaking points of $h$ converge at the accumulation point $(\xi,  h(\xi))$, where $h(\xi) = \sup_{x \in [0,1]} h(x) =2$.
 Outside of $\xi$ the function $h$ will be piecewise linear.
 As the argument of $h$ approaches $\xi$ the value of $h$ will be increasing, but the speed with which it will be increasing locally will
 be determined by the construction, thus making $h$ very steep around $\xi$. The reader is perhaps already convinced that this can be done primitively recursively by delaying, but we give more details nonetheless.


 To make $m(i)> g(i)$, we ensure that there is a pair of rational  points $x_i $ and $z_i$ so that $z_i < x_i< \xi$ and 
 $|x_i -z_i |  < 2^{-g(i)}$
 but $| h(z_i) -  h(x_i)| > 2^{-i}$. 
 For that, we wait for $g(i)$ to converge. While we wait, we define the function 
 on more and more rational points, as follows.
 If $r < \xi$ is a new rational point so that $h(r)$ has to be defined, then use bounded search to 
 find the closest rational $q <\xi$ so that $h(q)$ has already been defined at a previous stage. 
Set $h(r) = h(q)$, and also declare $\delta(r,n)$ to be so small that the point $\xi$ is not covered by the $\delta$-neighbourhood (nbhd) of $r$.
Notice that if a rational point $d$ is in the interval between $r$ and $q$, then at the stage at which we consider $d$ the value of $h(d)$
will be set equal to $h(r)$, and we can define $\delta$ for $d$ so that the $\delta$-nbhd of $d$ is inside the interval between $r$ and $q$.
We proceed in this manner primitively recursively until $g(i)$ is calculated.
Once this is done, we primitively recursively pick the right-most rational to the left of $\xi$ for which $h$ has already been defined and set  $z_i$ equal to this rational.
Note that $h(x_{i-1}) = h(z_i)$. We then pick a rational $x_i$ between $z_i$ and $\xi$ so that is $2^{-g(i)}$-close to $z_i$, is not covered by the $\delta$-nbhd around $z_i$ (for the precision moduli defined so far for $z_i$), and set 
$$h(x_i) = h(z_{i}) + 2^{-i+1} = h(x_{i-1}) + 2^{-i+1},$$
and we also define $
\delta$ so that $
\xi$ is covered by the $\delta$-nbhd of $x_i$ at the stage. 

\

\begin{footnotesize}
	A more formal construction could be arranged as follows. At a stage $s\in\omega$ we follow the (current) computation of $g(i_s)$ for a number $i_s\in\omega$. At stage $0$, we put $i_0 := 0$, $h(0) := 0$, and $\delta(0,n) := n$ for all $n\in\omega$. We also define auxiliary parameters $l_0:= 1$ and $t_0:=0$.
	
	\emph{Stage $s+1$.}	Suppose that the value $g(i_s)[s+1]$ is undefined (i.e., after $s+1$ steps of computation, the value $g(i_s)$ has not been computed yet). Without loss of generality, here we assume that $g(i_0)[1]$ is undefined. Every $j$ such that $0< j < 2^{l_{s}}$ \emph{and} $j\cdot 2^{-l_s} < \sqrt{2}/2$ satisfies one of the following three cases:
	\begin{enumerate}
		\item The value $h(j\cdot 2^{-l_s})$ has been already defined at one of the previous stages.
	
		\item The value $h(j\cdot 2^{-l_s})$ is still undefined, and there exists (the least) $k_1$ such that $k_1 > j$, $k_1\cdot 2^{-l_s} <\sqrt{2}/2$, and $h(k_1\cdot 2^{-l_s})$ was defined at previous stages. Then we find the greatest $k_0$ such that $k_0 < j$ and $h(k_0\cdot 2^{-l_s})$ was defined at previous stages. We set
		\[
			h(j\cdot 2^{-l_s}) := h(k_0\cdot 2^{-l_s}) + (j-k_0)\cdot \frac{h(k_1\cdot 2^{-l_s}) - h(k_0\cdot 2^{-l_s})}{k_1-k_0},\quad \delta(j\cdot 2^{-l_s}, n) := n + t_s.
		\]
		
		\item The value $h(j\cdot 2^{-l_s})$ is still undefined, and there is no $k_1$ such that $k_1 > j$, $k_1\cdot 2^{-l_s} <\sqrt{2}/2$, and $h(k_1\cdot 2^{-l_s})$ was defined at previous stages. Again, we find the greatest $k_0$ such that $k_0 < j$ and $h(k_0\cdot 2^{-l_s})$ was defined at previous stages. We define
		\[
			h(j\cdot 2^{-l_s}) := h(k_0\cdot 2^{-l_s}), \quad \delta(j\cdot 2^{-l_s}, n) := n + t_s.
		\]
	\end{enumerate}
	We set $i_{s+1} := i_s$, $l_{s+1} := l_s + 1$, and $t_{s+1} := t_s$.
	
	Now assume that $g(i_s)[s+1]$ is defined. Then we set $i_{s+1} := i_s +1$, $l_{s+1} := 1+\max(l_s, g(i_s))$, and $t_{s+1} := \max(t_s, g(i_s)+2)$. Let $m := \max(l_s,g(i_s))$. Every $j$ such that $0< j < 2^{m}$ and $j\cdot 2^{-m} < \sqrt{2}/2$ satisfies one of the following three cases:
	\begin{enumerate}
		\item The value $h(j\cdot 2^{-m})$ has been already defined at a previous stage.
		
		\item The value $h(j\cdot 2^{-m})$ is still undefined, and there exists the least $k_1$ such that $k_1 > j$, $k_1\cdot 2^{-m} <\sqrt{2}/2$, and $h(k_1\cdot 2^{-m})$ was defined at a previous stage. Then find the greatest $k_0$ such that $k_0 < j$ and $h(k_0\cdot 2^{-m})$ was defined at a previous stage. Set
		\[
			h(j\cdot 2^{-m}) := h(k_0\cdot 2^{-m}) + (j-k_0)\cdot \frac{h(k_1\cdot 2^{-m}) - h(k_0\cdot 2^{-m})}{k_1-k_0},\quad \delta(j\cdot 2^{-m}, n) := n + t_s.
		\]
		
		\item In the remaining case find the greatest $k_0$ such that $k_0 < j$ and $h(k_0\cdot 2^{-m})$ was defined at a previous stage. Declare
		\begin{equation}\label{equ:aux-final}
			h(j\cdot 2^{-m}) := h(k_0\cdot 2^{-m}) + 2^{-i+1}, \quad \delta(j\cdot 2^{-m}, n) := n + t_{s+1}. \tag{$*$}
		\end{equation}
	\end{enumerate}
	
	This concludes the description of the formal construction. Notice that here $z_i$ is chosen as $k_0\cdot 2^{-m}$ in~(\ref{equ:aux-final}), and one can take $x_i$ equal to $(k_0+1)\cdot 2^{-m}$.
\end{footnotesize}

\

To make the function continuous, we also implement the same procedure for rationals $r> \xi$, and simultaneously define
a sequences $(w_i)_{i \in \omega}$ and $(y_i)_{i \in \omega}$ that converge to $\xi$ from the right. This is done similarly to how we defined $z_i$ and $x_i$ mutatis mutandis; we omit this. 
Note that the function $h$ is indeed continuous at the point $\xi$, with $\lim_{x \rightarrow \xi}h(x)$ well-defined (and is equal to 2). The function is therefore continuous at $\xi$. It is also continuous at any other point, by the construction. It is also primitive recursive (by the construction).
\end{proof}
In any $\omega$-model, the theorem now follows by subrecursive relativisation of the above argument.
To get an argument in $\Pra$, we use the restricted Church--Turing thesis to produce a primitive recursive schema implementing the lemma above. For that, 
we fix an instance of $\rca$ (more formally, of $\qfac$) and use primitive recursion to produce an instance of  $\mathsf{UniformContinuityOn[0,1]}$ along the lines of the proof of the lemma above. Some care must be taken. For instance, it is perhaps most convenient to use Proposition~\ref{prop:mini} and refer to the minimisation operator applied to some function that exists in the model. We shall use this function in our primitive recursive schema.
Also, to avoid appealing to $\wkl$, we have to be very careful and explicit in the way we define $\delta$ and the associated covers of $[0,1]$. For that, for parameter $s$ in the scheme that corresponds to a `stage', we always subdivide $[0,1]$ into more and more refined rational intervals using, e.g., nested  partitioning of the form  $$0< 2^{-s} <  2 \cdot 2^{-s} < \ldots < (k+1) 2^{-s} < \ldots < 1-2^{-s}< 1,$$
which do correspond to covers of the whole $[0,1]$ without any reference to $\wkl$. It does not take any induction to conclude that the formal schema gives a presentation of a continuous function.
 We then argue  using only bounded induction and bounded comprehension that, using any solution of  this instance of $\mathsf{UniformContinuityOn[0,1]}$ produced by the schema, we can calculate the fixed instance of $\rca$. We invite the reader to reconstruct the tedious but not difficult formal details.
\end{proof}

\section{Further open questions}\label{sec:que}

Recall that we stated Questions \ref{q:1}, \ref{q:2}, \ref{q:3}, \ref{q:ramseyBaire}, \ref{q:ramseyInd}, \ref{q:padding} in the previous sections.
We also leave open whether all dashed lines in Fig.~\ref{fig:fig} correspond to strict implications. We state a few more questions below.



\begin{question}
 Study the behaviour of $\coh$ over $\Pra$. 
\end{question}
\noindent We note that $\coh$  behaves differently over $\rca$ and over $\rca^*$; only over the former is implied by $\rt 2 2$ (see \cite[Lemma 7.11]{CholakJockuschSlaman} and \cite{fkk:weak-cousins}).

\begin{question}
	Develop the theory of Weihrauch reductions in the primitive recursive setting.\end{question}

 In particular, some version of Weihrauch reduction may help to `separate' the categoricity principles discussed in the present paper for the dense linear order, the random graph, and the atomless Boolean algebra. We note that an `online' version of Weihrauch reduction has recently been suggested in  \cite{punc2}.

\begin{question}\label{q:alg}
Develop the reverse mathematics of countable algebra over $\Pra$.
\end{question}
\noindent For instance, how much of \cite{SolomonThesis} can be carried over $\Pra$?  We have not really looked at natural problems equivalent to $\aca$ over $\Pra$; see Subsection~\ref{subs:aca} for a brief discussion. We believe that systematically investigating into Question~\ref{q:alg} will help to fill this gap.

\

Of course, this list of potential questions is far from being complete.


\end{document}